\documentclass[dvipsnames,onefignum,onetabnum]{siamart171218}
\usepackage[top=1.4in,bottom=1.275in,left=1.4in,right=1.4in]{geometry}

\usepackage[T1]{fontenc}
\usepackage[english]{babel}\usepackage{csquotes}\usepackage[activate={true,nocompatibility},tracking=true]{microtype}\usepackage{hyphenat}\usepackage{siunitx}
\usepackage{placeins}
\usepackage{cite}\usepackage[begintext=``, endtext='']{quoting}
\SetBlockEnvironment{quoting}
\SetBlockThreshold{1}

\usepackage{amsfonts}
\usepackage{amssymb}
\usepackage{amscd}
\usepackage{mathtools}
\usepackage[mathscr]{eucal}
\usepackage{calligra}
\DeclareMathAlphabet{\mathcalligra}{T1}{calligra}{m}{n}
\DeclareFontShape{T1}{calligra}{m}{n}{<->s*[1.1]callig15}{}
\usepackage{bm}
\usepackage{autonum}\usepackage[nice]{nicefrac}

\usepackage{tikz}
\usepackage{pgf}
\usepackage{pgfplots}
\usepackage{pgfplotstable}
\usepackage{pgf-pie}
\usepackage{shellesc}
\pgfplotsset{compat=newest}
\usetikzlibrary{spy}
\usetikzlibrary{shapes}
\usetikzlibrary{backgrounds}
\usetikzlibrary{shadows}
\usetikzlibrary{matrix}
\usetikzlibrary{external}
\usepgfplotslibrary{fillbetween}
\usepackage{comment}
\usepackage{graphicx}
\usepackage[colorinlistoftodos,prependcaption,textsize=small]{todonotes}
\usepackage{cleveref}

\crefname{equation}{}{}
\usepackage{csvsimple}
\usepackage{booktabs}
\usepackage{environ}
\usepackage{adjustbox}
\usepackage{relsize}
\usepackage[font=small,figurewithin=section]{caption}\usepackage[subrefformat=parens,labelformat=simple]{subcaption}
\usepackage{afterpage}

\newtheorem{remark}{Remark}[section]

\newcommand{\changed}[1]{#1}

\let\inf\relax \DeclareMathOperator*\inf{\vphantom{p}inf}
\let\min\relax \DeclareMathOperator*\min{\vphantom{p}min}
\let\max\relax \DeclareMathOperator*\max{\vphantom{p}max}
\let\subset\relax \DeclareMathOperator{\subset}{\subseteq}
 
\let\tilde\widetilde
\let\hat\widehat

\DeclarePairedDelimiter\floor{\lfloor}{\rfloor}

\DeclareMathOperator*{\supp}{supp}

\newcommand{\integral}[4]{\int_{#1}^{#2} #3 \,\mathrm{d}#4}

\newcommand{\R}{\mathbb{R}} \newcommand{\N}{\mathbb{N}}  \ifdef{\C}{\renewcommand{\C}{\mathbb{C}}}{\newcommand{\C}{\mathbb{C}}}  \newcommand{\zI}{\mathrm{i}}     \newcommand{\tr}{\mathrm{tr}} \newcommand{\spann}{\mathrm{span}}        \newcommand{\conj}[1]{\overline{#1}}
\newcommand{\T}{{\raisebox{.1ex}[0ex][0ex]{$\scriptscriptstyle\mathsf{T}$}}}

\newcommand{\bdelta}{{\bm{\delta}}}

\newcommand{\bveps}{{\bm{\varepsilon}}}

\newcommand{\bxi}{\bm{\xi}}

\newcommand{\bsigma}{\bm{\sigma}}

\newcommand{\bvarphi}{\bm{\varphi}}

\newcommand{\iu}{\mathrm{i}}

\newcommand{\bmf}{\bm{f}}
\newcommand{\bmg}{\bm{g}}

\newcommand{\bmi}{\bm{i}}
\newcommand{\bmj}{\bm{j}}

\newcommand{\bmn}{\bm{n}}

\newcommand{\bmt}{\bm{t}}
\newcommand{\bmu}{\bm{u}}
\newcommand{\bmv}{\bm{v}}

\newcommand{\bmx}{\bm{x}}
\newcommand{\bmy}{\bm{y}}

\newcommand{\bmJ}{\bm{J}}

\newcommand{\sff}{\mathsf{f}}

\newcommand{\sfu}{\mathsf{u}}
\newcommand{\sfv}{\mathsf{v}}

\newcommand{\sfB}{\mathsf{B}}

\newcommand{\sfD}{\mathsf{D}}

\newcommand{\sfK}{\mathsf{K}}

\newcommand{\sfM}{\mathsf{M}}

\newcommand{\mcA}{\mathcal{A}}

\newcommand{\mcH}{\mathcal{H}}
\newcommand{\mcI}{\mathcal{I}}

\newcommand{\mcK}{\mathcal{K}}

\newcommand{\mcM}{\mathcal{M}}

\newcommand{\mcO}{\mathcal{O}}

\newcommand{\scD}{\mathscr{D}}

\newcommand{\bbX}{\mathbb{X}}

\newcommand{\bfc}{\mathbf{c}}

\newcommand{\bff}{\mathbf{f}}

\newcommand{\bfu}{\mathbf{u}}
\newcommand{\bfv}{\mathbf{v}}

\newcommand{\bfx}{\mathbf{x}}

\newcommand{\bfF}{\mathbf{F}}

\newcommand{\bfI}{\mathbf{I}}

\newcolumntype{?}{!{\vrule width 1.2pt}}

\makeatletter
\newsavebox{\measure@tikzpicture}
\NewEnviron{scaletikzpicturetowidth}[1]{	\tikzifexternalizingnext{		\def\tikz@width{#1}		\def\tikzscale{1}\begin{lrbox}{\measure@tikzpicture}			\tikzset{external/export next=false,external/optimize=false}			\BODY
		\end{lrbox}		\pgfmathparse{#1/\wd\measure@tikzpicture}		\edef\tikzscale{\pgfmathresult}		\BODY
	}{		\BODY
	}}
\makeatother

\definecolor{color1}{rgb}{0, 0.4470, 0.7410}
\definecolor{color2}{rgb}{0.8500, 0.3250, 0.0980}
\definecolor{color3}{rgb}{0.9290, 0.6940, 0.1250}
\definecolor{color4}{rgb}{0.7060, 0.3840, 0.7650}
\definecolor{color5}{rgb}{0.4660, 0.6740, 0.1880}
\definecolor{color6}{rgb}{0.3010, 0.7450, 0.9330}
\definecolor{color7}{rgb}{0.6350, 0.0780, 0.1840}
\definecolor{color8}{rgb}{0.7410, 0.3800, 0.1840}

\newcommand{\logLogSlopeTriangle}[6]
{
							
	\pgfplotsextra
	{
		\pgfkeysgetvalue{/pgfplots/xmin}{\xmin}
		\pgfkeysgetvalue{/pgfplots/xmax}{\xmax}
		\pgfkeysgetvalue{/pgfplots/ymin}{\ymin}
		\pgfkeysgetvalue{/pgfplots/ymax}{\ymax}
		
				\pgfmathsetmacro{\xArel}{#1-#2}
		\pgfmathsetmacro{\yArel}{#3}
		\pgfmathsetmacro{\xBrel}{#1}
		\pgfmathsetmacro{\yBrel}{\yArel}
		\pgfmathsetmacro{\xCrel}{\xArel}
				
		\pgfmathsetmacro{\lnxB}{\xmin*(1-(#1-#2))+\xmax*(#1-#2)} 		\pgfmathsetmacro{\lnxA}{\xmin*(1-#1)+\xmax*#1} 		\pgfmathsetmacro{\lnyA}{\ymin*(1-#3)+\ymax*#3} 		\pgfmathsetmacro{\lnyC}{\lnyA+#5/#4*(\lnxA-\lnxB)}
		\pgfmathsetmacro{\yCrel}{\lnyC-\ymin)/(\ymax-\ymin)} 		
				\coordinate (A) at (rel axis cs:\xArel,\yArel);
		\coordinate (B) at (rel axis cs:\xBrel,\yBrel);
		\coordinate (C) at (rel axis cs:\xCrel,\yCrel);
		
				\draw[gray!75!black, dashed, line width=1.0pt, #6]   (A)-- node[pos=0.5,anchor=north] {#4}
		(B)-- 
		(C)-- node[pos=0.5,anchor=east] {#5}
		cycle;
	}
}

\newcommand{\UpwardLogLogSlopeTriangle}[6]
{
							
	\pgfplotsextra
	{
		\pgfkeysgetvalue{/pgfplots/xmin}{\xmin}
		\pgfkeysgetvalue{/pgfplots/xmax}{\xmax}
		\pgfkeysgetvalue{/pgfplots/ymin}{\ymin}
		\pgfkeysgetvalue{/pgfplots/ymax}{\ymax}
		
				\pgfmathsetmacro{\xArel}{#1}
		\pgfmathsetmacro{\yArel}{#3}
		\pgfmathsetmacro{\xBrel}{#1-#2}
		\pgfmathsetmacro{\yBrel}{\yArel}
		\pgfmathsetmacro{\xCrel}{\xArel}
				
		\pgfmathsetmacro{\lnxB}{\xmin*(1-(#1-#2))+\xmax*(#1-#2)} 		\pgfmathsetmacro{\lnxA}{\xmin*(1-#1)+\xmax*#1} 		\pgfmathsetmacro{\lnyA}{\ymin*(1-#3)+\ymax*#3} 		\pgfmathsetmacro{\lnyC}{\lnyA+#5/#4*(\lnxA-\lnxB)}
		\pgfmathsetmacro{\yCrel}{\lnyC-\ymin)/(\ymax-\ymin)} 		
				\coordinate (A) at (rel axis cs:\xArel,\yArel);
		\coordinate (B) at (rel axis cs:\xBrel,\yBrel);
		\coordinate (C) at (rel axis cs:\xCrel,\yCrel);
		
				\draw[gray!75!black, dashed, line width=1.0pt, #6]   (A)-- node[pos=0.5,anchor=north] {#4}
		(B)-- 
		(C)-- node[pos=0.5,anchor=west] {#5}
		cycle;
	}
}

\pgfplotsset{
  log x ticks with fixed point/.style={
      xticklabel={
        \pgfkeys{/pgf/fpu=true}
        \pgfmathparse{exp(\tick)}        \pgfmathprintnumber[fixed relative, precision=3]{\pgfmathresult}
        \pgfkeys{/pgf/fpu=false}
      }
  },
  log y ticks with fixed point/.style={
      yticklabel={
        \pgfkeys{/pgf/fpu=true}
        \pgfmathparse{exp(\tick)}        \pgfmathprintnumber[fixed relative, precision=3]{\pgfmathresult}
        \pgfkeys{/pgf/fpu=false}
      }
  }
}

\tikzset{
  ashadow/.style={opacity=.25, shadow xshift=0.07, shadow yshift=-0.07},
}

\definecolor{CustomGreen}{RGB}{65,169,50}

\graphicspath{{./Figures/}}

\allowdisplaybreaks

\title{The surrogate matrix methodology: Accelerating isogeometric analysis of waves}

\author{Daniel~Drzisga\thanks{Lehrstuhl f\"ur Numerische Mathematik, Fakult\"at f\"ur Mathematik (M2), Technische Universit\"at M\"unchen, Garching bei M\"unchen (\email{drzisga@ma.tum.de}, \email{keith@ma.tum.de}, \email{wohlmuth@ma.tum.de})}
\and Brendan~Keith\footnotemark[1]
\and Barbara~Wohlmuth\footnotemark[1]}

\begin{document}

\maketitle

\begin{abstract}
The surrogate matrix methodology delivers low-cost approximations of matrices (i.e., surrogate matrices) which are normally computed in Galerkin methods via element-scale quadrature formulas.
In this paper, the methodology is applied to a number of model problems in wave mechanics treated in the Galerkin isogeometic setting.
Herein, the resulting surrogate methods are shown to significantly reduce the assembly time in high frequency wave propagation problems.
In particular, the assembly time is reduced with negligible loss in solution accuracy.
This paper also extends the scope of previous articles in its series by considering multi-patch discretizations of time-harmonic, transient, and nonlinear PDEs as particular use cases of the methodology.
Our a priori error analysis for the Helmholtz equation demonstrates that the additional consistency error introduced by the presence of surrogate matrices is \emph{independent of the wave number}.
In addition, our floating point analysis establishes that the computational complexity of the methodology compares favorably to other contemporary fast assembly techniques for isogeometric methods.
Our numerical experiments demonstrate clear performance gains for time-harmonic problems, both with and without the presence of perfectly matched layers.
Notable speed-ups are also presented for a transient problem with a compressible neo-Hookean material.
\end{abstract}


\section{Introduction} \label{sec:introduction}
Many techniques to accelerate the formation and assembly of coefficient matrices in Galerkin isogeometric analysis (Galerkin IGA) display their power only as the approximation order $p$ grows.
For example, in $n$-space dimensions, sum factorization reduces the computational complexity of element-wise matrix formation from $\mcO(p^{3n})$, realized with  standard nested quadrature loops, to $\mcO(p^{2n+1})$ \cite{bressan2018sum,ANTOLIN2015817}.
Alternatively, a weighted quadrature rule \cite{Calabro2017,sangalli2018matrix}, which specifies a different quadrature rule for a each individual test function, can reduce the number of quadrature points per element from $\mcO(p^n)$ to simply $\mcO(1)$.
In turn, using such a rule reduces the cost of matrix formation in nested quadrature loops from $\mcO(p^{3n})$ to $\mcO(p^{2n})$.
Combining both acceleration techniques can provide an even greater improvement to performance if element-wise assembly is superseded by a row/column loop.
Indeed, sum factorization and weighted quadrature, when used together with a row/column loop, has a floating point computational complexity of only $\mcO(p^{n+1})$ \cite{hiemstra2019fast}.

The surrogate matrix methodology is another way to reduce the assembly time in Galerkin methods.
However, unlike the strategies mentioned above, its power comes in the small mesh size limit $h\to0$.
In fact, the methodology was first born out of applications in the classical lowest-order ($p=1$) finite element setting \cite{bauer2017two,bauer2018new,bauer2018large,drzisga2018surrogate}; that is, where each of the preceding approaches mentioned above have roughly the same cost.
The surrogate matrix methodology is compatible with row/column loop assembly.
It can also be combined with sum factorization and weighted quadrature, however, that is not a focus of this work.

The fundamental observation behind the surrogate matrix methodology is that if the basis functions used in the trial and test spaces have a specific translational symmetry, then a functional relationship can be drawn between non-zero coefficients in the matrix and points in the reference domain.
This relationship is explicitly established via a finite number (specifically, $\mcO(p^{n})$) of so-called \emph{stencil functions}.
If these stencil functions are smooth, they need only to be sampled at a sparse collection of points (dependent upon $h$) in the reference domain in order to be accurately approximated.
In order to collect these sample values and, thus, define each approximate (i.e., \emph{surrogate}) stencil function, only specific rows/columns in the final matrix need to be computed via quadrature.
Thereafter, once enough samples have been collected, the remaining entries can be filled in by simply evaluating the surrogate stencil functions; an operation with a cost of $\mcO(p^{n}q)$, where $q$ is the (B-spline) degree of the surrogate stencil functions.

Even though $q$ is typically chosen larger than $p$, the floating point complexity remains comparable to that of other fast assembly strategies for Galerkin isogeometric methods\cite{hiemstra2017optimal,bressan2018sum,sangalli2018matrix,mantzaflaris2017low,hofreither2018black,ANTOLIN2015817,mantzaflaris2015integration,Mantzaflaris2014,Calabro2017,Fahrendorf2018,hughes2010efficient,auricchio2012simple,hiemstra2019fast,pan2019fast}.
More importantly, stencil functions provide a flexible platform for efficient processor-memory access which can be used to avoid cache thrashing and significantly reduce the time-to-solution in large scale, matrix-free, massively parallel computations.
This has been carefully demonstrated in previous work \cite{bauer2017two,bauer2018new,bauer2018large,drzisga2018surrogate} and is also not a focus of the present contribution.

Many mathematical aspects of the surrogate matrix methodology were worked out in the isogeometric setting in \cite{drzisga2019igasurrogate}.
In that paper, we showed that the use of surrogate matrices introduces an additional consistency error in the discrete solution which must be controlled by the discretization error of the original method.
The a priori error analysis, based on variational crimes \cite{brenner2007mathematical}, is not much different than that of reduced quadrature rules \cite{hughes2010efficient,SCHILLINGER2014} or of the integration by interpolation and look-up strategy investigated in \cite{Mantzaflaris2014,mantzaflaris2015integration,pan2019fast}.

This paper is part of a series which can be read in any order \cite{drzisga2018surrogate,drzisga2019igasurrogate,drzisga2019igasurrogateimpl}.
In this contribution, we advance the mathematical development of the methodology and focus on a representative set of time-harmonic, transient, and nonlinear wave propagation problems.
In particular, we present an a priori error analysis for the Helmholtz equation which shows that the additional consistency error introduced by the surrogate methodology is \emph{independent of the wave number}.
Although, we focus only on acoustic and hyperelastic waves, we expect that our conclusions will carry over to other material models as well as to other fields of application such as electro- and magnetodynamics and multi-physics wave propagation.
Complementary studies with vibration and plate bending are documented in \cite{drzisga2019igasurrogate}.

As we did in \cite{drzisga2019igasurrogate}, we present evidence of improved performance based only on small-scale feasibility studies with the MATLAB software library GeoPDEs \cite{de2011geopdes,vazquez2016new}.
In particular, we do not consider a parallel implementation or row/column loop assembly.
Although our floating point complexity analysis holds even without row/column loop assembly, both of these aspects are expected to only deliver added benefits to isogeometric surrogate matrix methods.

This paper deals with the Helmholtz equation, linearized elastic waves, and hyperelastic waves modeled with neo-Hookean materials.
In \Cref{sec:preliminaries}, we set the stage by introducing the various equations of interest and mention certain properties of their discretization which are required for the sections which follow.
In \Cref{sec:exploiting_basis_structure}, we describe the essential features of the surrogate matrix methodology in the context of scalar solution variables and its extension to vector-valued solution variables.
In \Cref{sec:contributions}, we present an a priori error analysis for the Helmholtz equation.
In addition, we specify certain aspects of surrogate methods in the presence of perfectly matched layers (PMLs) and in transient and nonlinear problems.
In \Cref{sec:computational_complexity}, we briefly remark on our implementation and establish its (floating point operation) computational complexity.
In \Cref{sec:helmholtz_results}, we provide computational evidence for the performance benefits of the methodology.
Finally, in \Cref{sec:conclusion}, we give some concluding remarks.
\section{Preliminaries} \label{sec:preliminaries}
In this section, we introduce the equations of interest and put forward the main notation of the paper.

\subsection{General equations} \label{sub:general_equations}

Let $\Omega$ be a fixed Lipschitz domain in $\R^n$, where $n=2,3$.
In addition, assume that the boundary of $\Omega$ is partitioned into two relatively open sets $\overline{\Gamma_\mathrm{D} \cup \Gamma_\mathrm{N}} = \partial\Omega$, $\Gamma_\mathrm{D} \cap \Gamma_\mathrm{N} = \emptyset$, and denote its outward unit normal by $\bmn$.
Let $W$ be a differentiable energy density functional, $\rho_0\colon \Omega \to \R_{>0}$ be a mass density function, and $\alpha,\beta \in \C$, $\alpha \neq 0$, two constants.
Consider the following abstract wave propagation problem on $\Omega$, taken over the time interval~$t \in [0,T]$:
\begin{equation}
\label{eq:GeneralSystem}
\begin{alignedat}{3}
\bmu &= \bmu_0
\qquad
&&\text{at } t=0
,
\\
\dot{\bmu} &= \bmv_0
\qquad
&&\text{at } t=0
,
\\
\mathrm{Div}\,\partial_{\bmu} W(\bmu) + \bmf
&= \rho_0\ddot{\bmu}
\qquad
&&\text{in } \Omega\times (0,T]
,
\\
\alpha \bmu + \beta \frac{\partial \bmu}{\partial \bmn} &= \bmg
\qquad
&&\text{on } \Gamma_\mathrm{D}\times (0,T]
,
\\
\partial_{\bmu} W(\bmu) \bmn
&= \bmt
\qquad
&&\text{on } \Gamma_\mathrm{N}\times (0,T]
.
\end{alignedat}
\end{equation}
As usual, the partial derivative in time $t$ is denoted by $\dot{}$ and $\mathrm{Div}$ denotes the (row-wise) divergence operator.

Note that when $W$ is quadratic in $\bmu$, we may also define the time harmonic form of~\cref{eq:GeneralSystem} as follows:
\begin{equation}
\label{eq:TimeHarmonicSystem}
\begin{alignedat}{3}
-\mathrm{Div}\,\partial_{\bmu} W(\bmu) - k^2 \bmu &= \bmf
\qquad
&&\text{in } \Omega
,
\\
\alpha \bmu + \beta \frac{\partial \bmu}{\partial \bmn} &= \bmg
\qquad
&&\text{on } \Gamma_\mathrm{D}
,
\\
\partial_{\bmu} W(\bmu) \bmn
&= \bmt
\qquad
&&\text{on } \Gamma_\mathrm{N}
.
\end{alignedat}
\end{equation}
Here, $k \in \R_{\geq0}$ is the wave number.

\subsection{Examples} \label{sub:examples}

Our focus lies on a number of equations that can be cast in this abstract form of~\cref{eq:GeneralSystem,eq:TimeHarmonicSystem}.
In the case of scalar-valued solution variables, we consider the energy density functional
\begin{align}
\label{eq:LinearWaveEnergy}
W(u) = \frac{c^2}{2} \nabla u^\T \nabla u,
\end{align}
where $c$ is the propagation speed.
Invoking this energy functional in~\cref{eq:GeneralSystem}, one retrieves the acoustic wave equation.
On the other hand, the linearized elastodynamic equations for compressible homogeneous and isotropic materials are obtained by employing the energy density functional
\begin{align}
\label{eq:LinearElastEnergy}
W(\bmu) = \frac{\lambda}{2} \tr (\bveps(\bmu))^2 + \mu \bveps(\bmu) \colon \bveps(\bmu),
\end{align}
where $\lambda$ and $\mu$ are the Lam\'e parameters and $\bveps(\bmu) = \frac{1}{2}\left(\nabla \bmu + \nabla \bmu^\T\right)$.
Lastly, we consider the nonlinear response of a compressible neo-Hookean material by invoking the energy density functional
\begin{align}
\label{eq:NeoHookeanEnergy}
W(\bmu) &= \frac{\lambda}{2} \ln(\det(\bfF(\bmu)))^2 - \mu \ln(\det(\bfF(\bmu))) + \frac{\mu}{2} \big( \tr\left(\bfF(\bmu)^\T \bfF(\bmu)\right) - \tr\left(\bfI\right) \big),
\end{align}
where $\bfF(\bmu) = \bfI + \nabla{\bfu}$.

Note that the time-harmonic form of the acoustic wave equation is equivalent to the Helmholtz equation.
In the next subsection, we give a short summary of several mathematical aspects of the Helmholtz equation which are used in the sequel.

\subsection{Helmholtz equation}
\label{sec:helmholtz_intro}

Let $\alpha = -\zI k$, $\beta = 1$, and $\Gamma_\mathrm{D} = \partial \Omega$.
In this setting, \cref{eq:TimeHarmonicSystem} results in the Helmholtz equation with impedance boundary conditions:
\begin{equation}
\label{eq:Helmholtz}
\begin{alignedat}{3}
-\Delta u - k^2 u &= f
\qquad
&&\text{in } \Omega
,
\\
\frac{\partial u}{\partial \bmn} - \zI k u &= g
\qquad
&&\text{on } \partial \Omega
.
\end{alignedat}
\end{equation}
For the sake of completeness, we now give a brief summary of results from \cite{melenk2011wavenumber, esterhazy2014analysis}.

Begin with a fixed wave number $1 \leq k_0 \leq k \leq k_1$ and define the $k$-dependent norm $\|u\|_{\mcH}^2 = \|\nabla u\|_{L^2(\Omega)}^2 + k^2 \| u\|_{L^2(\Omega)}^2$.
Next, assume that $\Omega$ is convex and that the domain mapping $\bvarphi\colon \hat{\Omega} \rightarrow \Omega$ from the reference domain $\hat{\Omega}$ to the physical domain $\Omega$ is smooth.
Let $\bmJ(\hat{\bmx})$ be the Jacobian of $\bvarphi(\hat{\bmx})$, $\det(\bmJ) > 0$.
We define the sesquilinear forms
\begin{equation}
\label{eqn:weakforms}
\begin{aligned}
a(u,v) &= \integral{\Omega}{}{\nabla u \cdot \nabla \conj{v}}{\bmx} = \integral{\hat{\Omega}}{}{\bmJ^{-\T}\hat{\nabla} \hat{u} \cdot \bmJ^{-\T}\hat{\nabla} \conj{\hat{v}} \det(\bmJ)}{\hat{\bmx}},\\
m(u,v) &= \integral{\Omega}{}{u \conj{v}}{\bmx} = \integral{\hat{\Omega}}{}{\hat{u} \conj{\hat{v}} \det(\bmJ)}{\hat{\bmx}},\\
b(u,v) &= \integral{\partial \Omega}{}{u \conj{v}}{\bmx} = \integral{\partial \hat{\Omega}}{}{\hat{u} \conj{\hat{v}}\, \det(\bmJ)\hspace{0.5pt}\|\bmJ^{-\T}\bmn\|}{\hat{\bmx}},
\end{aligned}
\end{equation}
where $\|\,\cdot\,\|$ denotes the Euclidean norm on $\R^n$.
The functions $\hat{u}$ on the reference domain $\hat{\Omega}$ are defined by the identity $\hat{u} = u \circ \bvarphi$.

Let $V_h \subset H^1(\Omega)$ be a finite-dimensional subspace with basis functions of order $p\in\N$ corresponding to a grid of length $h$.
In particular, assume that $p \geq c_p \log k_1$, for a suitable constant $c_p \in \R_{>0}$, and $h \leq c_h \frac{\log k_1}{k_1}$,
for a properly selected constant $c_h \in \R_{>0}$.
The interested reader is referred to \cite[Assumpt.~4.1]{esterhazy2014analysis} for more details on these assumptions.
According to \cite[Prop.~2.1]{esterhazy2014analysis}, the discrete variational Helmholtz formulation,
\begin{equation}
\left\{
\begin{alignedat}{3}
&\text{Find } u_h \in V_h
\text{ satisfying}
\quad
\\
&a(u_h,v) - k^2 m(u_h,v) - k \zI  b(u_h, v) = \integral{\Omega}{}{f \conj{v}}{x} + \integral{\partial \Omega}{}{g \conj{v}}{x}
\quad
\text{for all }
v\in V_h
\,,
\end{alignedat}
\right.
\label{eq:HelmholtzWeakForm}
\end{equation}
has a unique solution.
Let $u \in H^1(\Omega)$ be the solution of \cref{eq:HelmholtzWeakForm} over the space $H^1(\Omega)$.
By \cite[Prop.~8.1.3]{melenk1995generalized}, $\|u\|_\mcH \leq C(k,\Omega) \big(\|f\|_{H^1(\Omega)^\prime} + \|g\|_{H^{-1/2}(\partial\Omega)}\big)$, where $C(k,\Omega)>0$ is a wave number and domain-dependent constant.

Let the symbol $\lesssim$ denote inequality by a generic positive constant, independent of $k$ and $h$.
According to \cite[Prop.~8.1.4]{melenk1995generalized} and \cite{cummings2006sharp},
\begin{subequations}
\begin{equation}
\label{eq:HnormSolution}
	\|u\|_\mcH
	\lesssim
	\|f\|_{L^2(\Omega)} + \|g\|_{L^2(\partial\Omega)}
	,
\end{equation}
when $\Omega$ is convex.
If, in addition, $g=0$ and $f\in H^1(\Omega)$, then it also holds that
\begin{equation}
\label{eq:HnormSolutiong=0}
	\|u\|_\mcH
	\lesssim
	k^{-1}
	\|f\|_{H^1(\Omega)}
	,
\end{equation}
\end{subequations}
by \cite[Lemma~3.4]{esterhazy2014analysis}.
According to \cite[Cor.~4.6]{esterhazy2014analysis}, the bounds
\begin{subequations}
\begin{align}
\label{eq:Hnormerror}
	\|u - u_h\|_{\mcH}
	&\lesssim
	(hk)^p \big( \|f\|_{H^{p-1}(\Omega)} + \|g\|_{H^{p-1/2}(\partial\Omega)}\big)
	,
\end{align}
hold for convex domains with regularity $p-1$ and $f\in H^{p-1}(\Omega)$, $g \in H^{p-1/2}(\partial\Omega)$.
Furthermore, if one additionally assumes that $g = 0$, then one has the improved estimate
\begin{align}
\label{eq:Hnormerrorg=0}
	\|u - u_h\|_{\mcH}
	&\lesssim
	(hk)^p k^{-1}\|f\|_{H^{p-1}(\Omega)}
	.
\end{align}
\end{subequations}
Again, this follows from \cite[Cor.~4.6]{esterhazy2014analysis}.
Evidently, by the bounds above, uniform stability, i.e.,
\begin{equation}
	\|u_h\|_\mcH \lesssim \|u\|_\mcH
	\,,
\label{eq:UniformStability}
\end{equation}
is obtained for all wave numbers $k>0$.
For more details, as well as numerous generalizations of the bounds above, the interested reader is referred to \cite{esterhazy2014analysis} and the references therein.

Taking $u_h = \sum_{i=1}^N \sfu_i \phi_i$, where $\{\phi_i\}_{i=1}^N$ is a basis for $V_h$, problem \cref{eq:HelmholtzWeakForm} induces the following matrix equation for the coefficient vector $\sfu = [\sfu_1,\sfu_2,\ldots,\sfu_N]^\T$:
\begin{equation}
\label{eq:SurrogateEoM}
\sfK \sfu - k^2 \sfM \sfu - k \zI \sfB \sfu
=
{\sff}
,
\end{equation}
\changed{where $\sfK_{ij} = a(\phi_j,\phi_i)$, $\sfM_{ij} = m(\phi_j,\phi_i)$, $\sfB_{ij} = b(\phi_j,\phi_i)$, and $\sff_i = \integral{\Omega}{}{f \conj{\phi_i}}{x} + \integral{\partial \Omega}{}{g \conj{\phi_i}}{x}$.}
In the next section, we replace~\cref{eq:SurrogateEoM} by a closely related approximation (i.e., surrogate).
\section{Surrogate matrices: Exploiting basis structure} \label{sec:exploiting_basis_structure}

In this section, we illustrate the main ingredients of the surrogate matrix methodology in Galerkin IGA, using the Helmholtz equation as an example.
The goal is to show how to replace~\cref{eq:SurrogateEoM} by some closely related equation
\begin{equation}
\label{eq:SurrogateEoMSurrogate}
    \tilde{\sfK}\tilde{{\sfu}} - k^2 \tilde{\sfM}\tilde{\sfu} - k \zI \sfB \tilde{\sfu}
    =
    {\sff}
    ,
\end{equation}
where $\tilde{\sfM} \approx \sfM$ and $\tilde{\sfK} \approx \sfK$ are faster to assemble, and the two solutions $\sfu \approx \tilde{\sfu}$ are close to identical.
Note that we choose not to replace the matrix $\sfB$.
Its assembly cost is of reduced complexity, since only basis functions at the boundaries need to be considered, and this $\sfB$ is ultimately much sparser than either $\sfK$ or $\sfM$.
The first ingredient of the surrogate matrix construction, is the concept of cardinal B-splines.

\subsection{Cardinal B-splines} \label{sub:cardinal_b_splines_and_nurbs}

Every univariate B-spline basis $\{{b}_k\}_{k=1}^{m}$ is defined by an ordered multiset, or \emph{knot vector}, $\Xi = \{\xi_1,\ldots,\xi_{m+p+1}\}$ \cite{hughes2005isogeometric}.
From now on, we assume that every such $\Xi$ is an \emph{open uniform knot vector} on the unit interval $[0,1]$.
That is, $\xi_1,\ldots,\xi_{p+1}=0$, $\xi_{m+1},\ldots,\xi_{m+p+1}=1$, and $\xi_{k+1} - \xi_k = \frac{1}{m-p}$, otherwise.
For large enough $m$, such knot vectors deliver a vast majority of translation invariant basis functions, such as those depicted in gray in~\cref{fig:xtildek}.
These functions are called \emph{cardinal B-splines} \cite{schoenberg1973cardinal,schoenberg1946contributionsA,schoenberg1946contributionsB}.
We hereby refer to $h = \max_{1\leq k\leq m-1}|\xi_{k+1}-\xi_{k}| = \frac{1}{m-p}$ as the \emph{mesh size} parameter and define $\tilde{x}_k = (k - \frac{p+1}{2})\cdot h$, for each ${k} = p+1,\ldots,m-p$.
See \cref{fig:xtildek} for an illustration of the points $\tilde{x}_k$ and the mesh size $h$.

\begin{figure}
\centering
\includegraphics[height=8em]{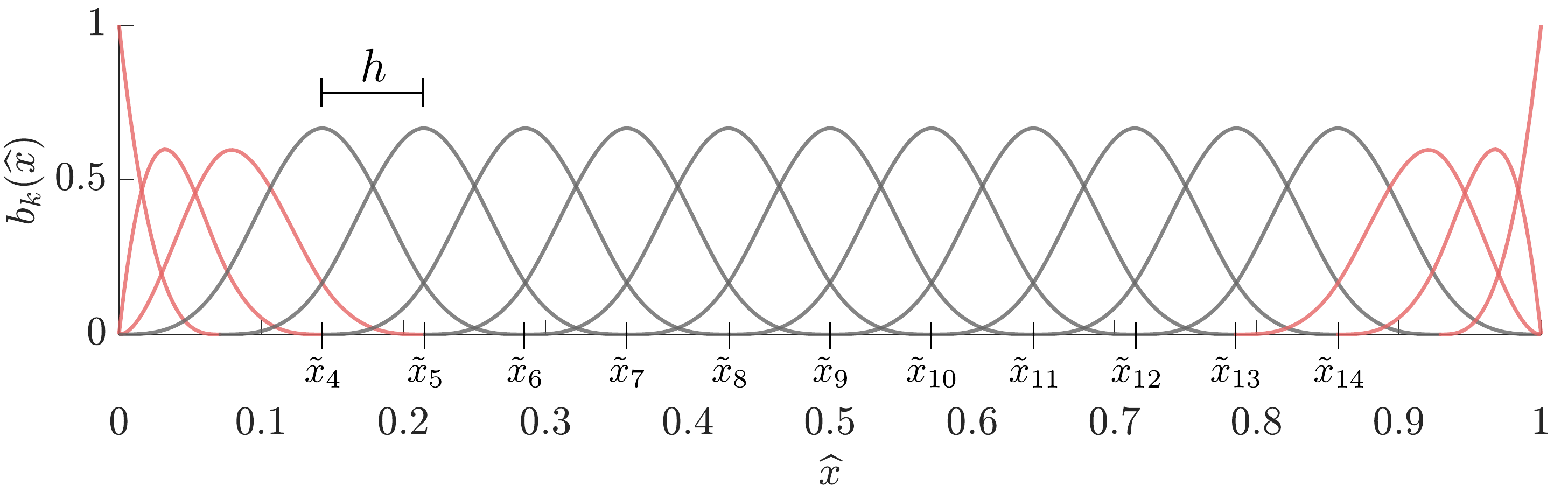}
\qquad
\includegraphics[height=8em]{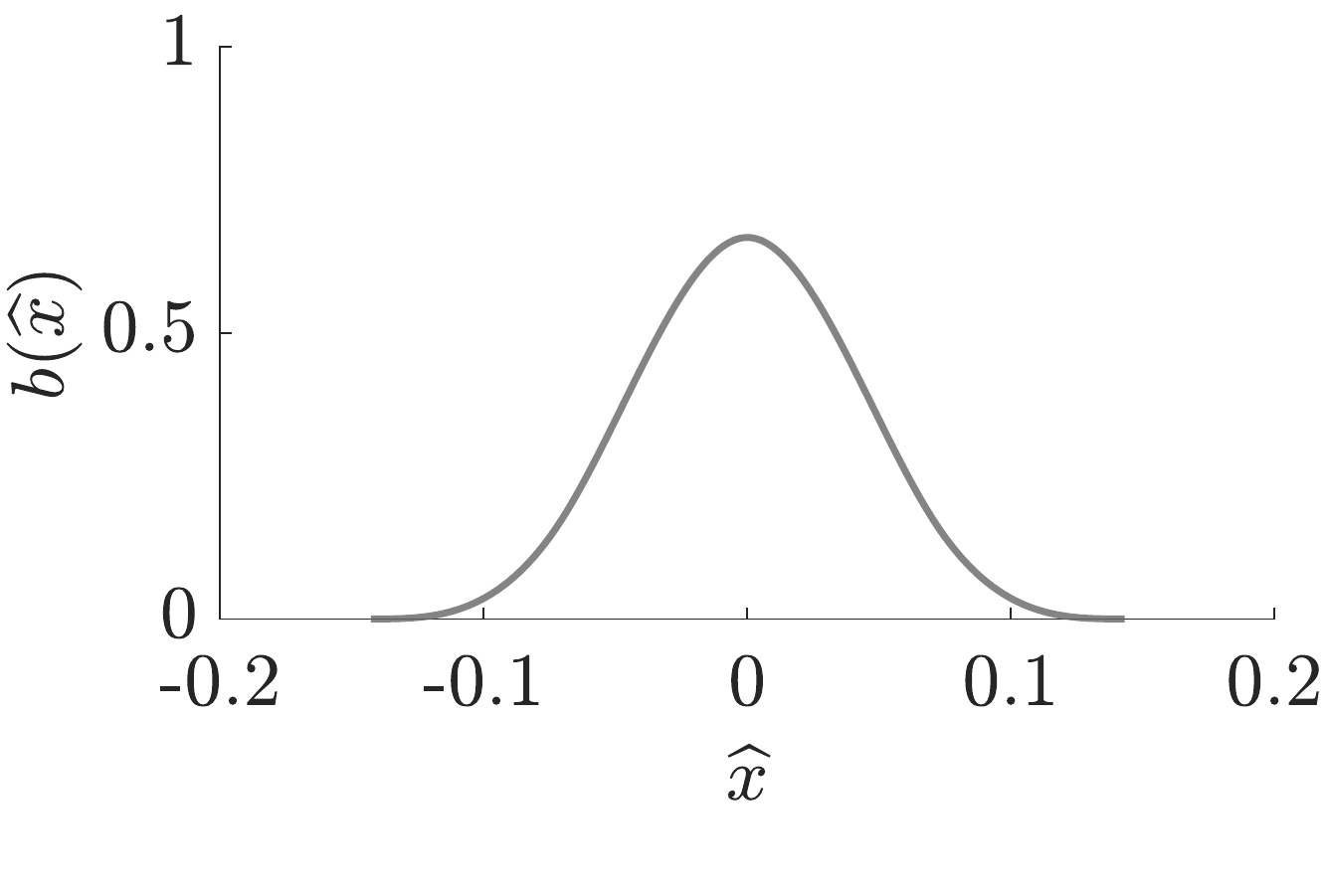}
\caption{\label{fig:xtildek} $1$D B-spline basis functions $\{b_k\}$ with cardinal B-splines in gray. These B-splines come from a third-order $p=3$ uniform knot vector with $m = 17$. Note the points $\tilde{x}_k$ for each ${k} = p+1,\ldots,m-p$ and the mesh size $h$ (left). Each gray basis function is equivalent, up to translation, to the function $b(\hat{x})$ (right).} 
\end{figure}

The open uniform knot vectors described above generate $m-2p$ univariate cardinal B-spline basis functions which can each be expressed as ${b}_{k}(\hat{x}) = {b}(\hat{x} - \tilde{x}_k)$, where ${b}(\hat{x})$ is a function centered at the origin, as depicted on the right of \cref{fig:xtildek}.
Just as in~\cite{drzisga2019igasurrogate}, we do not consider NURBS spaces with different polynomial orders $p_1,\ldots,p_n$ in each Cartesian direction.
Therefore, the tensor product definition of the multivariate B-spline basis, $\{\hat{B}_{\bmi}(\hat{\bmx})\}$, immediately delivers $(m-2p)^n$ \emph{multivariate} cardinal B-splines, $\hat{B}_{\bmi}(\hat{\bmx}) = \hat{B}(\hat{\bmx} - \tilde{\bmx}_{\bmi})$, where $\tilde{\bmx}_{\bmi} = \big(\tilde{x}_{i_1},\ldots, \tilde{x}_{i_n})$ and $\hat{B}({\bmx}) = b({{x}_{1}})\cdots b({{x}_{n}})$.

Here and from now on, we identify every global index $i\in\mcI = \{1,\ldots, N=m^n\}$ with a multi-index $\bmi = (i_1,\ldots,i_n)$, $1\leq i_k\leq m$, through the colexicographical relationship $i = i_1 + (i_2 -1) m + \cdots + (i_n -1) m^{n-1}$.
For future reference, we denote the set of all such $\tilde{\bmx}_{\bmi} = \tilde{\bmx}_{i}$ by $\tilde{\bbX}$.
Notice that the ratio of cardinal B-spline basis functions to total B-spline basis functions, $\big(\frac{m-2p}{m}\big)^n$, quickly tends to unity as $m$ increases.

\subsection{Stencil functions} \label{sub:stencil_functions}

In Galerkin methods, stencil functions provide an explicit functional relationship between entries in the global coefficient matrices, so long as the underlying basis has a particular structure.
Here, we recall a simple definition of stencil functions which comes about by exploiting the structure of cardinal B-splines.
For a generalization to NURBS bases made out of cardinal B-splines, see \cite{drzisga2019igasurrogate}.
Meanwhile, for a description of stencil functions derived from simplicial basis functions, see \cite{drzisga2018surrogate}. 

Begin by recalling the sesquilinear forms $m(\cdot,\cdot)$ and $a(\cdot,\cdot)$ defined in~\cref{eqn:weakforms} and the notation from \Cref{sub:cardinal_b_splines_and_nurbs}.
For $\hat{B} \colon\R^n \to \R$, let us consider the following scalar-valued functions:
\begin{equation}
  \mcM(\tilde{\bmx},\tilde{\bmy})
  =
    m(\hat{B}(\cdot-\tilde{\bmy}),\hat{B}(\cdot-\tilde{\bmx}))
  \,,
  \qquad
  \mcK(\tilde{\bmx},\tilde{\bmy})
  =
    a(\hat{B}(\cdot-\tilde{\bmy}),\hat{B}(\cdot-\tilde{\bmx}))
  \,.
\label{eq:TemporaryStencilFunction}
\end{equation}
It may be readily observed that
\begin{equation}
\label{eq:KeyIdentity}
  \mcM(\tilde{\bmx}_i,\tilde{\bmx}_j)
  =
  [\sfM]_{ij}
  \qquad
  \text{and}
  \qquad
  \mcK(\tilde{\bmx}_i,\tilde{\bmx}_j)
  =
          [\sfK]_{ij}
\end{equation}
for every $i,j\in\mcI$.

Notice that the mass matrix $\sfM$ and the stiffness matrix $\sfK$ are always sparse simply because $m(\hat{B}_{j},\hat{B}_{i})$ and  $a(\hat{B}_{j},\hat{B}_{i})$ both vanish whenever the supports of $\hat{B}_{j}$ and $\hat{B}_{i}$ do not overlap.
For the same reason, both $\mcM(\tilde{\bmx},\tilde{\bmy})$ and $\mcK(\tilde{\bmx},\tilde{\bmy})$ return zero whenever $\|\tilde{\bmy}-\tilde{\bmx}\| \geq 0$ is large enough.

In order to demarcate from these trivial outcomes, we rewrite $\mcM(\tilde{\bmx},\tilde{\bmy})$ and $\mcK(\tilde{\bmx},\tilde{\bmy})$ in terms of $\tilde{\bmx}$ and a translation $\bdelta = \tilde{\bmy}-\tilde{\bmx}$ by defining
\begin{equation}
\label{eq:DefinitionOfStencilFunction}
    \mcM_{\bdelta}(\tilde{\bmx}) = \mcM(\tilde{\bmx} ,\tilde{\bmx} + \bdelta)
    \qquad\text{and}\qquad
    \mcK_{\bdelta}(\tilde{\bmx}) = \mcK(\tilde{\bmx} ,\tilde{\bmx} + \bdelta).
\end{equation}
Taking $\bdelta = \tilde{\bmx}_j-\tilde{\bmx}_i$, we clearly have
\begin{equation}
  [{\sfM}]_{ij}
  =
  \mcM_{\bdelta}(\tilde{\bmx}_i)
  \qquad\text{and}\qquad
  [{\sfK}]_{ij}
  =
  \mcK_{\bdelta}(\tilde{\bmx}_i)
            .
\label{eq:StencilRoughDefinition}
\end{equation}
For this reason, we only need to pay attention to $\bdelta\in\scD$, where
\begin{equation}
    \scD
    =
    \{\tilde{\bmx}_j-\tilde{\bmx}_i : \supp(\hat{B}(\cdot-\tilde{\bmx}_i))\cap\supp(\hat{B}(\cdot-\tilde{\bmx}_j)) \neq \emptyset,~i,j\in\mcI\}
    \,.
    \end{equation}
Using the fact that each point in $\mcI$ is uniformly spaced through the reference domain, one can easily show that $\#\scD = (2p+1)^n$.
Thus, $\scD$ can be seen as a finite index set.
We call each function $\mcM_{\bdelta}$ and $\mcK_{\bdelta}$, enumerated by $\bdelta\in\scD$, a \emph{stencil function}.
The interested reader is referred to \cite{drzisga2018surrogate,drzisga2019igasurrogate} as well as \Cref{fig:OnePatchGeometry,fig:TwoPatchGeometry} for various pictures of stencil functions.

\begin{remark}
In a multi-patch setting, the physical domain $\Omega$ is partitioned into a finite number of disjoint subdomains $\overline{\Omega} = \bigcup_{\ell=1}^{L} \overline{\Omega^{(\ell)}}$.
Each \emph{patch} $\Omega^{(\ell)}$ is identified with the same parametric domain $\hat{\Omega}$ via a unique isogeometric transformation $\bvarphi^{(\ell)}(\hat{\Omega}) = \Omega^{(\ell)}$.
For this reason, extending the definition of the stencil functions to account for multi-patch geometries is straightforward.
Indeed, one only needs to define a separate set of stencil functions $\mcM_\bdelta^{(\ell)}$ and $\mcK_\bdelta^{(\ell)}$, for each patch index $\ell$.
\end{remark}

\subsection{Surrogate stencil functions} \label{sub:surrogate_stencil_functions}

The equations in~\cref{eq:StencilRoughDefinition} are simply functional relationships between the entries of each submatrix $\sfM$ and $\sfK$ and the arguments of $\mcM_{\bdelta}$ and $\mcK_{\bdelta}$, respectively.
In other words, evaluating $\mcM_{\bdelta}$ at any point $\tilde{\bmx}_i\in\tilde{\bbX}$ is operationally equivalent to computing the matrix entry $[{\sfM}]_{ij}$.
Therefore, evaluating $\mcM_{\bdelta}(\tilde{\bmx}_i)$, for each $\bdelta\in\scD$, requires computing precisely all the non-zero coefficients in the $i^\text{th}$ row of $\sfM$.
The same observation clearly also holds when evaluating $\mcK_{\bdelta}$.

If stencil functions are smooth, then they may be accurately approximated by their values at a relatively small number of points $\tilde{\bmx}_{i^\mathrm{s}}\in\hat{\Omega}$.
For our purposes, it is enough to let $\tilde{\bbX}^\mathrm{s}\subset \tilde{\bbX}$ be the set of all such \emph{sample points} $\tilde{\bmx}_{i^\mathrm{s}}$ and let $\mcI^\mathrm{s} \subset \mcI$ be the corresponding set of indices.
This procedure first requires collecting all pairs $(\tilde{\bmx}_{i^\mathrm{s}},[\sfM]_{i^\mathrm{s}j})$, for every $\tilde{\bmx}_j-\tilde{\bmx}_{i^\mathrm{s}} \in\scD$ and $i^\mathrm{s} \in \mcI^\mathrm{s}$, but may be done simply by modifying existing assembly algorithms to compute only the required rows.
Entries in surrogate matrices may then be generated by just evaluating the approximated stencil functions at the remaining points in $\tilde{\bbX}\setminus\tilde{\bbX}^\mathrm{s}$ and filling in the corresponding rows $\mcI\setminus\mcI^\mathrm{s}$.

Define
\begin{equation}
  [\tilde{\sfM}]_{ij}
  =
  \tilde{\mcM}_{\bdelta}(\tilde{\bmx}_i)
  \qquad\text{and}\qquad
  [\tilde{\sfK}]_{ij}
  =
  \tilde{\mcK}_{\bdelta}(\tilde{\bmx}_i)
            \,,
\label{eq:SurrogateRoughDefinition}
\end{equation}
where $\tilde{\mcM}_{\bdelta}$ and $\tilde{\mcK}_{\bdelta}$ are such approximations of ${\mcM}_{\bdelta}$ and $\mcK_{\bdelta}$, respectively.
If these so-called \emph{surrogate stencil functions}, $\tilde{\mcM}_{\bdelta}$ and $\tilde{\mcK}_{\bdelta}$, are expressed in an easily evaluated basis, then $\tilde{\sfM}$ (resp. $\tilde{\sfK}$) can be formed much faster than ${\sfM}$ (resp. ${\sfK}$), simply because of the numerical integration that is avoided.
Moreover, for large enough problems, the coefficients in the surrogate stencil functions should require significantly less storage than the coefficients in the original matrix they are used to approximate.
This makes simply storing the stencil function coefficients and reading out evaluations of $\tilde{\mcM}_{\bdelta}$ or $\tilde{\mcK}_{\bdelta}$ very desirable during each matrix-vector multiply in matrix-free methods, especially when the matrices themselves cannot fit in main memory; see, e.g., \cite{drzisga2018surrogate}.

In this paper, we construct surrogate stencil functions by interpolating ${\mcM}_{\bdelta}$ and ${\mcK}_{\bdelta}$ with a uniform B-spline basis of order $q\geq 0$ with a quasi-uniform knot vector $\tilde{\bm{\Xi}} = \tilde{\Xi}_1\times\cdots\times\tilde{\Xi}_n$, where each knot $\tilde{\bxi}_{\bmi}\in\tilde{\bm{\Xi}}$ is taken from $\tilde{\bbX}^\mathrm{s}$.
Just as the accuracy of the discrete solution $u_h$ is affected by the mesh size parameter $h$, the accuracy of surrogate stencil functions is affected by a sampling length
\begin{equation}
  H
  =
  \max_{|\bmj| = 1,\bmi}\big\{\|\tilde{\bxi}_{\bmi+\bmj}-\tilde{\bxi}_{\bmi}\|_{\infty} \,:\, \tilde{\bxi}_{\bmi},\tilde{\bxi}_{\bmi+\bmj} \in \tilde{\bm{\Xi}}\big\}
  .
\label{eq:Hdefn}
\end{equation}

As a simplifying accommodation, we assume that all stencil functions $\mcM_\bdelta$ and $\mcK_\bdelta$ are defined at every sampling point $\tilde{\bmx}_{i^\mathrm{s}}\in\tilde{\bbX}^\mathrm{s}$.
As argued in \cite[Section~4]{drzisga2019igasurrogate}, this implies that $\tilde{\bbX}^\mathrm{s}\subset\tilde{\Omega}\subsetneq\hat{\Omega}$, where
\begin{equation}
  \tilde{\Omega}
  =
  \bigg[\frac{3p+1}{2(m-p)},1-\frac{3p+1}{2(m-p)}\bigg]^n
  \,.
\end{equation}
For a more complete description of the interpolation strategy used in the coming experiments, as well as the resulting analysis, see \Cref{sub:implementation} and \cite{drzisga2019igasurrogateimpl}.
Note that explicit interpolation is not at all required to generate an accurate surrogate.
Indeed, a different least-squares regression approach, with a high-order polynomial basis, was successfully applied in \cite{drzisga2018surrogate}.
Many approximation alternatives remain to be investigated.

\subsection{Structure-preserving surrogates} \label{sub:surrogate_matrices}

Constructing the complete surrogate matrices $\tilde{\sfM}$ and $\tilde{\sfK}$ out of the corresponding stencil functions requires the consideration of interactions with non-cardinal basis functions.
The simplest way to account for the entries of $\tilde{\sfM}$ and $\tilde{\sfK}$ which are not defined via~\cref{eq:SurrogateRoughDefinition} is to compute them directly with numerical quadrature, as in traditional IGA assembly algorithms.
This is the choice we make here, however, alternative choices are available by using additional stencil functions which exploit symmetries on lower-dimensional planes, as described in \cite[Section~3.6]{drzisga2019igasurrogate}.

Exploiting the symmetry of the mass matrix, we define 
\begin{equation}
  [\tilde{\sfM}]_{ij}
  =
  \begin{cases}
  \tilde{\mcM}_{\bdelta}(\tilde{\bmx}_i) & \text{if } \tilde{\bmx}_i,\tilde{\bmx}_j\in\tilde{\Omega},~i\leq j,\\
  [\tilde{\sfM}]_{ji} & \text{if } \tilde{\bmx}_i,\tilde{\bmx}_j\in\tilde{\Omega},~i>j,\\
  [\sfM]_{ij} & \text{otherwise}.
  \end{cases}
\label{eq:DefinitionOfSurrogateMatrixSymmetric}
\end{equation}
Note that this definition requires interpolating $((2p+1)^n+1)/2$ stencil functions.
Constructing the surrogate stiffness matrix $\tilde{\sfK}$ could follow in the same manner as~\cref{eq:DefinitionOfSurrogateMatrixSymmetric}, however, a surrogate matrix with better approximation properties can be found if we attempt to preserve part of the kernel of the original matrix $\sfK$.

Note that the kernel of $\sfK$ contains all repeated coefficient vectors, $\spann\{\, [1,1,\ldots,1]^\T \}$.
Indeed, because $a(1,w) = a(w,1) = 0$ for all $w\in H^1(\Omega)$ and because B-splines and NURBS have the partition of unity property $\sum_j B_j(\bmx) = 1$, it holds that
\begin{equation}
  0
  =
  a(1,\phi_{i})
  =
  \sum_j a(B_j,\phi_{i})
  =
  \sum_j [\sfK]_{ij}
  \,,
  \quad
  \text{for each }
  i=1,\ldots,N
  .
\end{equation}
Note that this identity may be rewritten
\begin{equation}
\label{eq:RowSumProperty}
  [\sfK]_{ii}
  =
  -\sum_{j\neq i} [\sfK]_{ij}
  \,.
\end{equation}
For this reason, we pose the following symmetric kernel-preserving definition for the surrogate stiffness matrix:
\begin{equation}
  [\tilde{\sfK}]_{ij}
  =
  \begin{cases}
  \tilde{\mcK}_{\bdelta}(\tilde{\bmx}_i) & \text{if } \tilde{\bmx}_i,\tilde{\bmx}_j\in\tilde{\Omega},~i < j,\\
  [\tilde{\sfK}]_{ji} & \text{if } \tilde{\bmx}_i,\tilde{\bmx}_j\in\tilde{\Omega},~i>j,\\
  [\sfK]_{ij} & \text{in all other cases where } i \neq j,\\
  -\sum_{k\neq i} [\tilde{\sfK}]_{ik} & \text{if } i = j.\\
  \end{cases}
\label{eq:DefinitionOfSurrogateMatrixSymmetricKernel}
\end{equation}
Note that definition~\cref{eq:DefinitionOfSurrogateMatrixSymmetricKernel} requires interpolating $((2p+1)^n-1)/2$ stencil functions.
We define the corresponding surrogate sesquilinear forms $\tilde{a}(u,v) = \bar{\sfv}^\T \tilde{\sfK} \sfu$ and $\tilde{m}(u,v) = \bar{\sfv}^\T \tilde{\sfM} \sfu$ where $\sfu$ and $\sfv$ are the coefficient vectors of $u$ and $v$ in the $\{\phi_i\}_{i=1}^N$ basis, respectively.

\begin{remark}
  Definitions~\cref{eq:DefinitionOfSurrogateMatrixSymmetric,eq:DefinitionOfSurrogateMatrixSymmetricKernel} also apply in the obvious way to the multi-patch setting.
  Indeed, they can be simply used to define every patch-wise coefficient matrix $\tilde{\sfM}^{(\ell)}$ and $\tilde{\sfK}^{(\ell)}$ using the corresponding patch-wise stencil functions $\mcM_\bdelta^{(\ell)}$ and $\mcK_\bdelta^{(\ell)}$, respectively.
\end{remark}

\begin{remark}
The definition of the surrogate mass matrix $\tilde{\sfM}$ in \cref{eq:DefinitionOfSurrogateMatrixSymmetric} does not preserve the exact volume of the domain in the sense that $\sum_i \sum_j [\tilde{\sfM}]_{ij} \neq \integral{\Omega}{}{1}{\bmx}$; cf. \cite[Remark~5.1]{drzisga2019igasurrogate}.
However, the volume may still be preserved by changing its construction in the following way.
Let $\sfD$ be a diagonal matrix with $[\sfD]_{ii} = \integral{\Omega}{}{B_i(\bmx)}{\bmx}$ for each $i$.
The true stiffness matrix can be split into $\sfM = \sfD + \sfM_0$ where $[\sfM_0]_{ij} = [\sfM]_{ij}$ for all $j \neq i$ and $[\sfM_0]_{ii} = -\sum_{j\neq i} [\sfM]_{ij}$ for all $i$.
Since $\sfM_0$ has the same structure and zero row-sum property as $\sfK$, the surrogate matrix $\tilde{\sfM}_0$ may be defined as in \cref{eq:DefinitionOfSurrogateMatrixSymmetricKernel}.
Therefore, defining the mass matrix surrogate as $\tilde{\sfM} = \sfD + \tilde{\sfM}_0$ yields the desired property $\sum_i \sum_j [\tilde{\sfM}]_{ij} = \sum_i \sum_j [\sfM]_{ij} = \integral{\Omega}{}{1}{\bmx}$.
This definition only requires the additional assembly of the diagonal matrix $\sfD$ which can be stored in a vector.
The required quadrature formula may also be of lower accuracy, because functions of order $p$ and not $2p$ need to be integrated.
This observation is not further investigated here.
\end{remark}

\begin{remark}
\label{remark:vectorvalued}
The majority of the definitions above generalize immediately to variational problems with vector-valued solutions.
Nevertheless, in the case of linear elasticity, preserving all of the infinitesimal rigid body motions in the definition of a surrogate elasticity stiffness matrix, is more complicated and expensive than preserving the one-dimensional kernel of $\sfK$, as done in~\cref{eq:DefinitionOfSurrogateMatrixSymmetricKernel}.
Our numerical experiments do not show any significant need to incorporate such a feature.
\end{remark}

\subsection{Smooth geometry transformations} \label{sub:smooth_geometry_transformations}
In many studies (see, e.g., \cite{hughes2005isogeometric,hiemstra2019fast}) the geometry transformation $\bvarphi:\hat{\Omega}\to\Omega$ is not globally smooth.
For illustration, consider the singular transformation depicted in~\Cref{fig:OnePatchMap} which has a singularity coming from the lack of smoothness at the top right corner of the physical domain.
This singularity in the geometry transformation implies a singularity in the determinant of the Jacobian $\bmJ$ present in~\cref{eqn:weakforms}.
In turn, a singularity also appears in the corresponding stencil functions; see, e.g., \Cref{fig:PlateWithHole_OnePatchStencilFunction}.

Singular geometry transformations will usually introduce singular features in the stencil functions.
As singular functions are more difficult to approximate accurately, using unnecessary singular geometry maps should be avoided with surrogate matrix methods.
For instance, in the example above, the singularity may be removed simply by using two patches instead of just one.
In~\cref{fig:PlateWithHole_TwoPatches}, an obvious two-patch geometry parameterization is used and it is obvious to infer that the resulting stencil functions will be globally smooth; \Cref{fig:PlateWithHole_TwoPatches} shows one such representative.

\begin{figure}\centering
  \begin{subfigure}[c]{\textwidth}
  \centering
    \includegraphics[width=0.8\textwidth]{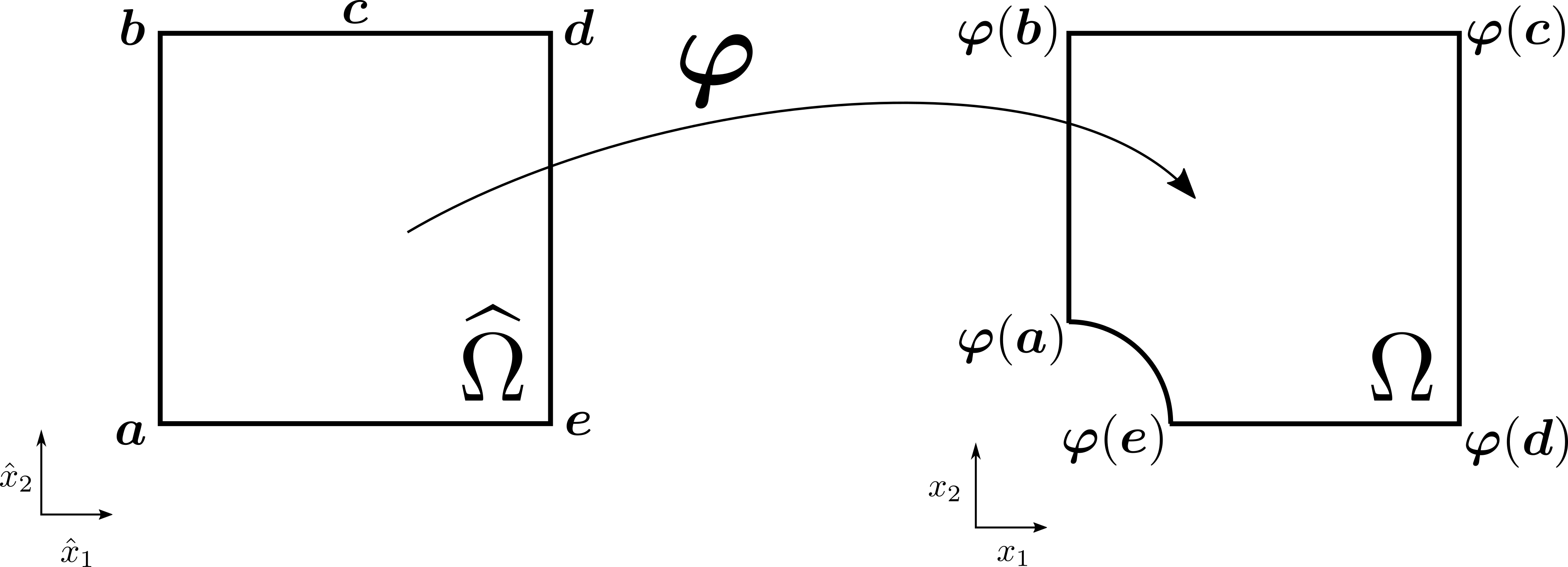}
    \caption{
    \label{fig:OnePatchMap}
    Single geometry map $\bvarphi$ that takes the reference domain $\hat{\Omega}$ to the physical domain $\Omega$.
    }
  \end{subfigure}
    \begin{subfigure}[c]{0.34\textwidth}
  \centering
  \includegraphics[trim=1.8cm 0cm 2cm 0cm,clip=true,height=4.8cm]{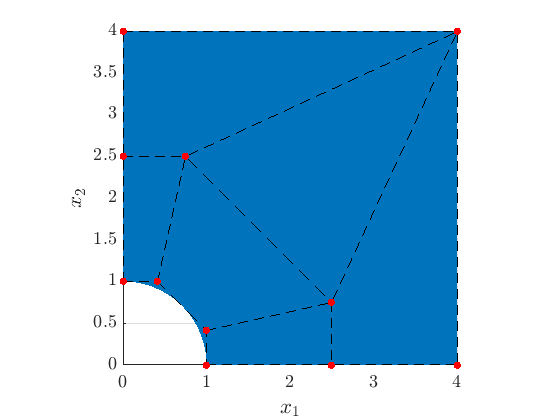}
      \caption{
      \label{fig:PlateWithHole_OnePatch}
      Single patch control net of the physical domain $\Omega$.
      }
  \end{subfigure}
  \quad
  \begin{subfigure}[c]{0.34\textwidth}
  \centering
    \includegraphics[height=4.8cm]{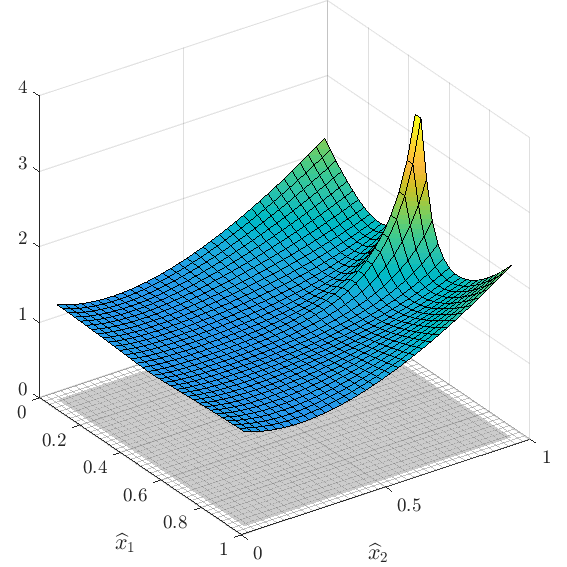}
    \caption{
    \label{fig:PlateWithHole_OnePatchStencilFunction}
    Stencil function of the stiffness matrix for $\delta = (0,0)^\T$ plotted over the reference domain.
    }
  \end{subfigure}
  \caption{
  \label{fig:OnePatchGeometry}
  Geometry map, control net, and stencil function in the case of a single patch geometry.
  }
\end{figure}
\begin{figure}\centering
  \begin{subfigure}[c]{\textwidth}
  \centering
    \includegraphics[width=0.7\textwidth]{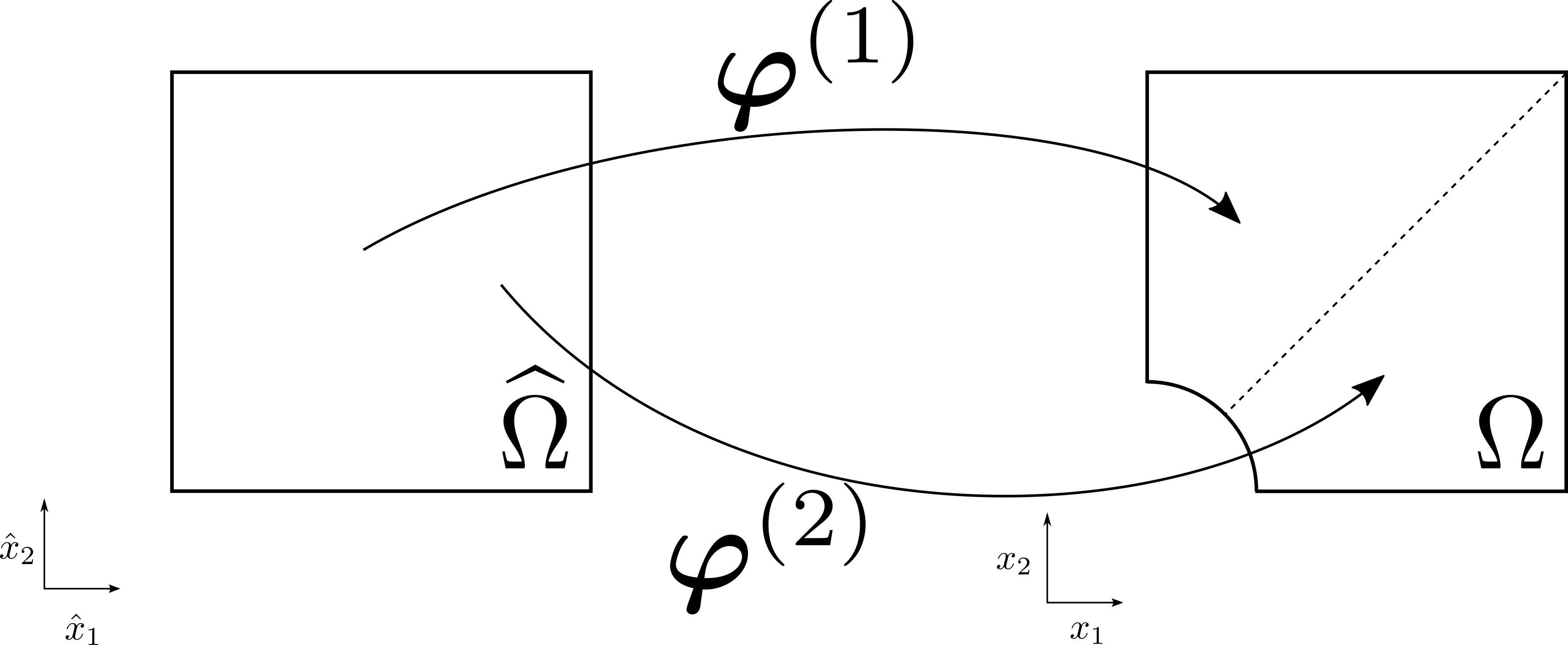}
    \caption{
    \label{fig:TwoPatchMap}
    Two geometry maps $\bvarphi^{(1)}$ and $\bvarphi^{(2)}$ that map the reference domain $\hat{\Omega}$ to the physical domain $\Omega$.
    }
  \end{subfigure}
    \begin{subfigure}[c]{0.32\textwidth}
  \centering
  \includegraphics[trim=1.8cm 0cm 2cm 0cm,clip=true,height=4.5cm]{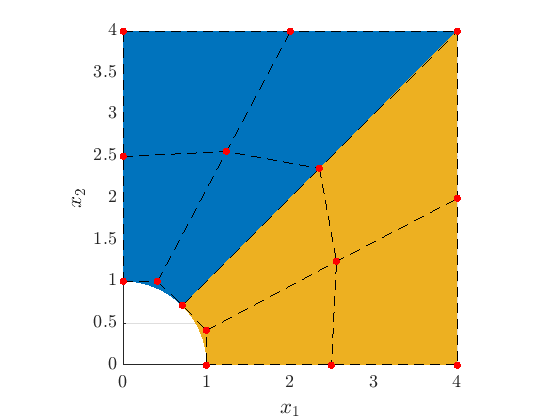}
      \caption{
      \label{fig:PlateWithHole_TwoPatches}
      Two patch control net of the physical domain $\Omega$.
      }
  \end{subfigure}
  \begin{subfigure}[c]{0.32\textwidth}
  \centering
    \includegraphics[height=4.5cm]{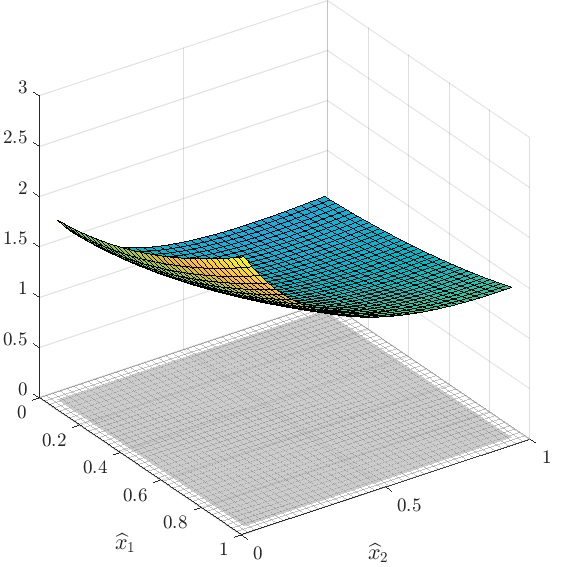}
    \caption{
    \label{fig:PlateWithHole_TwoPatchStencilFunction1}
    Stencil function of the stiffness matrix for $\delta = (0,0)^\T$ and the first patch.
    }
  \end{subfigure}
  \begin{subfigure}[c]{0.32\textwidth}
  \centering
    \includegraphics[height=4.5cm]{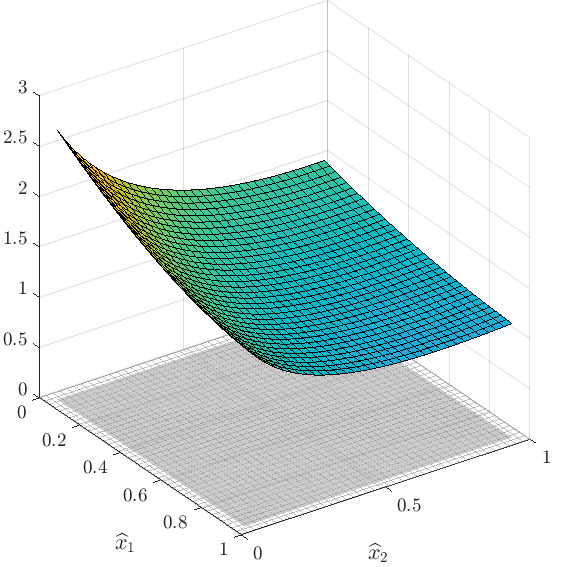}
    \caption{
    \label{fig:PlateWithHole_TwoPatchStencilFunction2}
    Stencil function of the stiffness matrix for $\delta = (0,0)^\T$ and the second patch.
    }
  \end{subfigure}
  \caption{
  \label{fig:TwoPatchGeometry}
  Geometry map, control net, and stencil function in the case of a two patch geometry.
  }
\end{figure}
\section{Surrogate matrices: Theory and applications}
\label{sec:contributions}

In this section, we present an a priori error estimate for the Helmholtz case which shows that the consistency error introduced by the surrogate methodology is \emph{wave number independent}.
Next, we explain how the surrogate methodology is advantageous in wave propagation problems with absorbing boundary conditions.
Finally, we give a short survey of other insights and interpretations which apply for time-dependent and nonlinear problems.
\FloatBarrier
\subsection{A \changed{priori} error estimates for the Helmholtz equation}
The following theorem certifies optimal order convergence of the discretization~\cref{eq:SurrogateEoMSurrogate}, under certain assumptions on $\tilde{a}(u,v) = \bar{\sfv}^\T \tilde{\sfK} \sfu$ and $\tilde{m}(u,v) = \bar{\sfv}^\T \tilde{\sfM} \sfu$.
\changed{
Justification for the stability assumptions made in~\cref{eq:assumption_a,eq:assumption_m} comes from previous work (e.g., \cite{drzisga2019igasurrogate}); for further details, see~\cref{rem:StabilityAssumptions}.
}

\begin{theorem}
\label{thm:helmholtz_surrogate}
Invoke all the hypotheses of \Cref{sec:helmholtz_intro} and define $H$ via~\cref{eq:Hdefn}.
Moreover, let $q_1,q_2 \in \N_0$ and assume that
\begin{subequations}
\label{eq:assumptions}
\begin{align}
\label{eq:assumption_a}
|a(u,v) - \tilde{a}(u,v)| &\lesssim H^{q_1+1} \|\nabla u \|_{L^2(\Omega)} \| \nabla v \|_{L^2(\Omega)},\\
\label{eq:assumption_m}
|m(u,v) - \tilde{m}(u,v)| &\lesssim H^{q_2+1} \| u \|_{L^2(\Omega)} \| v \|_{L^2(\Omega)},
\end{align}
\end{subequations}
for all $u,v\in V_h$.
Then, for all sufficiently small $H$, we have the existence of a unique solution $\tilde{\sfu}$ of \cref{eq:SurrogateEoMSurrogate} and the following a~priori error estimate for $\tilde{u}_h = \sum_i \tilde{\sfu}_i \phi_i$:
\begin{subequations}
\begin{equation}
\label{eqn:apriori_surrogate}
\|u - \tilde{u}_h\|_\mathcal{H}
\lesssim
(hk)^p\big( \|f\|_{H^{p-1}(\Omega)} + \|g\|_{H^{p-1/2}(\partial\Omega)}\big) + H^{q+1}\big(\|f\|_{L^2(\Omega)} + \|g\|_{L^2(\partial\Omega)}\big),
\end{equation}
where $q = \min\{q_1,q_2\}$.
If, in addition, $g=0$, then we have the alternative estimate
\begin{equation}
\label{eqn:apriori_surrogate_g=0}
\|u - \tilde{u}_h\|_\mathcal{H}
\lesssim
k^{-1}\big( (hk)^p  \|f\|_{H^{p-1}(\Omega)}+ H^{q+1}\|f\|_{H^1(\Omega)} \big).
\end{equation}
\end{subequations}
\end{theorem}
\begin{proof}
Fix $k>0$.
Define the sesquilinear form
\begin{equation}
\mcA(u,v) = a(u,v) - k^2 m(u,v) - k \zI b(u,v)
\quad
\text{for all } u, v \in H^1(\Omega)
,
\end{equation}
and denote its discrete stability constant by $\gamma_h$.
Likewise, define the surrogate sesquilinear form
\begin{equation}
\tilde{\mcA}(u,v) = \tilde{a}(u,v) - k^2 \tilde{m}(u,v) - k \zI b(u,v)
\quad
\text{for all } u, v \in V_h
.
\end{equation}
Observe that
\begin{equation}
\label{eq:proofstepA}
	\tilde{\mcA}(\tilde{u}_h,v)
	=
	\integral{\Omega}{}{f \conj{v}}{x} + \integral{\partial \Omega}{}{g \conj{v}}{x} = \mcA(u_h,v)
	\quad
	\text{for all } v \in V_h
	.
\end{equation}

The assumptions on $h$ and $p$ from~\Cref{sec:helmholtz_intro} imply that
\begin{equation}	
 \gamma_h = \inf_{u \in V_h\setminus\{0\}} \sup_{v \in V_h\setminus\{0\}} \frac{|\mcA(u,v)|}{\|u\|_\mcH\|v\|_\mcH} > 0
 .
\end{equation}
This guarantees uniform stability of the original isogeometric discretization for all sufficiently small $h$.
Our first aim is to demonstrate that a similar property holds for the surrogate discretization given by \cref{eq:SurrogateEoMSurrogate}.
Indeed, observe that for any arbitrary $u \in V_h$,
\begin{align}
\sup_{v \in V_h\setminus\{0\}} \frac{|\tilde{\mcA}(u,v)|}{\|v\|_\mathcal{H}}
&\geq
\sup_{v \in V_h\setminus\{0\}} \frac{|\mcA(u,v)|}{\|v\|_\mathcal{H}} - \sup_{v \in V_h\setminus\{0\}} \frac{|\tilde{\mcA}(u,v) - \mcA(u,v)|}{\|v\|_\mathcal{H}}\\
&\geq \big(\gamma_h - \max\{C_1 H^{q_1+1}, C_2 H^{q_2+1}\}\big) \|u\|_\mathcal{H}
.
\end{align}
Therefore,
\begin{equation}
	\tilde{\gamma}_h = \inf_{u \in V_h\setminus\{0\}} \sup_{v \in V_h\setminus\{0\}} \frac{|\mcA(u,v)|}{\|u\|_\mcH\|v\|_\mcH} > \gamma_h - \max\{C_1 H^{q_1+1}, C_2 H^{q_2+1}\} > 0
	,
\end{equation}
for all sufficiently small $H$.

Assuming sufficiently small $h$ and $H$, it now follows that $u_h$ and $\tilde{u}_h$ both exist and are unique.
By the triangle inequality, $\|u - \tilde{u}_h\|_\mathcal{H} \leq \|u - u_h\|_\mathcal{H} + \|u_h - \tilde{u}_h\|_\mathcal{H}$.
Invoking~\cref{eq:proofstepA} and then~\cref{eq:assumptions}, the consistency error term, $\|u_h - \tilde{u}_h\|_\mathcal{H}$, may be bounded from above as follows:
\begin{equation}
\label{eq:ConsistencyError}
\begin{aligned}
\tilde{\gamma}_h\|u_h - \tilde{u}_h\|_\mathcal{H}
&\leq \sup_{v \in V_h\setminus\{0\}}\! \frac{|\tilde{\mcA}(u_h - \tilde{u}_h,v)|}{\|v\|_\mathcal{H}}
=
\sup_{v \in V_h\setminus\{0\}}\! \frac{|\tilde{\mcA}(u_h,v) - \mcA(u_h,v)|}{\|v\|_\mathcal{H}}
\lesssim
H^{q+1} \| u_h \|_\mathcal{H}
.
\end{aligned}
\end{equation}
Inequality~\cref{eqn:apriori_surrogate} now follows from~\cref{eq:HnormSolution,eq:Hnormerror,eq:UniformStability}.
Likewise, if $g=0$, inequality~\cref{eqn:apriori_surrogate_g=0} follows from~\cref{eq:HnormSolutiong=0,eq:Hnormerrorg=0,eq:UniformStability}.
\end{proof}

\begin{remark}
In~\cref{eqn:apriori_surrogate}, it is important to note that the consistency error $\|u_h-\tilde{u}_h\|_\mcH$, stemming from the surrogate matrices, is independent of the wave number.
Meanwhile, in the same setting, the upper bound on the discretization error $\|u-u_h\|_\mcH$ scales like $k^p$.
This makes the surrogate methodology very attractive for large wave number problems, since the total error $\|u-\tilde{u}_h\|_\mcH$ will tend to be dominated by the discretization error $\|u-u_h\|_\mcH$.
\end{remark}

\begin{remark}
	In the special case $g=0$, considered by~\cref{eqn:apriori_surrogate_g=0}, it is well known that the discretization error improves by a factor of $k^{-1}$.
	What is perhaps surprising in the analysis above is that the consistency error of the surrogate method will also improve by the same factor, at least provided that $f\in H^1(\Omega)$.
	This conclusion follows immediately from the improved stability estimate~\cref{eq:HnormSolutiong=0}.
	Thus, in both the $g\neq0$ and $g=0$ settings, the ratio between the discretization error and consistency error remains $\mcO(k^p)$.
\end{remark}

\begin{remark}
				According to results from \cite{melenk2011wavenumber,esterhazy2014analysis}, the error bounds may carry an additional factor of $k^{\frac{5}{2}}$ if $\Omega$ is not convex.
\end{remark}

\begin{remark}\label{rem:StabilityAssumptions}
Theorem~7.2 in \cite{drzisga2019igasurrogate} shows that~\cref{eq:assumption_a} holds for the surrogate stiffness matrix $\tilde{\sfK}$ defined by~\cref{eq:DefinitionOfSurrogateMatrixSymmetricKernel}.
Likewise, assumption~\cref{eq:assumption_m} can be shown to hold for the surrogate mass matrix $\tilde{\sfM}$ defined in~\cref{eq:DefinitionOfSurrogateMatrixSymmetric}; cf. \cite[Section~8.1]{drzisga2019igasurrogate}.
\end{remark}

\subsection{Perfectly matched layer boundary conditions}
Open wave problems posed on unbounded domains are commonly solved on truncated computational domains.
In order to solve such problems accurately, spurious reflections of the outgoing waves, caused by the truncated domain, need to be absorbed.
One approach to simulate this behavior is the perfectly matched layer (PML) absorbing boundary condition introduced in \cite{berenger1994perfectly}.
With this approach, the domain of interest is extended by an artificial absorbing layer made from a special medium.
Many alternative strategies for general curvilinear domains have been proposed since then, but the underlying idea stays the same.

One possibility, which we choose to follow, is the stretching of the real domain into the complex domain.
This is achieved by replacing the physical domain map $\bvarphi(\hat{\bfx})$ by an artificial map
\begin{align}
\tilde{\bvarphi}(\hat{\bfx}) = \bvarphi(\hat{\bfx}) + \zI C\bff(\hat{\bfx}),
\end{align}
where $\bff$ is zero on the domain of interest and is smoothly increasing to unity on the layer's boundary.
The constant $C>0$ is a problem dependent penalty term controlling the strength of the absorption of the layer.
Details on the integral transformations introduced by this complex stretching may be found in \cite{Matuszyk2013}.

The surrogate matrix methodology is very suitable for simulations with PMLs because the discretization error $\|u-u_h\|_\mcH$ is usually bounded from below by a positive constant depending on the size and shape of the absorbing layer.
On the other hand, the consistency error $\|u_h-\tilde{u}_h\|_\mcH$ only measures the distance between the two approximate solutions and, therefore, still tends to zero as the mesh is refined.
Our experience has indicated that the difference between the standard IGA solution and the surrogate IGA solution is rarely distinguishable, even at low wave numbers.
Moreover, as we demonstrate in \Cref{sec:helmholtz_results}, the consistency error, although generally small, tends to be largest in the absorbing layer.
Because only the non-absorbing part of the domain is of interest, these errors in the absorbing layer are of no interest.
We consider PML boundary conditions for a linear elastodynamics problem with periodic pressure loading in \Cref{sub:elastodynamics_with_periodic_pressure_loading}.

\subsection{Discretization in time}
Explicit and implicit time discretization schemes require matrices to propagate solutions forward in time.
Implicit schemes additionally require solving one or more linear systems at each time step.
The surrogate matrix methodology can also be used in such cases for assembling these propagation matrices.
However, if the problem is linear and the iteration matrices do not change over time, and unless a matrix-free approach is considered, each matrix only needs to be assembled once.
In this case, the achievable speed-up depends on the number of time steps.
Indeed, the total relative performance improvement will diminish as the number of time steps grows.

The upshot changes for nonlinear problems where the performance of the surrogate methodology is independent of the number of time steps.
Indeed, the propagation matrices need to be reassembled throughout the simulation because they depend both on the solution at previous time steps and on the iterates of the current time step.
We showcase a time-dependent nonlinear hyperelastic wave problem in \Cref{sec:nonlinear_elasticity_results} and use it to compare performance.

\subsection{Nonlinear problems using Newton's method}
It has already been demonstrated in \cite{drzisga2018surrogate} that the surrogate matrix methodology is suitable for nonlinear problems.
However, in that work, we only considered Picard fixed point iterations.
Although our results were promising, we found that many iterations were required to arrive at the desired solver tolerance.
In this work, we chose to focus on solving nonlinear problems with Newton's method where the Jacobian matrix needs to be reassembled in each iteration.
Now, because the surrogate matrix methodology only yields approximations of matrices, the surrogate Jacobian matrix is simply just an approximation of the true Jacobian matrix.
This means that a Newton method combined with a surrogate method may be more easily interpreted as just a sophisticated quasi-Newton method for the original problem.
One particular consequence is that the consistency error $\|u_h-\tilde{u}_h\|_\mcH$ will vanish with the number of Newton iterations.
Note that in many nonlinear problems, optimizations such as exploiting the symmetry in \cref{eq:DefinitionOfSurrogateMatrixSymmetric} or the row-sum property in \cref{eq:DefinitionOfSurrogateMatrixSymmetricKernel} cannot be used.
\section{Surrogate matrices: Algorithmic considerations} \label{sec:computational_complexity}
In this section, we give a short comment on the differences of the implementation used in this paper when compared to implementations of our previous work in \cite{drzisga2019igasurrogate,drzisga2019igasurrogateimpl}.
We conclude this section with a computational complexity estimate for the asymptotic number of floating point operations (FLOPs) required for the surrogate matrix methodology.

\subsection{Implementation} \label{sub:implementation}
As in \cite{drzisga2019igasurrogate}, all of the experiments documented in this paper were implemented using the GeoPDEs package for Isogeometric Analysis in MATLAB and Octave \cite{de2011geopdes,vazquez2016new}.
Our implementation reused most of the original functionality in GeoPDEs.
A detailed explanation of the modifications and extensions is given in \cite{drzisga2019igasurrogateimpl}, albeit only for the Poisson equation.
Apart from the software implementation aspects, which are more or less unchanged from our previous work, in this paper we utilized a slightly different strategy for selecting the sample points $\tilde{\bmx}_A^\mathrm{s}$ and we used a different B-spline interpolation function.

Let $M>0$ be a fixed integer.
Roughly speaking, when constructing the multivariate B-spline functions, $\tilde{\mcM}_{\bdelta}$ and $\tilde{\mcK}_{\bdelta}$, our goal is to interpolate only about $1/M$ of the points in $\tilde{\bbX}$, in each Cartesian direction.
In \cite{drzisga2019igasurrogate}, this was done by simply taking every $M^{\mathrm{th}}$ point in $\tilde{\bbX}$, in each direction, and adding in every $M^{\mathrm{th}}$ boundary point, if it was skipped over.
In this work, in order to better distribute the sample points, we first find the total number of points $L$ in one Cartesian direction in $\tilde{\bbX}$, and then sample every $(L-1)/\operatorname{ceil}\{\frac{L-1}{M}\}$ point, after rounding to the nearest integer.
By starting at a given corner, this strategy makes sure boundary points are sampled and that all points are roughly evenly spaced; cf.~\Cref{fig:ActiveElements}.
Of course, other sampling point distributions, as for example Chebyshev nodes, may also be used with this approach.
Moreover, in this paper, we used the function \texttt{spapi}, provided by the MATLAB curve fitting toolbox, instead of the standard MATLAB functions \texttt{interp2} and \texttt{interp3} or the SciPy Python function \texttt{RectBivariateSpline}.
\texttt{spapi} allows for more general higher-order B-spline interpolations although it is slightly slower than the other functions.

\begin{figure}\centering
  \begin{subfigure}[c]{0.40\textwidth}
  \centering
    \includegraphics[width=\textwidth]{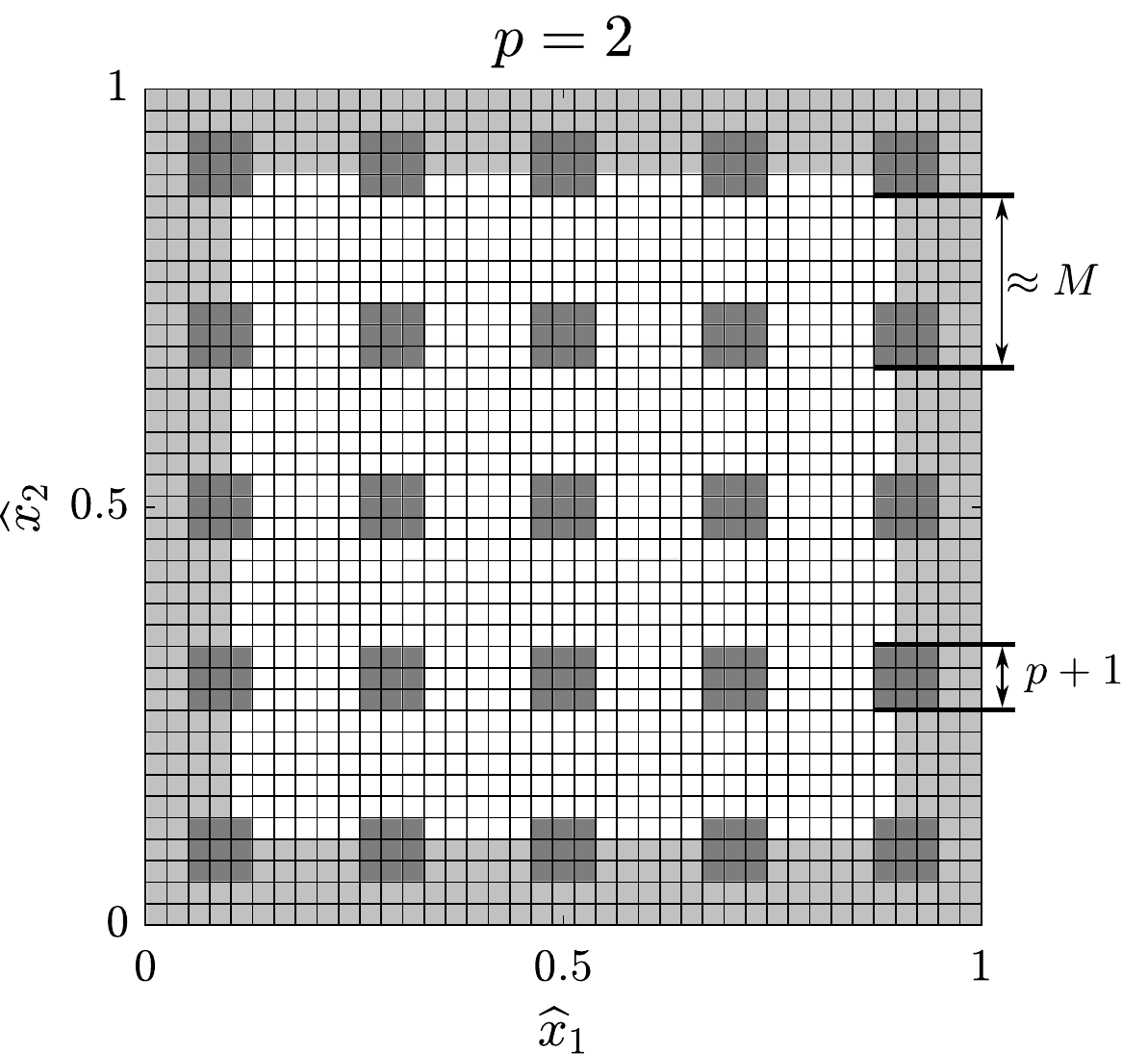}
  \end{subfigure}
  \qquad
  \begin{subfigure}[c]{0.40\textwidth}
  \centering
  \includegraphics[width=\textwidth]{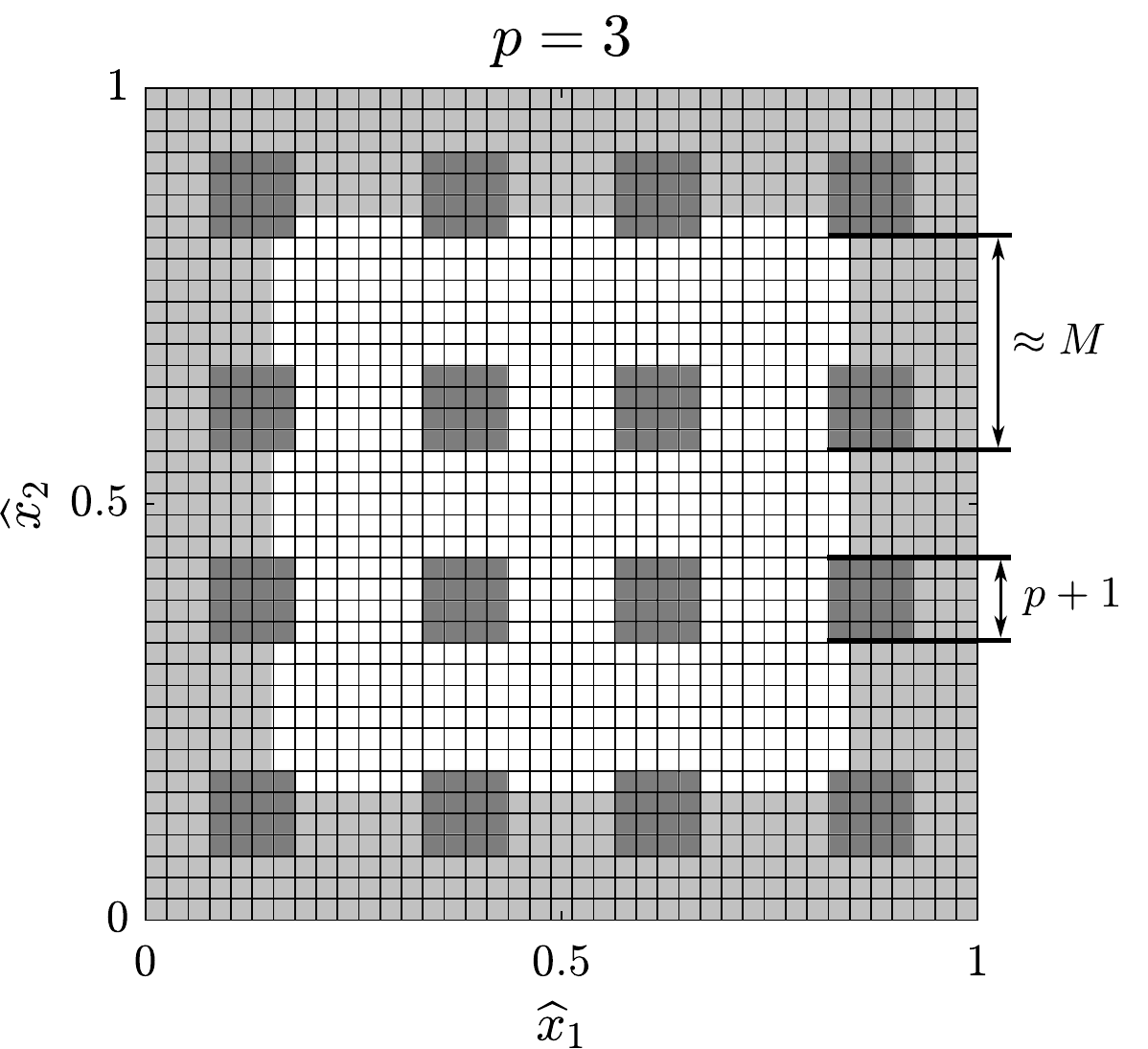}
  \end{subfigure}
  \caption{\label{fig:ActiveElements} The active elements (shown in gray) involved in the surrogate assembly for $M = 10$ with forty knots in each Cartesian direction.
  The light gray elements correspond to the active boundary elements and the dark gray elements correspond to the inner active elements required for the sampling of the stencil functions.}
\end{figure}

Note that our method for evaluating $\tilde{\mcM}_{\bdelta}(\tilde{\bmx}_i)$ and $\tilde{\mcK}_{\bdelta}(\tilde{\bmx}_i)$ is by no means optimal; cf. \cite{drzisga2019igasurrogate,drzisga2019igasurrogateimpl}.
Ideally, we would employ a row- or column-wise loop assembly procedure and only loop over the required rows or columns as it is done in \cite{bauer2017two,bauer2018new,bauer2018large,drzisga2018surrogate}; see \Cref{fig:SparsityPattern}.
Instead, we decided to construct our tests using an established software which employs element-wise loops and standard Gaussian quadrature.
Because the vast majority of IGA software employs element-wise loops, our tests can provide references for many readers to predict how surrogate matrices could accelerate their own codes. 
One drawback of our decision is that in order to evaluate $\tilde{\mcM}_{\bdelta}$ and $\tilde{\mcK}_{\bdelta}$ at any single point $\tilde{\bmx}_i$, we had to perform quadrature on every ``active element'' located in the support of the basis function centered at $\tilde{\bmx}_i$; see \Cref{fig:ActiveElements}.
It is notable that we could easily overcome this wasteful expense to provide significant speed-ups; cf. \Cref{sec:helmholtz_results}.
An explanation for this, using an estimate for the asymptotic number of required floating point operations, is given in \Cref{sub:floating_point_computational_complexity}.

\subsection{Mesh-dependent sampling lengths}
\label{sub:meshdependentsampling}
In this subsection, we recall the concept of mesh-dependent sampling lengths.
Instead of using a fixed $M$ for any mesh, we may allow $M$ to depend on the mesh size $h$.
Let $H$ be the maximum distance in any Cartesian direction, between any two points in $\tilde{\bbX}^\mathrm{s}$, cf.~\cref{eq:Hdefn}.
Recall that $q\geq 0$ is the order of the B-spline interpolation space used in constructing the surrogate stencil functions.
Generally, the error in a surrogate matrix method has the following form:
\begin{equation}
\label{eq:GeneralErrorEstimate}
              \|u-\tilde{u}_h\|
  \leq
  C_a h^{p+a}
  +
  C_b H^{q+b}
  \,,
\end{equation}
where $\|\cdot\|$ is a generic norm and each $a,b\geq 0$, $C_a,C_b > 0$ are real-valued constants, independent of $h$.
The first term on the right hand side of~\cref{eq:GeneralErrorEstimate} controls the discretization error one finds in the standard IGA method; namely $\|u-u_h\| \leq C_a h^{p+a}$.
The second term accounts for the loss of consistency in the surrogate method, $\|u_h-\tilde{u}_h\| \leq C_b H^{q+b}$.
See \cref{eqn:apriori_surrogate,eqn:apriori_surrogate_g=0} for particular examples of such estimates in the case of the Helmholtz equation.

\changed{
A necessary property is that the discretization error dominates the consistency error in the small mesh size limit $h\to 0$.
}
That is, we must design $H^{q+b} = o(h^{p+a})$ because it will cause the surrogate method to have the same asymptotic accuracy as the method it is replacing.
If $H$ is related to $h$ via a constant factor $M>0$, i.e., $H = M \cdot h$, this property is guaranteed, so long as $q + b > p + a$.
However, it is by no means necessary for $H$ and $h$ to be proportional to each other or even for $H = \mcO(h)$.

The best performance is achieved when $H = o(h)$; that is, when $H = M \cdot h$ and $M = H/h \to \infty$ as $h\to 0$.
Provided that $q + b > p + a$, a natural choice which enforces this property is the definition
\begin{equation}
  M(h) = \max\{1, \lfloor C \cdot h^{\epsilon-1 + \frac{p+a}{q+b}} \rfloor\}
  ,
\label{eq:MeshDependentSamplingLength}
\end{equation}
where $\epsilon, C > 0$ are a tunable parameters.
Notice that this definition implies that $M(h)$ will grow more rapidly, in the $h\to 0$ limit, as the interpolation order $q$ is increased.
For further information about mesh-dependent sampling lengths, see \cite[Section~7.3.3]{drzisga2019igasurrogate}.

\subsection{Floating point computational complexity} \label{sub:floating_point_computational_complexity}

The time-to-solution in a simulation depends on the culmination of many factors, not solely the number of FLOPs.
Indeed, good performance usually relies on a good problem-, scale-, and architecture-dependent balance between FLOPs and memory traffic.
In this subsection we present a simple back-of-the-envelope complexity argument, based only on FLOPs, for the surrogate mass matrix $\tilde{\sfM}$ defined in~\cref{eq:DefinitionOfSurrogateMatrixSymmetric}.
A complexity argument for the surrogate stiffness matrix $\tilde{\sfK}$ would be almost identical.
The order estimates presented here should only be understood as crude predictions of the overall performance to be expected in practice.

Begin with an open uniform knot vector $\Xi = \{\xi_1,\ldots,\xi_{m+p+1}\}$ and let it define the multivariate B-spline basis $\{\hat{B}_{i}(\hat{x})\}$, $i = 1,\ldots,N$, described in \Cref{sub:cardinal_b_splines_and_nurbs}.
We assume that this B-spline basis forms the approximation space $V_h$ used in the discretization of both $\sfM$ and $\tilde{\sfM}$.
As mentioned in the introduction, the best complexity of formation and assembly with Galerkin IGA may be as little as $\mcO(r^n p)$ with $r = p$ \cite{hiemstra2019fast}.
Employing standard element-wise Gaussian quadrature, the complexity increases to $\mcO(r^n p)$ for $r = p^2$.
Of course, such an estimate has an implicit dependence on the mesh size $h = \frac{1}{m-p}$.
Accounting for both $h$- and $p$-dependence, the IGA assembly has at least a complexity of $\mcO(N r^{n}p)$, where $N = \mcO(h^{-n})$ is the number of degrees of freedom.

\begin{figure}\centering
    \includegraphics[width=0.5\textwidth]{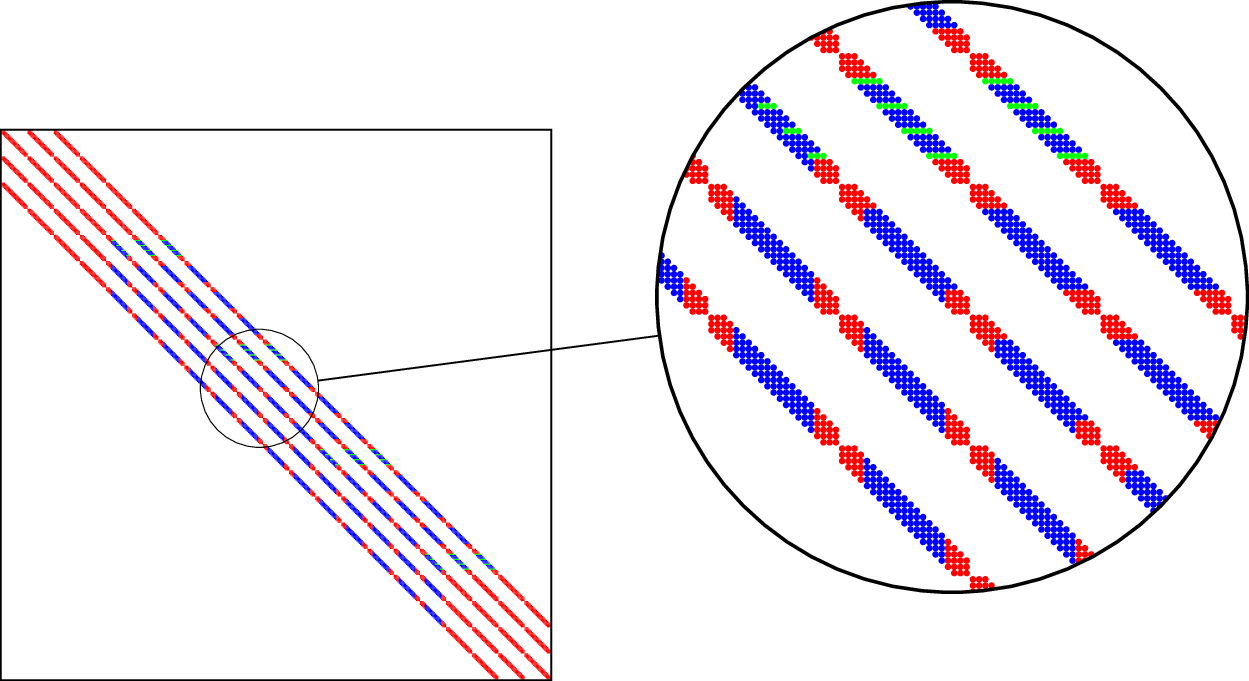}
  \caption{\label{fig:SparsityPattern} Sparsity pattern of the surrogate mass matrix $\tilde{\sfM}$ where $H= 5 h$.
The red and green points indicate the entries of the matrix which are evaluated with Gaussian quadrature.
The blue points indicate the entries which are obtained by evaluating the surrogate stencil functions $\mcM_{\bdelta}$.
The red points correspond to the basis functions near the boundaries and the green entries are used as supporting points for the interpolation.}
\end{figure}

We now argue that if $H = o(h)$, assembling the surrogate mass matrix~\cref{eq:DefinitionOfSurrogateMatrixSymmetric}, with a B-spline interpolation of order $q$, costs $\mcO(h^{-n} p^{n}q)$ FLOPs, with a small leading constant, regardless of the quadrature rule used.
As usual in such analysis, we assume that the univariate B-spline or NURBS basis functions and their gradients are pre-evaluated at the quadrature points and stored in memory.
This assumption, is not a great drawback because the knot vector $\Xi$ is uniform and so the memory footprint of the univariate basis functions evaluated at the quadrature points is small.

In order to estimate the complexity of forming $\tilde{\sfM}$, we must separately account for the cost of computing each of the different nonzero entries in the matrix.
However, we immediately disregard the cost of enforcing symmetry and assume that it only changes the constants found in the final FLOP estimate.

There are three different types of non-zero entries in $\tilde{\sfM}$; see, e.g.,~\cref{fig:SparsityPattern}.
First, there are the entries computed by evaluating the surrogate stencil functions $\tilde{\mcM}_{\bdelta}$ at points in $\tilde{\bbX}$; cf. the blue points in~\cref{fig:SparsityPattern}.
There are $\mcO(p^{n})$ surrogate stencil functions which need to be evaluated at $\#\tilde{\bbX} = (m-2p)^n$ rows. Employing sum-factorization, the final estimate of the surrogate stencil function evaluation is $\mcO(m^n p^{n} q) = \mcO(h^{-n} p^n q)$.
Next, there are each of the non-zero coefficients coming from interaction with basis functions which do not have the cardinal structure; i.e., $\hat{B}_i(\hat{\bmx}) \neq \hat{B}(\bmx-\tilde{\bmx}_i)$, cf. the red points in~\cref{fig:SparsityPattern}.
There are $\mcO(N - \#\tilde{\bbX}) = \mcO(m^n - (m-2p)^n) = \mcO(m^{n-1}) = \mcO(h^{-n+1})$ basis functions without this structure.
In turn, there are $\mcO(h^{-n+1})$ rows/column in $\tilde{\sfK}$ which are filled in using standard IGA assembly procedures, thus providing an optimal complexity of $\mcO(h^{-n+1} r^{n}p)$.
Asymptotically, as $h\to 0$, this contribution is negligible compared to the cost of evaluating the surrogate stencil functions.
However, for large $h$, this term may significantly contribute to the total performance.

Lastly, there are the computations which must be performed in order to sample the stencil functions.
Recall the identification $[\sfM]_{ij} = \mcM_{\bdelta}(\tilde{\bmx}_i)$.
These coefficients are precisely those appearing at the green points in \cref{fig:SparsityPattern}.
Since each point in $\tilde{\bbX}^\mathrm{s}$ is at most a distance $H$ apart, in each Cartesian direction, this leads to at most $\mcO(H^{-n})$ rows, each with a cost of $\mcO(r^{n}p)$.
Written in terms of $H$, the cost of sampling has a total complexity of $\mcO(H^{-n} r^{n}p)$.
\changed{Exploiting the tensor-product structure of the approximation space, the B-spline interpolation itself requires $n$ LU decompositions of sparse univariate collocation matrices which are banded with bandwidth $\mcO(q)$.
Computing the LU decomposition of one such banded matrix without pivoting requires $\mcO(H^{-1} q^2)$ operations \cite{golub2013matrix}.
Applying the forward and backward substitutions to the $\mcO(H^{-n+1})$ right-hand sides requires $\mcO(H^{-n} q)$ operations.
Since the interpolation needs to be done for all $\mcO(p^n)$ stencil functions, the total cost of the interpolation step is $\mcO(H^{-1} p^n q^2 + H^{-n} p^{n} q)$ FLOPs.}

Having accounted for the three different types of non-zero entries, it is now evident that the cost of forming $\tilde{\sfM}$ can be separated into four contributions: evaluating the stencil functions ($\mcO(h^{-n} p^n q)$); numerical integration with the non-cardinal basis functions ($\mcO(h^{-n+1} r^n p)$); sampling the stencil functions ($\mcO(H^{-n} r^n p)$); and performing the interpolation of the stencil functions (\changed{$\mcO(H^{-1} p^n q^2 + H^{-n} p^{n} q)$}).
Since there are always $\mcO(h^{-n})$ rows in the final matrix, the average complexity per row is as follows:
\begin{equation}
\begin{aligned}	
		  \mathrm{avg.~cost}~=~~
  &\frac{\mcO(h^{-n} p^n q) + \mcO(h^{-n+1} r^n p) + \mcO(H^{-n} r^n p) + \changed{\mcO(H^{-1} p^n q^2 + H^{-n} p^{n} q)}}{h^{-n}}
															\,.
\end{aligned}
\label{eq:RelativeComplexity}
\end{equation}

In the small mesh size limit, employing a fixed {sampling parameter} $M>0$ throughout the full sequence of meshes, we see that the complexity in $p$ is still at least $\mcO(r^{n}p)$.
However, in this setting, the constant factor in the $\mcO(r^{n}p)$ term is proportional to $1/M>0$, which may be very small.
In general,
\begin{equation}
  \mathrm{avg.~cost}~=~~\mcO(p^n q) + \mcO(r^n p),
  \qquad
  \text{if } H = \mcO(h).
\end{equation}
In the case of a mesh-dependent sampling length, $\lim_{h\to0} \frac{h}{H} = 0$, so the complexity estimate is improved.
Indeed,
\begin{equation}
  \mathrm{avg.~cost}~=~~\mcO(p^n q),
  \qquad
  \text{if } H = o(h).
\end{equation}

Some remarks about these estimates are now in order.
From the deductions above, it is clear that for any interpolation order $q > p$, forming $\tilde{\sfM}$ will have a poorer floating point complexity than $\mcO(p^{n+1})$, which is achievable with some other methods \cite{hiemstra2019fast}.
Nevertheless, experience has lead the authors to conclude that it tends to actually be very desirable to select a large $q > p$ when forming any surrogate matrix.
Although it is quite clear that a large $q$ positively influences the convergence rate of the consistency error term in~\cref{eq:GeneralErrorEstimate}, we have seen very little change in performance with any $q$ we have studied.
One reason for this may be that the constant factor in the $\mcO(p^n q)$ estimate is extremely small; in particular, much smaller than the constant in front of the $\mcO(r^n p)$ terms attributed to performing quadrature.
We posit that this may be the case because the term derives only from function evaluation which tend to be very cache-aware operations.
Note that the experiments and measurements in \Cref{sec:helmholtz_results} use an implementation with element-loop assembly and Gaussian quadrature; i.e., $r = p^2$.
\section{Numerical examples} \label{sec:helmholtz_results}
In order to show the applicability and efficiency of the presented methods, we performed a set of numerical experiments which are documented here.
In \Cref{subsec:helmholtz_results}, we consider the Helmholtz equation with various boundary conditions\changed{, and in \Cref{sub:varyingwavenumber} the same problem with a non-constant wave number}.
In \Cref{sub:elastodynamics_with_periodic_pressure_loading}, we consider a time harmonic problem involving linear elasticity.
Finally, in \Cref{sec:nonlinear_elasticity_results}, we consider a nonlinear, transient, hyperelastic wave propagation example.

All run-time measurements were obtained on a machine equipped with two Intel\textsuperscript{\textregistered} Xeon\textsuperscript{\textregistered} Gold 6136 processors with a nominal base frequency of 3.0 GHz.
Each processor has 12 physical cores which results in a total of 24 physical cores.
The total available memory of \SI{251}{\giga\byte} is split into two NUMA domains; one for each socket.

\subsection{Helmholtz equation}
\label{subsec:helmholtz_results}
In this subsection, we investigate the surrogate matrix methodology in case of the Helmholtz example \cref{eq:Helmholtz} and verify the theoretical results stated in \Cref{thm:helmholtz_surrogate}.
We investigate the problem on three representative domains.
Namely, the convex domain $\Omega_1$ depicted in \cref{fig:DomainTeaCup} and two non-convex domains; the quarter annulus with bumps $\Omega_2$ depicted in \cref{fig:DomainNonconvex} and the part of a spherical shell $\Omega_3$ shown in \cref{fig:Domain3D}.

In the first set of experiments, we fix the trial space $V_h$ and the surrogate matrix parameters and vary the wave number.
This will indicate dependence of the various errors on the wave number $k$.
For $\Omega_1$ and $\Omega_2$ we fix $m = 640$ and $M = 5$, and for $\Omega_3$, $m = 100$ and $M = 17$.
In each setting, we set $p = 2$ and $q = 5$.
Let $\mcH_a^{(1)}$ be a Hankel function of the first kind and $r = \|\bfx\|$.
As analytical solutions, we choose
\begin{align}
u(r) = \frac{\zI}{4} \mcH_0^{(1)}(kr) \text{ in 2D\qquad and \qquad} u(r) = \frac{\zI}{4r} \mathrm{e}^{\zI k r} \text{ in 3D.}
\label{eq:ExactSolutionsHelmholtz}
\end{align}
These choices yield $f = 0$ so long as the origin is not included in the domain.
The Robin-type boundary term $g$ is computed by using the analytical solution $u$.

The relative errors in the $\mathcal{H}$-norm are presented in \cref{fig:HelmholtzWavenumber} for the first two domains and on the left-hand side of \cref{fig:HelmholtzWavenumber3D} for $\Omega_3$.
Additionally, we plot the real part of the surrogate solution for $k = 30$ in the center of \cref{fig:HelmholtzWavenumber3D} and show the assembly time comparison between the standard and surrogate method on the right-hand side of \cref{fig:HelmholtzWavenumber3D}.
Already, with a fixed $M = 17$, a speed-up of about $251$\% can be observed.
The discretization error in all cases grows like $k^p$, as predicted in \Cref{thm:helmholtz_surrogate}.
Moreover, the relative consistency error in the $\mathcal{H}$-norm is almost independent of $k$.
This agrees well enough with our predictions since the assumptions made in \Cref{sec:helmholtz_intro} may not hold for the very highest wave numbers we considered.
\begin{figure}
\centering
\begin{subfigure}[b]{0.32\linewidth}
\includegraphics[width=\linewidth]{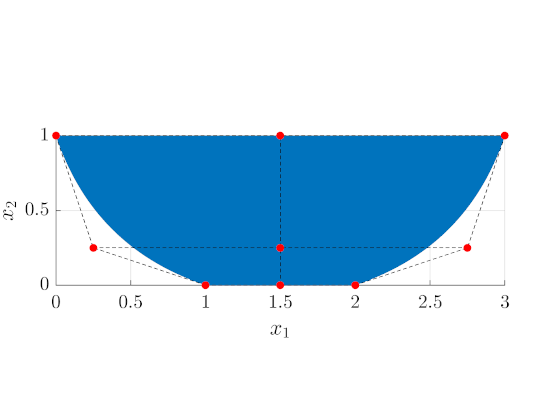}
\caption{\label{fig:DomainTeaCup}Convex domain $\Omega_1$}
\end{subfigure}
\begin{subfigure}[b]{0.32\linewidth}
\includegraphics[width=\linewidth]{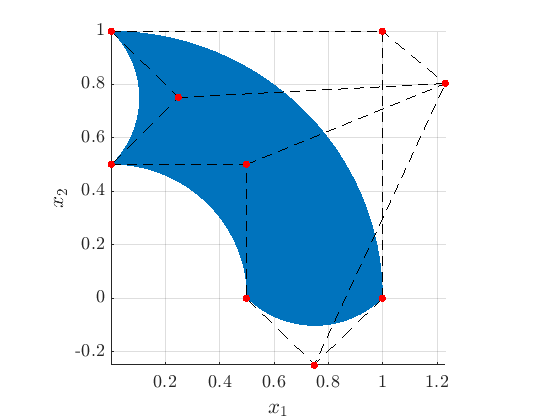}
\caption{\label{fig:DomainNonconvex}Non-convex domain $\Omega_2$}
\end{subfigure}
\begin{subfigure}[b]{0.32\linewidth}
\includegraphics[width=\linewidth]{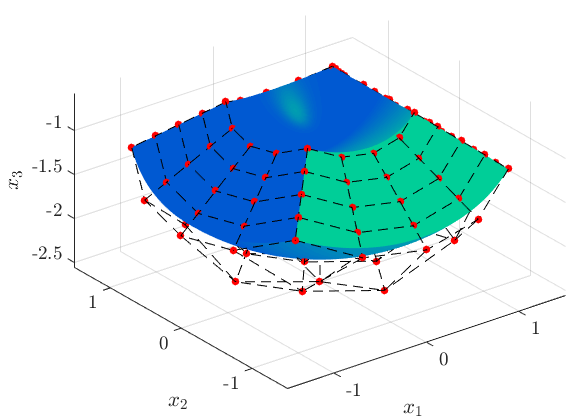}
\caption{\label{fig:Domain3D}Non-convex domain $\Omega_3$}
\end{subfigure}
\caption{\label{fig:DomainsHelmholtz} Domains considered for the Helmholtz problem.}
\end{figure}
\begin{figure}\centering
\hspace*{2em}
\begin{minipage}{0.32\textwidth}
\begin{scaletikzpicturetowidth}{\textwidth}
\begin{tikzpicture}[scale=\tikzscale,font=\large]
\begin{loglogaxis}[
xlabel={Wave number $k$},
ylabel={Relative error},
xmajorgrids,
ymajorgrids,
title={Convex domain $\Omega_1$},
legend style={fill=white, fill opacity=0.6, draw opacity=1, text opacity=1},
legend pos=north west,
]
\addplot[black, mark=*, very thick, mark options={scale=1.5}, fill opacity=0.6] table [x index = {0}, y index={1}, col sep=comma] {./Results/Helmholtz_paper_convex/Helmholtz_errors.csv};
\addlegendentry{$\|u - u_h\|_\mathcal{H} / \|u\|_\mathcal{H}$};
\addplot[color2, mark=diamond*, very thick] table [x index = {0}, y index={1}, col sep=comma] {./Results/Helmholtz_paper_convex/Helmholtz_errors.csv};
\addlegendentry{$\|u - \tilde{u}_h\|_\mathcal{H} / \|u\|_\mathcal{H}$};
\addplot[color1, mark=diamond*, very thick] table [x index = {0}, y index={3}, col sep=comma] {./Results/Helmholtz_paper_convex/Helmholtz_errors.csv};
\addlegendentry{$\|u_h - \tilde{u}_h\|_\mathcal{H} / \|u\|_\mathcal{H}$};
\UpwardLogLogSlopeTriangle{0.75}{0.2}{0.63}{1}{2}{};
\legend{}; \end{loglogaxis}
\end{tikzpicture}
\end{scaletikzpicturetowidth}
\end{minipage}
\qquad
\begin{minipage}{0.47\textwidth}
\begin{scaletikzpicturetowidth}{\textwidth}
\begin{tikzpicture}[scale=\tikzscale,font=\large]
\begin{loglogaxis}[
xlabel={Wave number $k$},
ylabel={Relative error},
xmajorgrids,
ymajorgrids,
title={Non-convex domain $\Omega_2$},
legend style={fill=white, fill opacity=0.6, draw opacity=1, text opacity=1},
legend pos=outer north east,
mark repeat=1,
]
\addplot[black, mark=*, very thick, mark options={scale=1.5}, fill opacity=0.6] table [x index = {0}, y index={1}, col sep=comma] {./Results/Helmholtz_paper_nonconvex/Helmholtz_errors.csv};
\addlegendentry{$\|u - u_h\|_\mathcal{H} / \|u\|_\mathcal{H}$};
\addplot[color2, mark=diamond*, very thick] table [x index = {0}, y index={1}, col sep=comma] {./Results/Helmholtz_paper_nonconvex/Helmholtz_errors.csv};
\addlegendentry{$\|u - \tilde{u}_h\|_\mathcal{H} / \|u\|_\mathcal{H}$};
\addplot[color1, mark=diamond*, very thick] table [x index = {0}, y index={3}, col sep=comma] {./Results/Helmholtz_paper_nonconvex/Helmholtz_errors.csv};
\addlegendentry{$\|u_h - \tilde{u}_h\|_\mathcal{H} / \|u\|_\mathcal{H}$};
\UpwardLogLogSlopeTriangle{0.7}{0.2}{0.6}{1}{2}{};
\end{loglogaxis}
\end{tikzpicture}
\end{scaletikzpicturetowidth}
\end{minipage}
\caption{\label{fig:HelmholtzWavenumber} Demonstration of $k$-dependency on the various errors for the Helmholtz problem on $\Omega_1$ and $\Omega_2$, respectively.}
\end{figure}
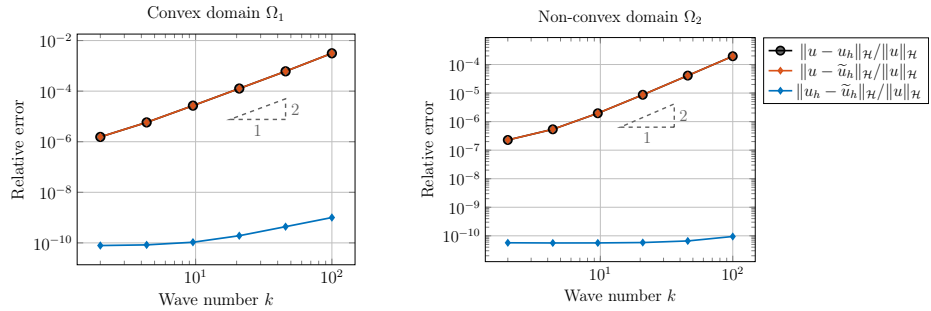
\begin{figure}\centering
\begin{minipage}{0.31\textwidth}
\begin{scaletikzpicturetowidth}{\textwidth}
\begin{tikzpicture}[scale=\tikzscale,font=\large]
\begin{loglogaxis}[
xlabel={Wave number $k$},
ylabel={Relative error},
xmajorgrids,
ymajorgrids,
title={Non-convex domain $\Omega_3$},
legend style={fill=white, fill opacity=0.6, draw opacity=1, text opacity=1},
legend pos=north west,
mark repeat=1,
ymax=5e-1
]
\addplot[black, mark=*, very thick, mark options={scale=1.5}, fill opacity=0.6] table [x index = {0}, y index={1}, col sep=comma] {./Results/Helmholtz_paper_3d/Helmholtz_errors.csv};
\addlegendentry{$\|u - u_h\|_\mathcal{H} / \|u\|_\mathcal{H}$};
\addplot[color2, mark=diamond*, very thick] table [x index = {0}, y index={1}, col sep=comma] {./Results/Helmholtz_paper_3d/Helmholtz_errors.csv};
\addlegendentry{$\|u_h - \tilde{u}_h\|_\mathcal{H} / \|u\|_\mathcal{H}$};
\addplot[color1, mark=diamond*, very thick] table [x index = {0}, y index={3}, col sep=comma] {./Results/Helmholtz_paper_3d/Helmholtz_errors.csv};
\addlegendentry{$\|u_h - \tilde{u}_h\|_\mathcal{H} / \|u\|_\mathcal{H}$};
\UpwardLogLogSlopeTriangle{0.7}{0.2}{0.45}{1}{2}{};
\end{loglogaxis}
\end{tikzpicture}
\end{scaletikzpicturetowidth}
\end{minipage}
\hfill
\begin{minipage}{0.30\textwidth}
\begin{center}
\includegraphics[width=0.9\textwidth]{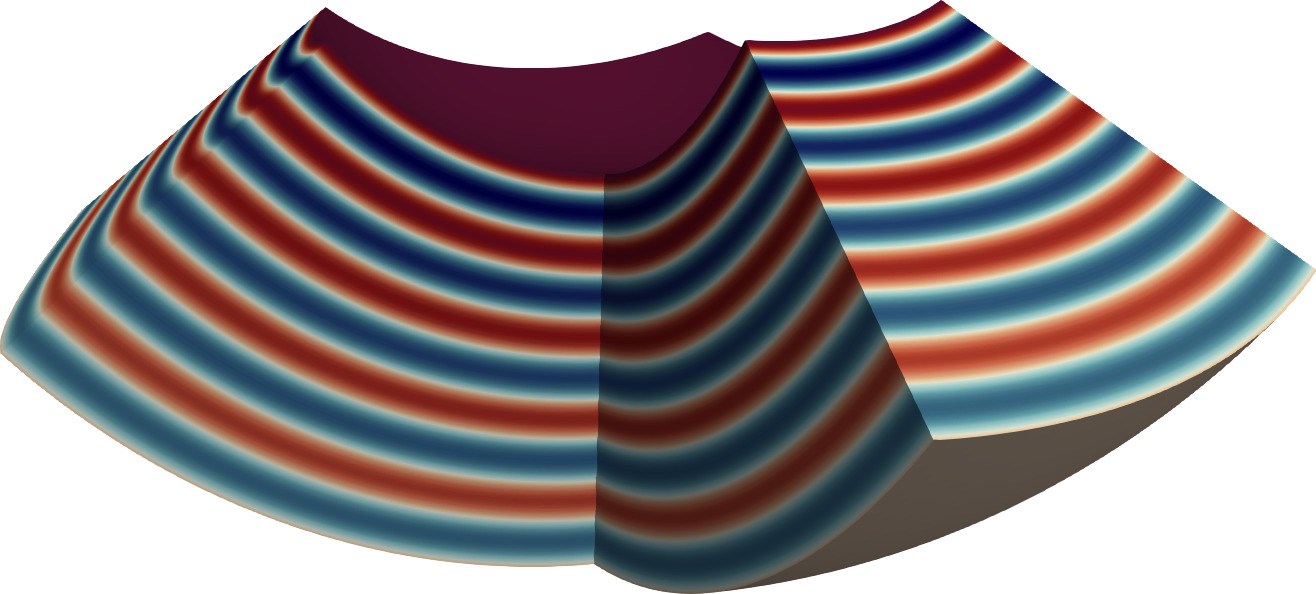}\\[0.5em]
\begin{scaletikzpicturetowidth}{0.7\textwidth}
\begin{tikzpicture}[scale=\tikzscale,font=\large]
\pgfmathsetlengthmacro\MajorTickLength{
  \pgfkeysvalueof{/pgfplots/major tick length} * 0.5
}
\begin{axis}[
title={\Large $\mathrm{Re}\{\tilde{u}_h\}$},
xmin=-2.4e-1, xmax=2.30e-1,
ymin=0, ymax=0.02,
axis on top,
scaled x ticks=false,
scaled y ticks=false,
ytick=\empty,
yticklabels=\empty,
yticklabel pos=right,
x tick label style={
  /pgf/number format/.cd,
            sci zerofill,
            precision=0,
  /tikz/.cd  
},
extra x tick style={
    font=\large,
    tick style=transparent,     yticklabel pos=left,
    x tick label style={
        /pgf/number format/.cd,
            std,
            precision=3,
      /tikz/.cd
    }
},
width=7cm,
height=1.82cm,
major tick length=\MajorTickLength,
max space between ticks=1000pt,
try min ticks=4,
]
\addplot graphics [
includegraphics cmd=\pgfimage,
xmin=\pgfkeysvalueof{/pgfplots/xmin}, 
xmax=\pgfkeysvalueof{/pgfplots/xmax}, 
ymin=\pgfkeysvalueof{/pgfplots/ymin}, 
ymax=\pgfkeysvalueof{/pgfplots/ymax}
] {./Figures/cool_to_warm_extended_rot.png};
\end{axis}
\end{tikzpicture}
\end{scaletikzpicturetowidth}
\end{center}
\end{minipage}
\hfill
\begin{minipage}{0.31\textwidth}
\begin{tikzpicture}
\begin{axis}[font=\footnotesize,
ymin=0,
height=4.25cm,
x=3.6cm,
enlarge x limits={abs=0.45cm},
bar width=0.3cm,
ybar,
xmajorticks=false,
nodes near coords,
xmin=0.0,xmax=0.3,
ylabel={Assembly time [s]},
ymax=6800,
legend style={at={(0.5,-0.15)}, anchor=north,legend columns=-1},]
\addplot+[black] coordinates {(0.1, 5636)};
\addplot+[color2] coordinates {(0.2, 1607)};
\legend{ref,surr}
\end{axis}
\end{tikzpicture}
\end{minipage}
\caption{\label{fig:HelmholtzWavenumber3D} Demonstration of $k$-dependency on the various errors for the Helmholtz problem on $\Omega_3$ (left). Plot of the real part of the surrogate solution for $k = 30$ (center). Assembly time comparison between reference (ref) and surrogate (surr) method (right).}
\end{figure}

For our second set of experiments, we consider the non-convex domain $\Omega_2$ and the same 2D analytical solution defined in~\cref{eq:ExactSolutionsHelmholtz}.
Here, we vary $h$ but fix $q = 5$, use the mesh-dependent sampling parameter $M = \max\Big\{1, \lfloor 0.5 \cdot m^{1-\frac{p+1}{q+1}}\rfloor\Big\}$, and consider each $k \in \{8,16,32,64,128\}$.
For this scenario, we plot relative total errors, relative consistency errors, and speed-ups versus $\frac{m}{k} \propto \frac{1}{kh}$.

The relative errors in the $\mathcal{H}$-norm for the selected wave numbers $k \in \{8,128\}$ can be observed in the plot on the left-hand side of \Cref{fig:HelmholtzError}.
Here, both $\frac{\|u-u_h\|_\mcH}{\|u\|_\mcH}$ and $\frac{\|u-\tilde{u}_h\|_\mcH}{\|u\|_\mcH}$ are presented.
For a common wave number $k$, the two relative error curves lie almost perfectly on top of each other and clearly demonstrate the estimated optimal order of convergence, $\mathcal{O}((\frac{m}{k})^{-2})$.
In the center plot of \Cref{fig:HelmholtzError}, we present the relative consistency errors for each $k \in \{8,16,32,64,128\}$.
From this plot, it is both obvious that the consistency errors are much smaller than the corresponding discretization errors and that they do not have any notable dependence on the ratio $m/k$.
On the right-hand side of \Cref{fig:HelmholtzError}, the speed-ups of the assembly time for those wave numbers are presented.
The largest speed-up of $3178$\% may be observed for $k = 128$ on the finest mesh.
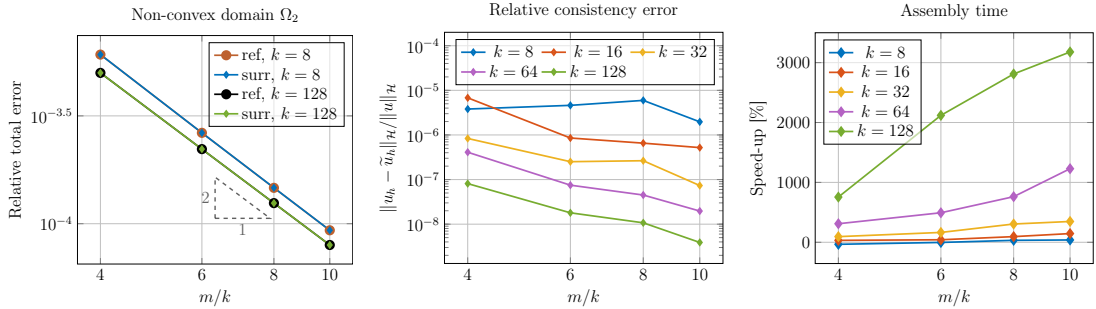
\begin{figure}\centering
\begin{scaletikzpicturetowidth}{0.32\textwidth}
\pgfplotstableread[col sep=comma,header=false]{./Results/Helmholtz_paper_study_nonconvex/Helmholtz_errors_8.csv}\helmholtzstudyxlabels
\begin{tikzpicture}[scale=\tikzscale,font=\large]
\begin{loglogaxis}[
xlabel={$m/k$},
ylabel={Relative total error},
xmajorgrids,
ymajorgrids,
title={Non-convex domain $\Omega_2$},
legend style={fill=white, fill opacity=0.6, draw opacity=1, text opacity=1},
legend pos=north east,
xtick=data,
xticklabels from table={\helmholtzstudyxlabels}{[index] 0},
legend cell align={left},
]
\addplot[color8, mark=*, very thick, mark options={scale=1.5}, fill opacity=1.0] table [x index = {0}, y index={1}, col sep=comma] {./Results/Helmholtz_paper_study_nonconvex/Helmholtz_errors_8.csv};
\addlegendentry{ref, $k = 8$};
\addplot[color1, mark=diamond*, very thick] table [x index = {0}, y index={2}, col sep=comma] {./Results/Helmholtz_paper_study_nonconvex/Helmholtz_errors_8.csv};
\addlegendentry{surr, $k = 8$};

\addplot[black, mark=*, very thick, mark options={scale=1.5}, fill opacity=1.0] table [x index = {0}, y index={1}, col sep=comma] {./Results/Helmholtz_paper_study_nonconvex/Helmholtz_errors_128.csv};
\addlegendentry{ref, $k = 128$};
\addplot[color5, mark=diamond*, very thick] table [x index = {0}, y index={2}, col sep=comma] {./Results/Helmholtz_paper_study_nonconvex/Helmholtz_errors_128.csv};
\addlegendentry{surr, $k = 128$};
\logLogSlopeTriangle{0.7}{0.2}{0.2}{1}{2}{};
\end{loglogaxis}
\end{tikzpicture}
\end{scaletikzpicturetowidth}
\hfill
\begin{scaletikzpicturetowidth}{0.32\textwidth}
\pgfplotstableread[col sep=comma,header=false]{./Results/Helmholtz_paper_study_nonconvex/Helmholtz_errors_8.csv}\helmholtzstudyxlabels
\begin{tikzpicture}[scale=\tikzscale,font=\large]
\begin{loglogaxis}[
xlabel={$m/k$},
ylabel={$\|u_h - \tilde{u}_h\|_\mathcal{H} / \|u\|_\mathcal{H}$},
xmajorgrids,
ymajorgrids,
title={Relative consistency error},
legend style={fill=white, fill opacity=0.6, draw opacity=1, text opacity=1},
legend pos=north east,
legend columns=3,
xtick=data,
xticklabels from table={\helmholtzstudyxlabels}{[index] 0},
ymax=2e-4
]
\addplot[color1, mark=diamond*, very thick] table [x index = {0}, y index={5}, col sep=comma] {./Results/Helmholtz_paper_study_nonconvex/Helmholtz_errors_8.csv};
\addlegendentry{$k = 8$};

\addplot[color2, mark=diamond*, very thick] table [x index = {0}, y index={5}, col sep=comma] {./Results/Helmholtz_paper_study_nonconvex/Helmholtz_errors_16.csv};
\addlegendentry{$k = 16$};

\addplot[color3, mark=diamond*, very thick] table [x index = {0}, y index={5}, col sep=comma] {./Results/Helmholtz_paper_study_nonconvex/Helmholtz_errors_32.csv};
\addlegendentry{$k = 32$};

\addplot[color4, mark=diamond*, very thick] table [x index = {0}, y index={5}, col sep=comma] {./Results/Helmholtz_paper_study_nonconvex/Helmholtz_errors_64.csv};
\addlegendentry{$k = 64$};

\addplot[color5, mark=diamond*, very thick] table [x index = {0}, y index={5}, col sep=comma] {./Results/Helmholtz_paper_study_nonconvex/Helmholtz_errors_128.csv};
\addlegendentry{$k = 128$};

\end{loglogaxis}
\end{tikzpicture}
\end{scaletikzpicturetowidth}
\hfill
\begin{scaletikzpicturetowidth}{0.32\textwidth}
\pgfplotstableread[col sep=comma,header=false]{./Results/Helmholtz_paper_study_nonconvex/Helmholtz_errors_8.csv}\helmholtzstudyxlabels
\begin{tikzpicture}[scale=\tikzscale,font=\large]
\begin{semilogxaxis}[
xlabel={$m/k$},
ylabel={Speed-up [\%]},
xmajorgrids,
ymajorgrids,
title={Assembly time},
legend style={fill=white, fill opacity=0.6, draw opacity=1, text opacity=1},
legend pos=north west,
mark repeat=1,
/pgf/number format/.cd,
use comma,
1000 sep={},
xtick=data,
xticklabels from table={\helmholtzstudyxlabels}{[index] 0}
]
\addplot[color1, mark=diamond*, very thick, mark options={scale=1.5}, fill opacity=1.0] table [x index = {1}, y index={7}, col sep=comma] {./Results/Helmholtz_paper_study_nonconvex/Helmholtz_speedup_8.csv};
\addlegendentry{$k = 8$};
\addplot[color2, mark=diamond*, very thick, mark options={scale=1.5}, fill opacity=1.0] table [x index = {1}, y index={7}, col sep=comma] {./Results/Helmholtz_paper_study_nonconvex/Helmholtz_speedup_16.csv};
\addlegendentry{$k = 16$};
\addplot[color3, mark=diamond*, very thick, mark options={scale=1.5}, fill opacity=1.0] table [x index = {1}, y index={7}, col sep=comma] {./Results/Helmholtz_paper_study_nonconvex/Helmholtz_speedup_32.csv};
\addlegendentry{$k = 32$};
\addplot[color4, mark=diamond*, very thick, mark options={scale=1.5}, fill opacity=1.0] table [x index = {1}, y index={7}, col sep=comma] {./Results/Helmholtz_paper_study_nonconvex/Helmholtz_speedup_64.csv};
\addlegendentry{$k = 64$};
\addplot[color5, mark=diamond*, very thick, mark options={scale=1.5}, fill opacity=1.0] table [x index = {1}, y index={7}, col sep=comma] {./Results/Helmholtz_paper_study_nonconvex/Helmholtz_speedup_128.csv};
\addlegendentry{$k = 128$};
\end{semilogxaxis}
\end{tikzpicture}
\end{scaletikzpicturetowidth}
\caption{\label{fig:HelmholtzError} Relative reference and surrogate errors in the $\mathcal{H}$-norm for different wave numbers $k$ computed on the non-convex domain $\Omega_2$ (left). Relative consistency errors (center) and speed-up of the assembly time of the same problem (right).
\changed{Recall here that $M = \max\Big\{1, \lfloor 0.5 \cdot m^{1-\frac{p+1}{q+1}}\rfloor\Big\}$.}
}
\end{figure}
\changed{\subsection{Helmholtz equation with non-constant wave number}
\label{sub:varyingwavenumber}
In this subsection, we consider the Helmholtz equation \cref{eq:Helmholtz} in which the wave number is non-constant over the physical domain.
This type of problem occurs, for instance, in the modeling of acoustic waves with heterogeneous wavespeed; see, e.g.,~\cite{chan2017weight} and the references therein.
Here, we use an example inspired from \cite{erlangga2006novel} on the wedge domain $(0,6) \times (0,10) \subset \R^2$ presented in the left of \cref{fig:varying_wave_number}.
The domain is discretized with three patches.
In the top and bottom patches, we utilize a constant wave number and in the central  patch, we use a spatially varying wave number.
This choice introduces a jump in the coefficient along the patch interfaces.
Since the matrix entries corresponding to the basis functions close to the patch boundaries are integrated by the standard approach, the surrogate method is not impeded by this discontinuity.
In particular, we choose the spatially varying wave number
\begin{align}
\label{eq:variablewavenumber}
k(x_1,x_2) = \begin{cases} 20 & \text{for}\: 0 \leq x_2 < \frac{x_1}{3} + 2, \\5 \sin{\left(2 \pi x_2 \right)} + 15 & \text{for}\: \frac{x_1}{3} + 2 \leq x_2 < 6 - \frac{x_1}{6}, \\30 & \text{for}\: 6 - \frac{x_1}{6} \leq x_2 \leq 10, \end{cases}
\end{align}
and we consider two settings.

In the first one, we consider a manufactured solution $u(x_1,x_2) = \sin{\left(20 \pi x_1 \right)} \sin{\left(20 \pi x_2 \right)}$
which we use to obtain the right-hand side $f$ and $g$ in \cref{eq:Helmholtz}.
In the center and right of \cref{fig:varying_wave_number}, we present the relative $H^1(\Omega)$ and $L^2(\Omega)$ errors for $p = 3$, $M = 5$, $q \in \{1,2,3,4,5\}$, and decreasing $h$.
We observe that the surrogate method is able to reproduce the solution of the standard approach.
Moreover, the surrogate solutions exhibit the same error convergence rates as the reference solution ($M=1$) for all choices of $q$.
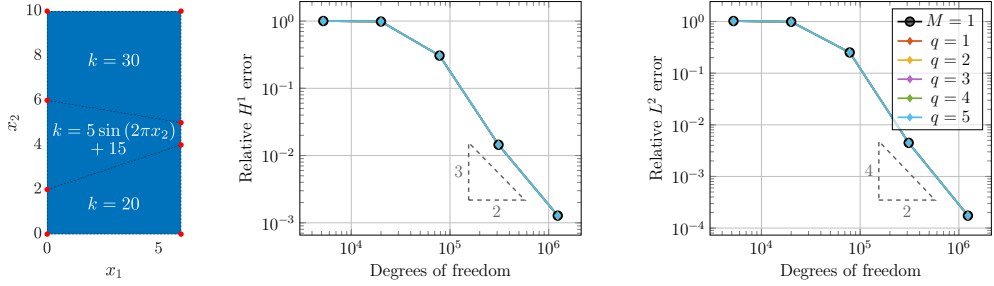
\begin{figure}\begin{minipage}{\textwidth}
\centering
\begin{minipage}{0.2\textwidth}
\includegraphics[height=10.5em]{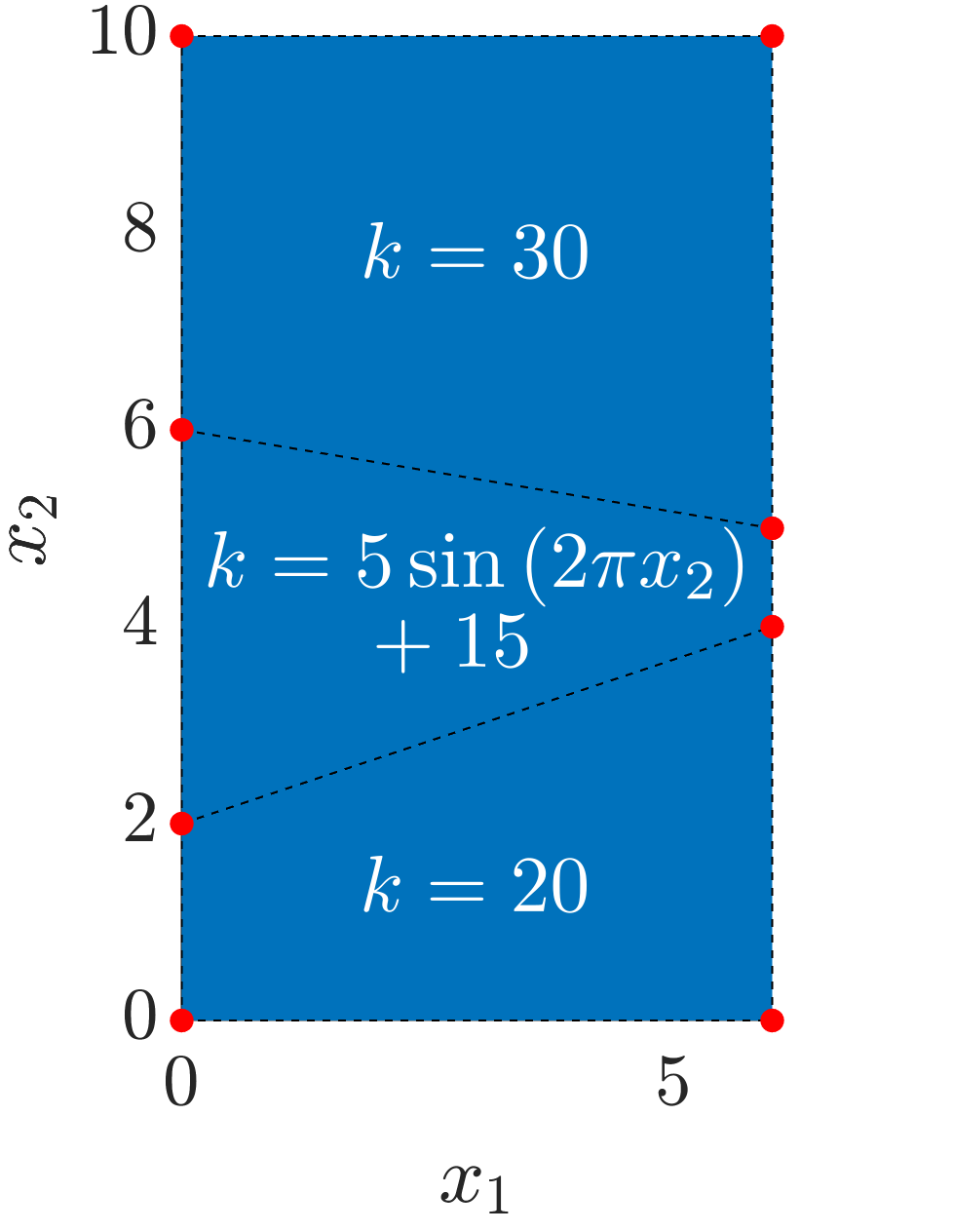}
\end{minipage}
\begin{minipage}{0.32\textwidth}
\begin{scaletikzpicturetowidth}{\textwidth}
\begin{tikzpicture}[scale=\tikzscale,font=\large]
\begin{loglogaxis}[
xlabel={Degrees of freedom},
ylabel={Relative $H^1$ error},
xmajorgrids,
ymajorgrids,
legend style={fill=white, fill opacity=0.6, draw opacity=1, text opacity=1},
legend pos=north east
]
\addplot[black, mark=*, very thick, mark options={scale=1.5}, fill opacity=0.6] table [x index = {0}, y index={1}, col sep=comma] {./Results/HelmholtzNonconstantWavenumber/h1errors_p3.csv};
\addlegendentry{$M=1$};
\addplot[color2, mark=diamond*, very thick] table [x index = {0}, y index={2}, col sep=comma] {./Results/HelmholtzNonconstantWavenumber/h1errors_p3.csv};
\addlegendentry{$q=1$};
\addplot[color3, mark=diamond*, very thick] table [x index = {0}, y index={3}, col sep=comma] {./Results/HelmholtzNonconstantWavenumber/h1errors_p3.csv};
\addlegendentry{$q=2$};
\addplot[color4, mark=diamond*, very thick] table [x index = {0}, y index={4}, col sep=comma] {./Results/HelmholtzNonconstantWavenumber/h1errors_p3.csv};
\addlegendentry{$q=3$};
\addplot[color5, mark=diamond*, very thick] table [x index = {0}, y index={5}, col sep=comma] {./Results/HelmholtzNonconstantWavenumber/h1errors_p3.csv};
\addlegendentry{$q=4$};
\addplot[color6, mark=diamond*, very thick] table [x index = {0}, y index={6}, col sep=comma] {./Results/HelmholtzNonconstantWavenumber/h1errors_p3.csv};
\addlegendentry{$q=5$};
\logLogSlopeTriangle{0.80}{0.2}{0.15}{2}{3}{};
\legend{}; \end{loglogaxis}
\end{tikzpicture}
\end{scaletikzpicturetowidth}
\end{minipage}
\qquad
\begin{minipage}{0.32\textwidth}
\begin{scaletikzpicturetowidth}{\textwidth}
\begin{tikzpicture}[scale=\tikzscale,font=\large]
\begin{loglogaxis}[
xlabel={Degrees of freedom},
ylabel={Relative $L^2$ error},
xmajorgrids,
ymajorgrids,
legend style={fill=white, fill opacity=0.6, draw opacity=1, text opacity=1},
legend pos=north east
]
\addplot[black, mark=*, very thick, mark options={scale=1.5}, fill opacity=0.6] table [x index = {0}, y index={1}, col sep=comma] {./Results/HelmholtzNonconstantWavenumber/l2errors_p3.csv};
\addlegendentry{$M=1$};
\addplot[color2, mark=diamond*, very thick] table [x index = {0}, y index={2}, col sep=comma] {./Results/HelmholtzNonconstantWavenumber/l2errors_p3.csv};
\addlegendentry{$q=1$};
\addplot[color3, mark=diamond*, very thick] table [x index = {0}, y index={3}, col sep=comma] {./Results/HelmholtzNonconstantWavenumber/l2errors_p3.csv};
\addlegendentry{$q=2$};
\addplot[color4, mark=diamond*, very thick] table [x index = {0}, y index={4}, col sep=comma] {./Results/HelmholtzNonconstantWavenumber/l2errors_p3.csv};
\addlegendentry{$q=3$};
\addplot[color5, mark=diamond*, very thick] table [x index = {0}, y index={5}, col sep=comma] {./Results/HelmholtzNonconstantWavenumber/l2errors_p3.csv};
\addlegendentry{$q=4$};
\addplot[color6, mark=diamond*, very thick] table [x index = {0}, y index={6}, col sep=comma] {./Results/HelmholtzNonconstantWavenumber/l2errors_p3.csv};
\addlegendentry{$q=5$};
\logLogSlopeTriangle{0.80}{0.2}{0.15}{2}{4}{};
\end{loglogaxis}
\end{tikzpicture}
\end{scaletikzpicturetowidth}
\end{minipage}
\end{minipage}
\caption{\label{fig:varying_wave_number} \changed{Wedge domain (left). Relative $H^1(\Omega)$ and $L^2(\Omega)$ errors for $p = 3$, $M = 5$, and the manufactured solution $u(x_1,x_2) = \sin{\left(20 \pi x_1 \right)} \sin{\left(20 \pi x_2 \right)}$ for the Helmholtz problem with non-constant wave number $k(x_1,x_2)$ from \cref{eq:variablewavenumber} (center and right).}}
\end{figure}

In the second setting, we provide only the source and boundary terms which do not stem from a manufactured solution.
This allows us to compare the discrete solutions only.
We choose $f(\bfx) = \frac{1}{\pi\,a} \exp({-\|\bfx - \bfc\|^2a^{-2}})$ with $a = 5\cdot 10^{-3}$, $\bfc = (3, 9.5)^\T$, $g = 0$, and $k(x_1,x_2)$ from \cref{eq:variablewavenumber}.
For the discretization parameters, we choose $p=3$, $q \in \{1,2,3\}$, and $M = 5$.
In the left of \cref{fig:varying_wave_number_nonanalytic}, we present the real part of the solution obtained with the standard approach ($M = 1$).
The solutions obtained with the surrogate method do not differ visually.
We emphasize this fact by showing the relative differences of the real part of the standard solution and the real part of each of the surrogate solutions in the right of \cref{fig:varying_wave_number_nonanalytic}.
We observe that the differences decrease with increasing $q$ and that the largest differences are located within the patch in which the wave number varies as well as in the left part of the bottom patch.
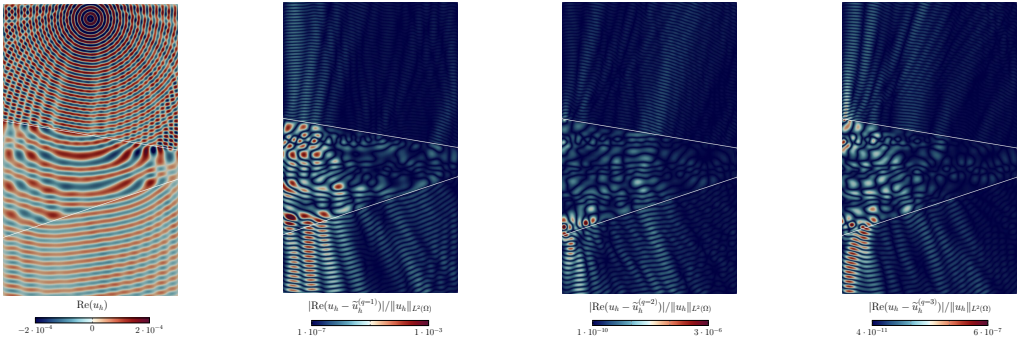
\begin{figure}\begin{minipage}{\textwidth}
\centering
\begin{minipage}{0.23\textwidth}
\begin{center}
\includegraphics[height=11em]{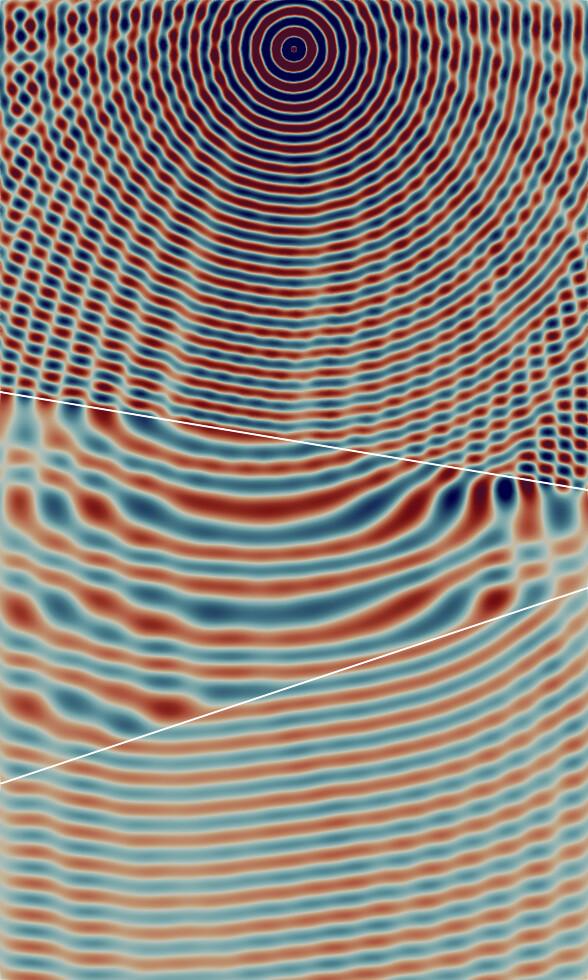}
\begin{scaletikzpicturetowidth}{0.6\textwidth}
\begin{tikzpicture}[scale=\tikzscale,font=\large]
\pgfmathsetlengthmacro\MajorTickLength{
  \pgfkeysvalueof{/pgfplots/major tick length} * 0.5
}
\begin{axis}[
title={\Large \vspace*{2em}$\mathrm{Re}(u_h)$},
xmin=-2e-4, xmax=2e-4,
ymin=0, ymax=0.02,
axis on top,
scaled x ticks=false,
scaled y ticks=false,
ytick=\empty,
yticklabels=\empty,
yticklabel pos=right,
extra x tick style={
    font=\large,
    tick style=transparent,     yticklabel pos=left,
    y tick label style={
        /pgf/number format/.cd,
            std,
            precision=3,
      /tikz/.cd
    }
},
width=7cm,
height=1.82cm,
major tick length=\MajorTickLength,
max space between ticks=1000pt,
try min ticks=3,
]
\addplot graphics [
includegraphics cmd=\pgfimage,
xmin=\pgfkeysvalueof{/pgfplots/xmin}, 
xmax=\pgfkeysvalueof{/pgfplots/xmax}, 
ymin=\pgfkeysvalueof{/pgfplots/ymin}, 
ymax=\pgfkeysvalueof{/pgfplots/ymax}
] {./Figures/cool_to_warm_extended_rot.png};
\end{axis}
\end{tikzpicture}
\end{scaletikzpicturetowidth}
\end{center}
\end{minipage}
\hfill
\begin{minipage}{0.23\textwidth}
\begin{center}
\includegraphics[height=11em]{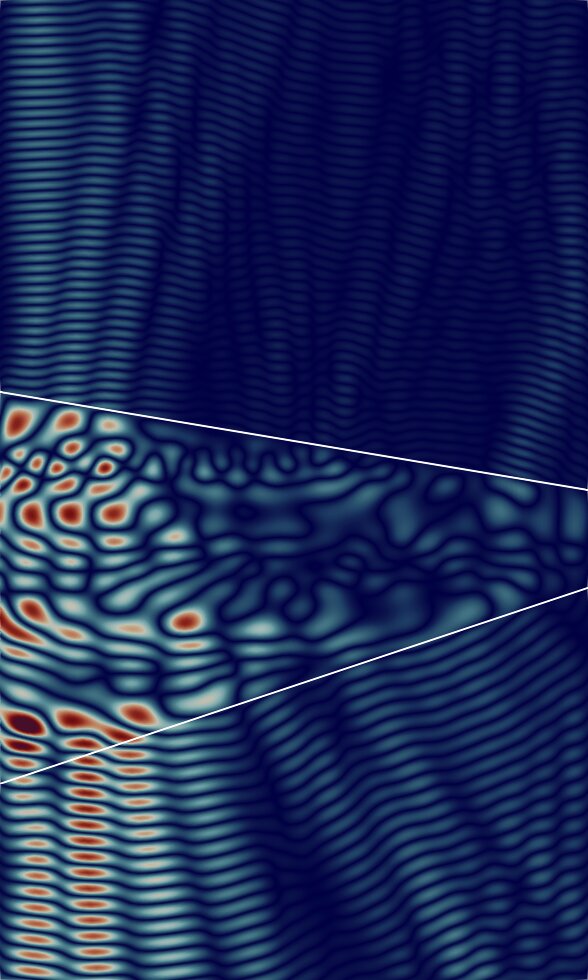}
\begin{scaletikzpicturetowidth}{0.6\textwidth}
\begin{tikzpicture}[scale=\tikzscale,font=\large]
\pgfmathsetlengthmacro\MajorTickLength{
  \pgfkeysvalueof{/pgfplots/major tick length} * 0.5
}
\begin{axis}[
title={\Large $|\mathrm{Re}(u_h - \widetilde{u}_h^{(q=1)})|   / \|u_h\|_{L^2(\Omega)}$},
xmin=1.0e-7, xmax=1.0e-3,
ymin=0, ymax=0.02,
axis on top,
scaled x ticks=false,
scaled y ticks=false,
ytick=\empty,
yticklabels=\empty,
yticklabel pos=right,
xticklabels={},
 extra x ticks={
   \pgfkeysvalueof{/pgfplots/xmin},
   \pgfkeysvalueof{/pgfplots/xmax}
 },
extra x tick style={
    font=\large,
    tick style=transparent,     yticklabel pos=left,
    y tick label style={
        /pgf/number format/.cd,
            std,
            precision=1,
      /tikz/.cd
    }
},
width=7cm,
height=1.82cm,
major tick length=\MajorTickLength,
max space between ticks=1000pt,
try min ticks=3,
]
\addplot graphics [
includegraphics cmd=\pgfimage,
xmin=\pgfkeysvalueof{/pgfplots/xmin}, 
xmax=\pgfkeysvalueof{/pgfplots/xmax}, 
ymin=\pgfkeysvalueof{/pgfplots/ymin}, 
ymax=\pgfkeysvalueof{/pgfplots/ymax}
] {./Figures/cool_to_warm_extended_rot.png};
\end{axis}
\end{tikzpicture}
\end{scaletikzpicturetowidth}
\end{center}
\end{minipage}
\hfill
\begin{minipage}{0.23\textwidth}
\begin{center}
\includegraphics[height=11em]{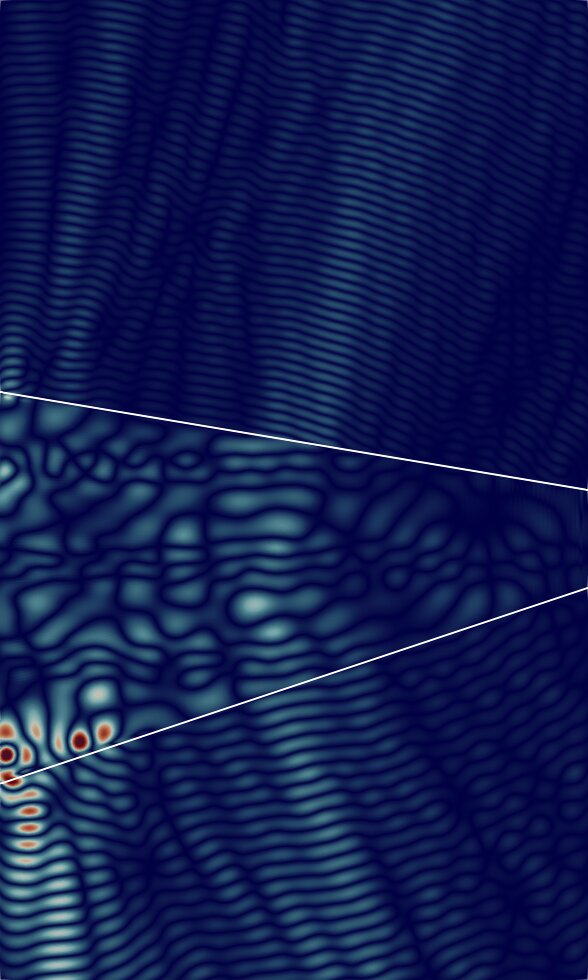}
\begin{scaletikzpicturetowidth}{0.6\textwidth}
\begin{tikzpicture}[scale=\tikzscale,font=\large]
\pgfmathsetlengthmacro\MajorTickLength{
  \pgfkeysvalueof{/pgfplots/major tick length} * 0.5
}
\begin{axis}[
title={\Large $|\mathrm{Re}(u_h - \widetilde{u}_h^{(q=2)})|   / \|u_h\|_{L^2(\Omega)}$},
xmin=1.0e-10, xmax=3e-6,
ymin=0, ymax=0.02,
axis on top,
scaled x ticks=false,
scaled y ticks=false,
ytick=\empty,
yticklabels=\empty,
yticklabel pos=right,
xticklabels={},
 extra x ticks={
    \pgfkeysvalueof{/pgfplots/xmin},
    \pgfkeysvalueof{/pgfplots/xmax}
 },
extra x tick style={
    font=\large,
    tick style=transparent,     yticklabel pos=left,
    y tick label style={
        /pgf/number format/.cd,
            std,
            precision=1,
      /tikz/.cd
    }
},
width=7cm,
height=1.82cm,
major tick length=\MajorTickLength,
max space between ticks=1000pt,
try min ticks=3,
]
\addplot graphics [
includegraphics cmd=\pgfimage,
xmin=\pgfkeysvalueof{/pgfplots/xmin}, 
xmax=\pgfkeysvalueof{/pgfplots/xmax}, 
ymin=\pgfkeysvalueof{/pgfplots/ymin}, 
ymax=\pgfkeysvalueof{/pgfplots/ymax}
] {./Figures/cool_to_warm_extended_rot.png};
\end{axis}
\end{tikzpicture}
\end{scaletikzpicturetowidth}
\end{center}
\end{minipage}
\hfill
\begin{minipage}{0.23\textwidth}
\begin{center}
\includegraphics[height=11em]{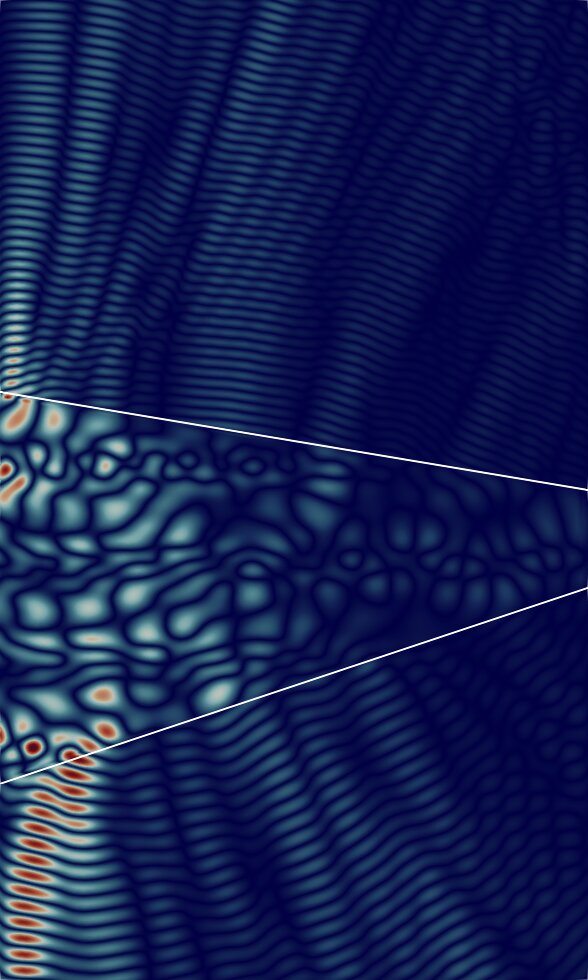}
\begin{scaletikzpicturetowidth}{0.6\textwidth}
\begin{tikzpicture}[scale=\tikzscale,font=\large]
\pgfmathsetlengthmacro\MajorTickLength{
  \pgfkeysvalueof{/pgfplots/major tick length} * 0.5
}
\begin{axis}[
title={\Large $|\mathrm{Re}(u_h - \widetilde{u}_h^{(q=3)})|   / \|u_h\|_{L^2(\Omega)}$},
xmin=4.0e-11, xmax=6e-7,
ymin=0, ymax=0.02,
axis on top,
scaled x ticks=false,
scaled y ticks=false,
ytick=\empty,
yticklabels=\empty,
yticklabel pos=right,
xticklabels={},
 extra x ticks={
    \pgfkeysvalueof{/pgfplots/xmin},
    \pgfkeysvalueof{/pgfplots/xmax}
 },
extra x tick style={
    font=\large,
    tick style=transparent,     yticklabel pos=left,
    y tick label style={
        /pgf/number format/.cd,
            std,
            precision=1,
      /tikz/.cd
    }
},
width=7cm,
height=1.82cm,
major tick length=\MajorTickLength,
max space between ticks=1000pt,
try min ticks=3,
]
\addplot graphics [
includegraphics cmd=\pgfimage,
xmin=\pgfkeysvalueof{/pgfplots/xmin}, 
xmax=\pgfkeysvalueof{/pgfplots/xmax}, 
ymin=\pgfkeysvalueof{/pgfplots/ymin}, 
ymax=\pgfkeysvalueof{/pgfplots/ymax}
] {./Figures/cool_to_warm_extended_rot.png};
\end{axis}
\end{tikzpicture}
\end{scaletikzpicturetowidth}
\end{center}
\end{minipage}
\end{minipage}
\caption{\label{fig:varying_wave_number_nonanalytic} \changed{Real part of the solution obtained with the standard approach and $p = 3$ for the Helmholtz problem with non-constant wave number $k(x_1,x_2)$ from \cref{eq:variablewavenumber} and non-manufactured solution (left). Relative difference of the real part of the standard solution and the real part of each of the surrogate solutions with increasing $q \in \{1,2,3\}$ (second from left to right).}}
\end{figure}
}
\subsection{Linear elastodynamics with periodic pressure loading} \label{sub:elastodynamics_with_periodic_pressure_loading}

In this set of experiments, we focus on the time-harmonic linear elastodynamics problem \cref{eq:TimeHarmonicSystem} with the energy density functional $W$ in \cref{eq:LinearElastEnergy}.
We start with \cref{eq:GeneralSystem} and apply a pressure that fluctuates periodically, with angular frequency $\omega$, in the interior of the circular hole.
The setup is depicted in \cref{fig:PlateWithHoleFrequencyBCs}.
Let $\bsigma = \partial_{\bmu} W(\bmu)$.
On the circular boundary, we apply a time-dependent pressure of the form $-\bsigma\bmn\cdot\bmn = p(t) = \frac{p_0}{\sqrt{2\,\pi}} \mathrm{e}^{-\iu\omega t}$.
The problem is transformed into the frequency domain resulting in the time harmonic equations \cref{eq:TimeHarmonicSystem}.
In this particular experiment, we choose $\omega = 50\,\pi$, $p_0 = 1$, $r_0 = 1$, $L = 4$, $\lambda = 2$, $\mu = 1$, and $\rho_0 = 1$.
The analytical solution to this problem may be written in polar coordinates $(r,\theta)$ as follows \cite{kausel2006fundamental}:
\begin{equation}
u_r(r) = - \frac{p_0 r_{0} \zeta H^{(2)}_{1}\left(\frac{\omega r}{\gamma}\right)}{\mu \left(\gamma \omega r_{0} H^{(2)}_{0}\left(\omega r_{0}\right) - 2 \zeta H^{(2)}_{1}\left(\omega r_{0}\right)\right)}
\,.
\end{equation}
Here, $\mcH_a^{(2)}$ is the Hankel function of the second kind, $\gamma = \sqrt{\frac{\lambda + 2 \mu}{\rho_{0}}}$, and $\zeta = \sqrt{\frac{\mu}{\rho_{0}}}$.

We investigate PML absorbing boundary conditions.
For this problem, the stiffness matrix $\sfK$ and mass matrix $\sfM$ in \cref{eq:TimeHarmonicSystem} need to be assembled.
For the surrogate matrices, this is achieved by employing definitions similar to \cref{eq:DefinitionOfSurrogateMatrixSymmetric} and \cref{eq:DefinitionOfSurrogateMatrixSymmetricKernel} but for the vector-valued setting mentioned in \cref{remark:vectorvalued}.
\begin{figure}
\centering
\begin{minipage}{0.27\textwidth}
\includegraphics[width=\textwidth]{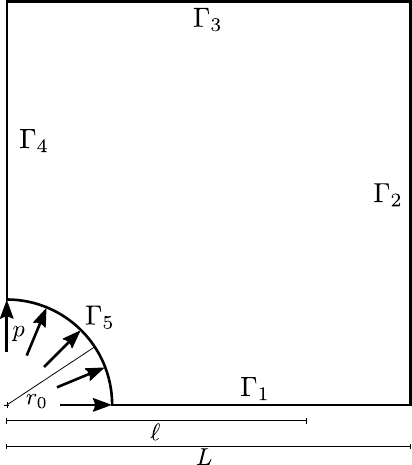}
\end{minipage}
\caption{\label{fig:PlateWithHoleFrequencyBCs} Plate with circular hole setup in linear elastodynamics problem.}
\end{figure}
We prescribe symmetric boundary conditions on $\Gamma_1$ and $\Gamma_4$ and the pressure is applied to $\Gamma_5$.
On the remaining boundaries, $\Gamma_2$ and $\Gamma_3$, homogeneous Dirichlet boundary conditions are prescribed.
The region of interest and the PML region are separated through $\ell$.
Each physical coordinate $x_k$ for $k=1,2$ is mapped according to the following specific stretching function from \cite{michler2007improving,astaneh2018perfectly}:
\begin{align}
\label{eq:StretchingFunction}
\tilde{x}_k = \begin{cases} 
x_k & \text{if } 0 < x_k \leq \ell, \\
x_k + \iu \frac{C}{\omega}\left(\frac{x_k - \ell}{L - \ell}\right)^n & \text{if } \ell < x_k \leq L,
\end{cases}
\end{align}
where $\ell = 3$, $n = 2$, and $C = 5$.

As with the second set of experiments with the Helmholtz equation, we also choose to adopt a mesh-dependent sampling parameter $M = M(h)$ which balances of the error and performance in our favor.
Here, we consider $M(h) = \max\Big\{2,\floor[\big]{2 \cdot h^{\frac{p-q+1/2}{q+1}}}\Big\}$, where it is implicitly understood that $q>p$.

For $p=2$ and $q=5$, the real part of the surrogate solution and difference from the standard solution are presented in \cref{fig:PlateWithHolePML}.
Note that the largest difference is observed in the PML region.
This can be attributed to the fact that the stencil functions are also affected by the stretching function \cref{eq:StretchingFunction}.
Inspecting the right-hand side of \cref{fig:PlateWithHolePML}, it is clear that the difference in the two solutions in the domain of interest, $\Omega \cap (0,\ell)^2$, is an order of magnitude less than the difference in the two solutions in the PML.

Relative $L^2$ errors in $\Omega \cap (0,\ell)^2$ and the associated assembly times are shown in \cref{fig:PlateWithHolePMLErrorTime}.
For all $h$ and $q$ considered, $\|\bmu-\bmu_h\|_{L^2(\Omega)}$ is indistinguishable from $\|\bmu-\tilde{\bmu}_h\|_{L^2(\Omega)}$.
We do not observe asymptotic error convergence, even in the standard IGA case, because of the presence of the PML.
For $q=5$ on the finest mesh, we observe a speed-up of $1679\%$, without any degradation in the $L^2$ error.
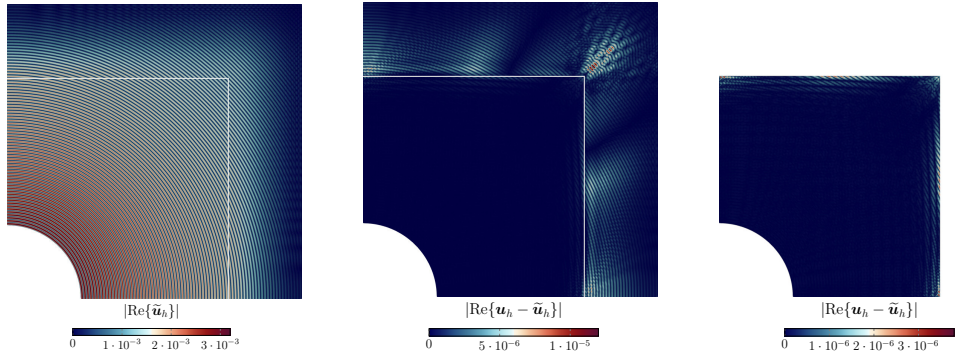
\begin{figure}
\centering
\qquad
\begin{minipage}{0.27\textwidth}
\begin{center}
\includegraphics[width=\textwidth]{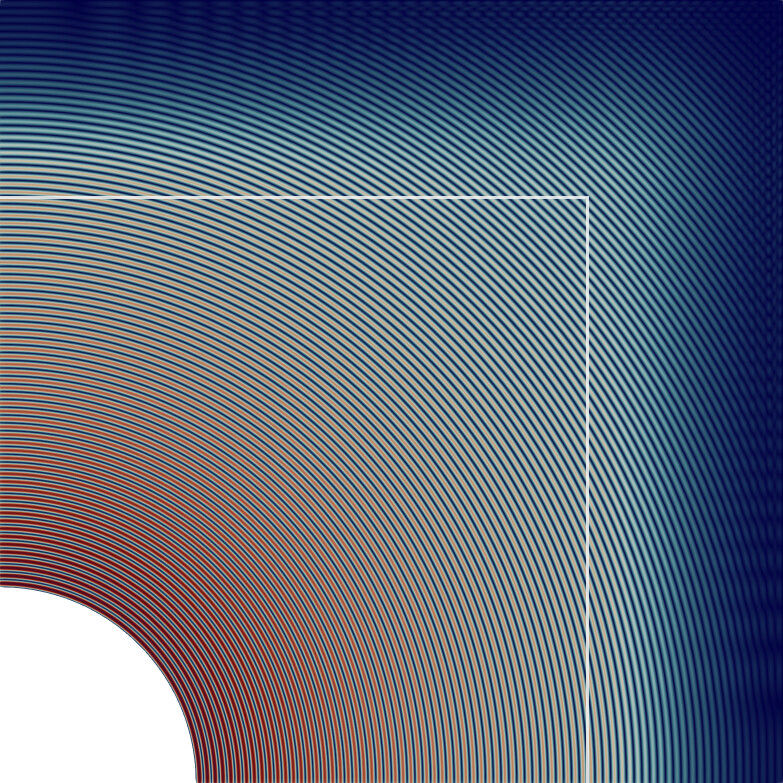}
\begin{scaletikzpicturetowidth}{0.6\textwidth}
\begin{tikzpicture}[scale=\tikzscale,font=\large]
\pgfmathsetlengthmacro\MajorTickLength{
  \pgfkeysvalueof{/pgfplots/major tick length} * 0.5
}
\begin{axis}[
title={\Large $|\mathrm{Re}\{\tilde{\bmu}_h\}|$},
xmin=0, xmax=0.0032,
ymin=0, ymax=0.02,
axis on top,
scaled x ticks=false,
scaled y ticks=false,
ytick=\empty,
yticklabels=\empty,
yticklabel pos=right,
x tick label style={
  /pgf/number format/.cd,
            sci zerofill,
            precision=0,
  /tikz/.cd  
},
extra x tick style={
    font=\large,
    tick style=transparent,     yticklabel pos=left,
    x tick label style={
        /pgf/number format/.cd,
            std,
            precision=3,
      /tikz/.cd
    }
},
width=7cm,
height=1.82cm,
major tick length=\MajorTickLength,
max space between ticks=1000pt,
try min ticks=4,
]
\addplot graphics [
includegraphics cmd=\pgfimage,
xmin=\pgfkeysvalueof{/pgfplots/xmin}, 
xmax=\pgfkeysvalueof{/pgfplots/xmax}, 
ymin=\pgfkeysvalueof{/pgfplots/ymin}, 
ymax=\pgfkeysvalueof{/pgfplots/ymax}
] {./Figures/cool_to_warm_extended_rot.png};
\end{axis}
\end{tikzpicture}
\end{scaletikzpicturetowidth}
\end{center}
\end{minipage}
\qquad
\begin{minipage}{0.27\textwidth}
\begin{center}
\includegraphics[width=\textwidth]{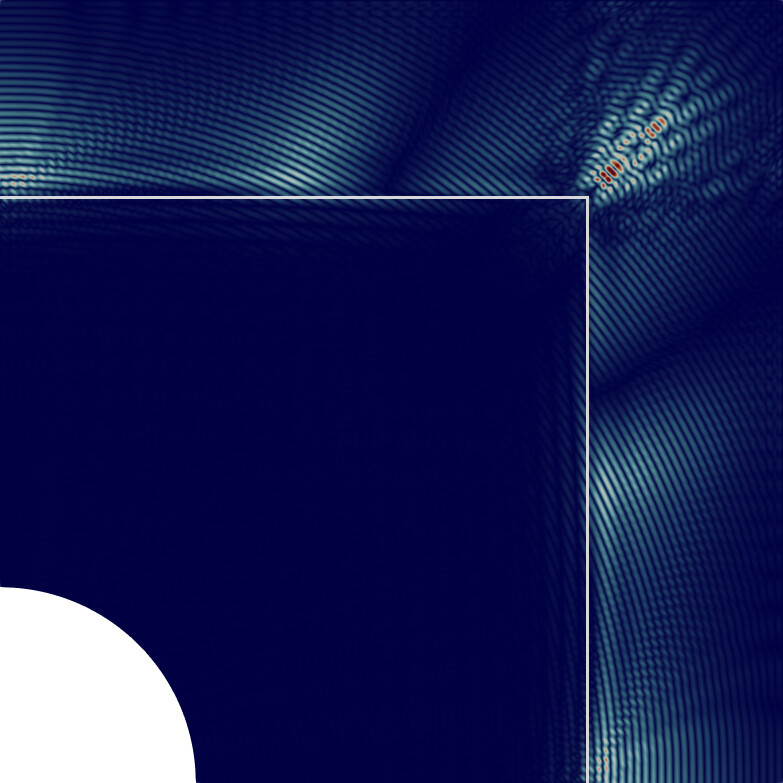}
\begin{scaletikzpicturetowidth}{0.6\textwidth}
\begin{tikzpicture}[scale=\tikzscale,font=\large]
\pgfmathsetlengthmacro\MajorTickLength{
  \pgfkeysvalueof{/pgfplots/major tick length} * 0.5
}
\begin{axis}[
title={\Large $|\mathrm{Re}\{\bmu_h - \tilde{\bmu}_h\}|$},
xmin=0, xmax=1.2e-5,
ymin=0, ymax=0.02,
axis on top,
scaled x ticks=false,
scaled y ticks=false,
ytick=\empty,
yticklabels=\empty,
yticklabel pos=right,
x tick label style={
  /pgf/number format/.cd,
            sci zerofill,
            precision=0,
  /tikz/.cd  
},
extra x tick style={
    font=\large,
    tick style=transparent,     yticklabel pos=left,
    x tick label style={
        /pgf/number format/.cd,
            std,
            precision=3,
      /tikz/.cd
    }
},
width=7cm,
height=1.82cm,
major tick length=\MajorTickLength,
max space between ticks=1000pt,
try min ticks=4,
]
\addplot graphics [
includegraphics cmd=\pgfimage,
xmin=\pgfkeysvalueof{/pgfplots/xmin}, 
xmax=\pgfkeysvalueof{/pgfplots/xmax}, 
ymin=\pgfkeysvalueof{/pgfplots/ymin}, 
ymax=\pgfkeysvalueof{/pgfplots/ymax}
] {./Figures/cool_to_warm_extended_rot.png};
\end{axis}
\end{tikzpicture}
\end{scaletikzpicturetowidth}
\end{center}
\end{minipage}
\qquad
\begin{minipage}{0.27\textwidth}
\begin{center}
\includegraphics[width=\textwidth]{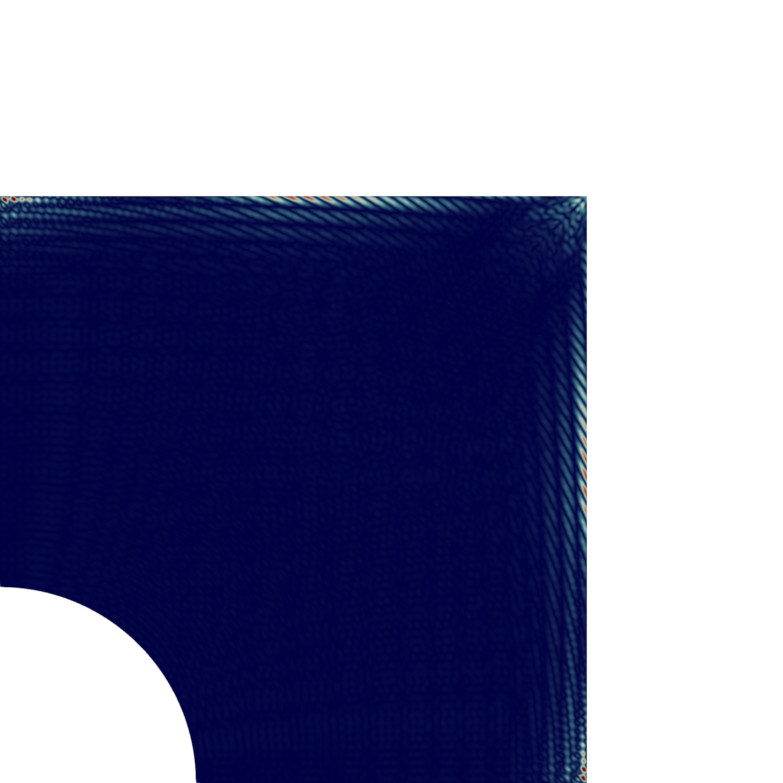}
\begin{scaletikzpicturetowidth}{0.6\textwidth}
\begin{tikzpicture}[scale=\tikzscale,font=\large]
\pgfmathsetlengthmacro\MajorTickLength{
  \pgfkeysvalueof{/pgfplots/major tick length} * 0.5
}
\begin{axis}[
title={\Large $|\mathrm{Re}\{\bmu_h - \tilde{\bmu}_h\}|$},
xmin=0, xmax=3.8e-6,
ymin=0, ymax=0.02,
axis on top,
scaled x ticks=false,
scaled y ticks=false,
ytick=\empty,
yticklabels=\empty,
yticklabel pos=right,
x tick label style={
  /pgf/number format/.cd,
            sci zerofill,
            precision=0,
  /tikz/.cd  
},
extra x tick style={
    font=\large,
    tick style=transparent,     yticklabel pos=left,
    x tick label style={
        /pgf/number format/.cd,
            std,
            precision=3,
      /tikz/.cd
    }
},
width=7cm,
height=1.82cm,
major tick length=\MajorTickLength,
max space between ticks=1000pt,
try min ticks=4,
]
\addplot graphics [
includegraphics cmd=\pgfimage,
xmin=\pgfkeysvalueof{/pgfplots/xmin}, 
xmax=\pgfkeysvalueof{/pgfplots/xmax}, 
ymin=\pgfkeysvalueof{/pgfplots/ymin}, 
ymax=\pgfkeysvalueof{/pgfplots/ymax}
] {./Figures/cool_to_warm_extended_rot.png};
\end{axis}
\end{tikzpicture}
\end{scaletikzpicturetowidth}
\end{center}
\end{minipage}
\caption{\label{fig:PlateWithHolePML}Magnitude of the real part of the solution on the finest mesh with $\omega = 50\,\pi$ and PML boundary conditions. Surrogate IGA solution with $q=5$ and mesh-dependent sampling parameter $M$ (left). Difference between standard and surrogate IGA solutions (center and right).}
\end{figure}
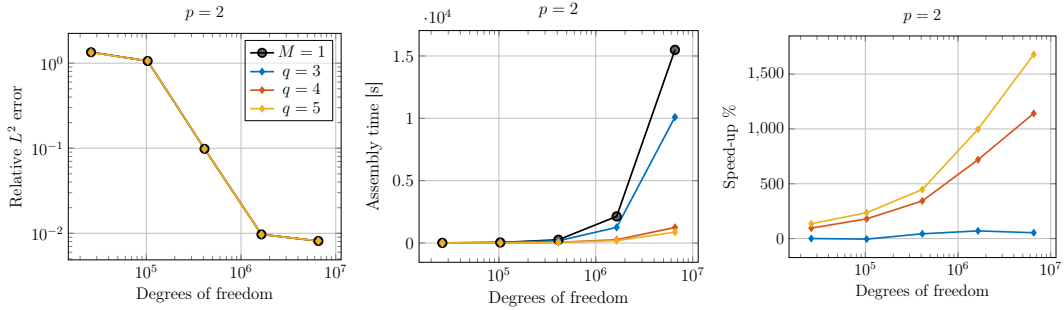
\begin{figure}
\centering
\begin{minipage}{0.32\textwidth}
\begin{scaletikzpicturetowidth}{\textwidth}
\begin{tikzpicture}[scale=\tikzscale,font=\large]
\begin{loglogaxis}[
xlabel={Degrees of freedom},
ylabel={Relative $L^2$ error},
xmajorgrids,
ymajorgrids,
title={$p=2$},
legend style={fill=white, fill opacity=0.6, draw opacity=1, text opacity=1},
legend pos=north east
]
\addplot[black, mark=*, very thick, mark options={scale=1.5}, fill opacity=0.6] table [x index = {0}, y index={1}, col sep=comma] {./Results/Frequency_Elastodynamics_PML/l2errors_real_p2.csv};
\addlegendentry{$M=1$};
\addplot[color1, mark=diamond*, very thick] table [x index = {0}, y index={4}, col sep=comma] {./Results/Frequency_Elastodynamics_PML/l2errors_real_p2.csv};
\addlegendentry{$q=3$};
\addplot[color2, mark=diamond*, very thick] table [x index = {0}, y index={5}, col sep=comma] {./Results/Frequency_Elastodynamics_PML/l2errors_real_p2.csv};
\addlegendentry{$q=4$};
\addplot[color3, mark=diamond*, very thick] table [x index = {0}, y index={6}, col sep=comma] {./Results/Frequency_Elastodynamics_PML/l2errors_real_p2.csv};
\addlegendentry{$q=5$};
\end{loglogaxis}
\end{tikzpicture}
\end{scaletikzpicturetowidth}
\end{minipage}
\begin{minipage}{0.32\textwidth}
\begin{scaletikzpicturetowidth}{\textwidth}
\begin{tikzpicture}[scale=\tikzscale,font=\large]
\begin{semilogxaxis}[
xlabel={Degrees of freedom},
ylabel={Assembly time [s]},
xmajorgrids,
ymajorgrids,
title={$p=2$},
legend style={fill=white, fill opacity=0.6, draw opacity=1, text opacity=1},
legend pos=north east
]
\addplot[black, mark=*, very thick, mark options={scale=1.5}, fill opacity=0.6] table [x index = {0}, y index={1}, col sep=comma] {./Results/Frequency_Elastodynamics_PML/time_assembly_p2.csv};
\addlegendentry{$M=1$};
\addplot[color1, mark=diamond*, very thick] table [x index = {0}, y index={4}, col sep=comma] {./Results/Frequency_Elastodynamics_PML/time_assembly_p2.csv};
\addlegendentry{$q=3$};
\addplot[color2, mark=diamond*, very thick] table [x index = {0}, y index={5}, col sep=comma] {./Results/Frequency_Elastodynamics_PML/time_assembly_p2.csv};
\addlegendentry{$q=4$};
\addplot[color3, mark=diamond*, very thick] table [x index = {0}, y index={6}, col sep=comma] {./Results/Frequency_Elastodynamics_PML/time_assembly_p2.csv};
\addlegendentry{$q=5$};
\legend{}; \end{semilogxaxis}
\end{tikzpicture}
\end{scaletikzpicturetowidth}
\end{minipage}
\begin{minipage}{0.32\textwidth}
\begin{scaletikzpicturetowidth}{\textwidth}
\begin{tikzpicture}[scale=\tikzscale,font=\large]
\begin{semilogxaxis}[
xlabel={Degrees of freedom},
ylabel={Speed-up \%},
xmajorgrids,
ymajorgrids,
title={$p=2$},
legend style={fill=white, fill opacity=0.6, draw opacity=1, text opacity=1},
legend pos=north east
]
\addplot[color1, mark=diamond*, very thick] table [x index = {0}, y index={3}, col sep=comma] {./Results/Frequency_Elastodynamics_PML/speedup_p2.csv};
\addlegendentry{$q=3$};
\addplot[color2, mark=diamond*, very thick] table [x index = {0}, y index={4}, col sep=comma] {./Results/Frequency_Elastodynamics_PML/speedup_p2.csv};
\addlegendentry{$q=4$};
\addplot[color3, mark=diamond*, very thick] table [x index = {0}, y index={5}, col sep=comma] {./Results/Frequency_Elastodynamics_PML/speedup_p2.csv};
\addlegendentry{$q=5$};
\legend{}; \end{semilogxaxis}
\end{tikzpicture}
\end{scaletikzpicturetowidth}
\end{minipage}
\caption{\label{fig:PlateWithHolePMLErrorTime}\changed{Relative $L^2$ errors in $\Omega \cap (0,\ell)^2$ and assembly times for periodic pressure loading with $\omega = 50\pi$ and $p = 2$ computed with PML absorbing boundary conditions.}}
\end{figure}

We refrain from showing non-harmonic linear elastodynamic examples.
They are only of little relevance because the stiffness matrix $\sfK$ needs to be computed only once in the first time step and can be reused for each subsequent step.
Nevertheless, the surrogate approach may be of interest when applied in a matrix-free setting since $\sfK$ would need to be recomputed for each matrix-vector product.

\subsection{Nonlinear hyperelastic waves} \label{sec:nonlinear_elasticity_results}
In this final set of experiments, we consider transient, nonlinear, hyperelastic wave propagation obeying \cref{eq:GeneralSystem} with the energy density functional \eqref{eq:NeoHookeanEnergy}.
The problem setup is illustrated in \cref{fig:nonlinear_wave_setup}.
As domain, we choose the annulus $\Omega = \{\bfx \in \R^2 : 1 < \|\bfx\| < 2\}$ and the time interval $[0,T]$, with $T = 7.5$.
We employ the material parameters $\rho_0 = 1$, $E = 1$, and $\nu = 0.35$.
The corresponding Lam\'e parameters are obtained via the expressions $\mu = \frac{E}{2(1+\nu)}$ and $\lambda = \frac{\nu E}{(1+\nu)(1-2\nu)}$.
We prescribe zero initial displacements and zero initial velocity; i.e.,~$\bfu_0 = \bfv_0 = \bm{0}$.
At the boundary $\partial \Omega$, we apply a pointwise force pulse at the right side, i.e.,~$\partial_{\bmu} W(\bmu) \bmn = \delta(x_1 - 2)\cdot f(t) \cdot (-1,0)^\T$, with
\begin{align}
\label{eq:ForcePulse}
f(t) = \begin{cases} \frac{1}{10} \sin\left(\frac{\pi \, t}{t_f}\right) & t \leq t_f, \\
0 & \text{otherwise},
\end{cases}
\end{align}
where $t_f = \nicefrac{1}{5}$.
The computational domain is split into four patches, as depicted on the right hand side of \cref{fig:nonlinear_wave_setup}, each of which is discretized with $m = 100$ and $p=2$.
The sub-matrices belonging to each patch are assembled in parallel.

For the time-discretization, we employ the nonlinear generalized-$\alpha$ method described in \cite{cottrell2009isogeometric} with the damping parameters $\rho_\infty = \frac{1}{2}$, $\alpha_m = \frac{1}{2} \frac{3 - \rho_\infty}{1 + \rho_\infty}$, and $\alpha_f = \frac{1}{1 + \rho_\infty}$.
Moreover, we choose $\Delta t = 5 \cdot 10^{-3}$ as the time step size.
At each time step, we perform two Newton iterations.
In each iteration the nonlinear matrix is being reassembled using either the standard approach or the surrogate matrix approach, with $q = 5$ and $M = 18$.
Since the mass matrix term does not change over time, we assemble it once using the standard approach and re-use it in the subsequent steps.

In order to compare the solutions of the standard and surrogate method, we record the displacement in $y$-direction over time at two positions, $\bfx_1 = (0,1)^\T$ and $\bfx_2 = (0,2)^\T$, in \cref{fig:nonlinear_wave_displacements}.
No visual difference can be observed.
On the left-hand side of \cref{fig:nonlinear_wave_energy}, we present the kinetic, internal, and total energy divided by two versus the time $t \geq 0.3$ for both approaches.
We observe, once again, no visual difference in the two solutions.
The central plot in \cref{fig:nonlinear_wave_energy} shows the time required to complete each time step.
For each time step, this \changed{required time} includes the right-hand side evaluation as well as the reassembly and inversion of the tangent matrices for each Newton iteration.
\changed{We use the MATLAB backslash operator to invert the emerging systems which takes on average \SI{6.41}{\second} per time step independent of which assembly method is used.
In total, inverting the tangent matrices in the standard approach takes up approximately 10.7\% of the time and approximately 26.6\%, if the surrogate method is used.
For the sake of completeness, we present the accumulated total time and the accumulated time required for the inversion of the systems in the right-hand side of \cref{fig:nonlinear_wave_energy}.
In this scenario, a speed-up of about 142\% may be observed.}
Finally, in \cref{fig:nonlinear_wave_stress}, we illustrate the von Mises stress at different times.
The faster traveling body waves reach the point $\bfx_2$ first, followed by the surface waves, which result in the greatest displacements; cf.~\cref{fig:nonlinear_wave_displacements}.

\begin{figure}
\centering
\includegraphics[width=0.4\textwidth]{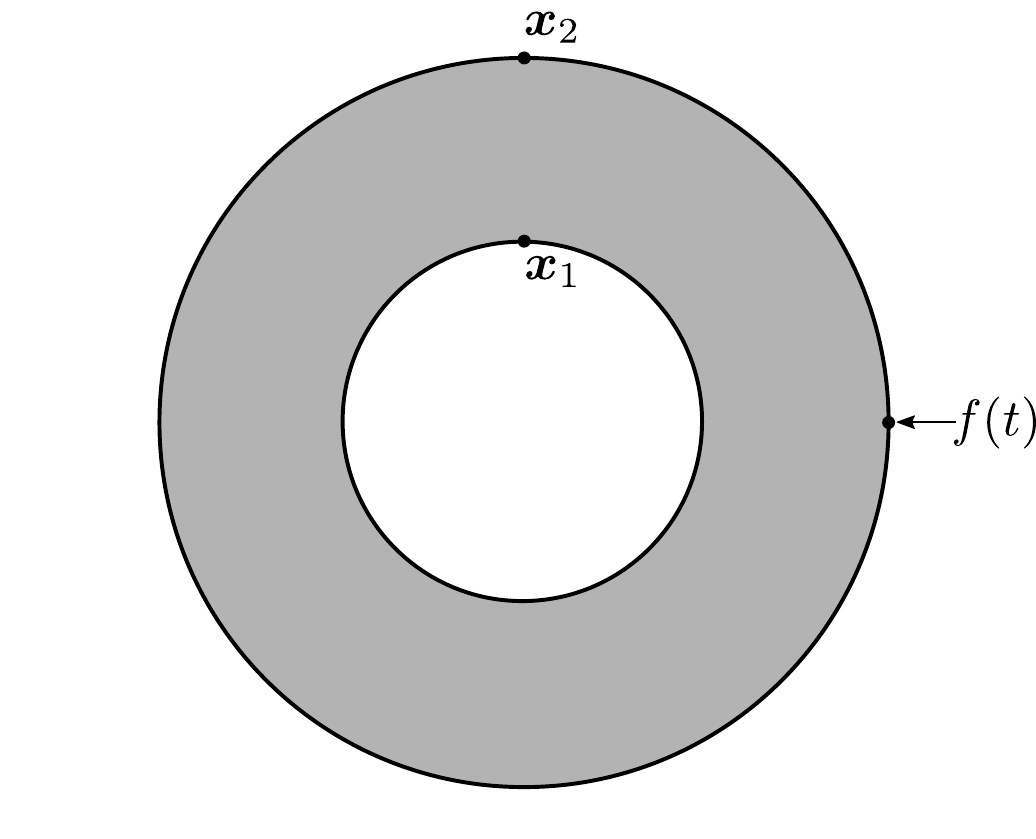}
\includegraphics[width=0.4\textwidth]{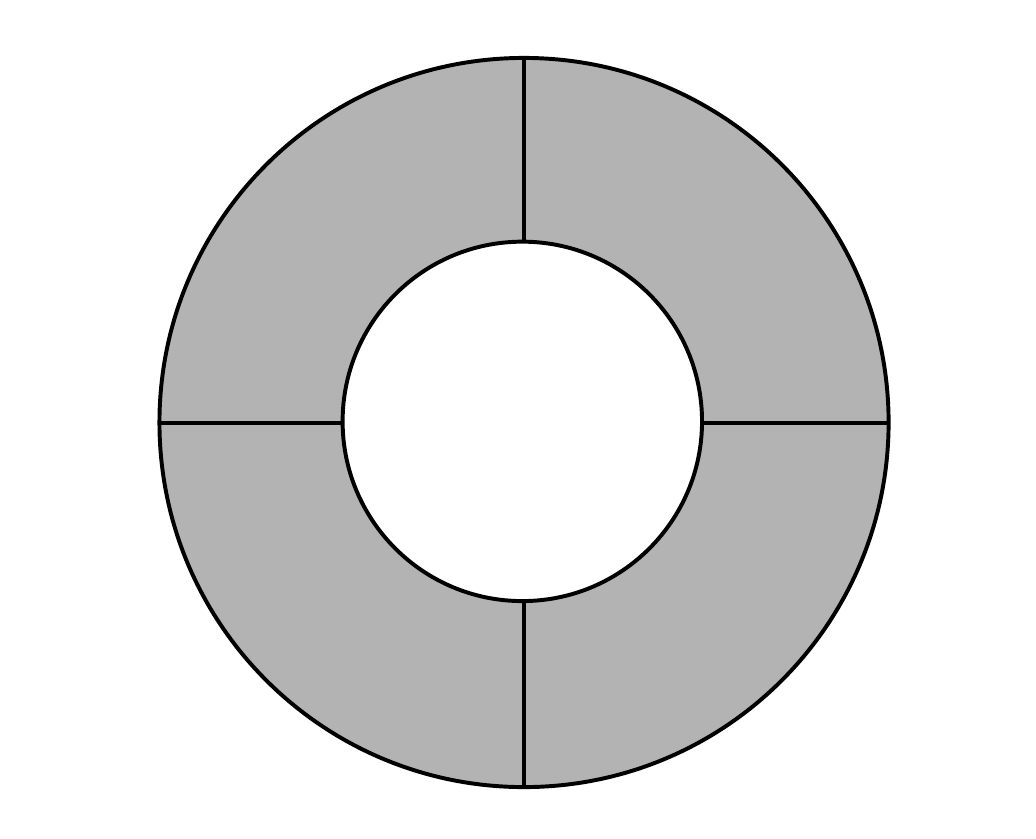}
\caption{\label{fig:nonlinear_wave_setup}Problem setup (left) and patches (right).}
\end{figure}
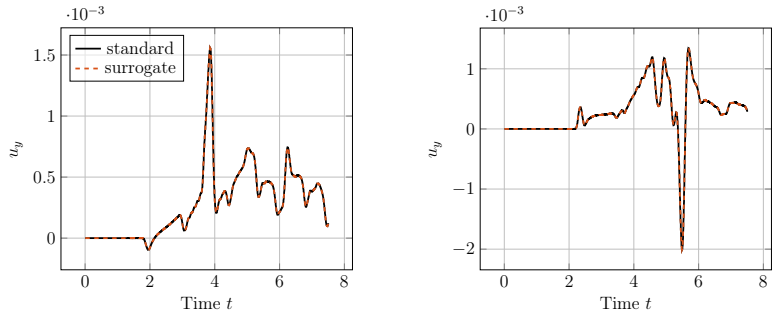
\begin{figure}
\centering
\begin{scaletikzpicturetowidth}{0.32\textwidth}
\begin{tikzpicture}[scale=\tikzscale,font=\large]\centering
\begin{axis}[
xlabel={Time $t$},
xmajorgrids,
ylabel={$u_y$},
ymajorgrids,
mark repeat=100,
legend style={fill=white, fill opacity=0.6, draw opacity=1, text opacity=1},
legend pos=north west
]
\addplot[black,very thick] table [x="Time", y="avg(u (1))", col sep=comma] {Results/displacements_time.csv};
\addlegendentry{standard};
\addplot[color2,very thick,dashed] table [x="Time", y="avg(u_surr (1))", col sep=comma] {Results/displacements_time.csv};
\addlegendentry{surrogate};
\end{axis}
\end{tikzpicture}
\end{scaletikzpicturetowidth}
\qquad
\begin{scaletikzpicturetowidth}{0.32\textwidth}
\begin{tikzpicture}[scale=\tikzscale,font=\large]\centering
\begin{axis}[
xlabel={Time $t$},
xmajorgrids,
ylabel={$u_y$},
ymajorgrids,
mark repeat=50,
legend style={fill=white, fill opacity=0.6, draw opacity=1, text opacity=1},
legend pos=north west
]
\addplot[black,very thick] table [x="Time", y="avg(u_input_1 (1))", col sep=comma] {Results/displacements_time.csv};
\addlegendentry{standard};
\addplot[color2,very thick,dashed] table [x="Time", y="avg(u_surr_input_1 (1))", col sep=comma] {Results/displacements_time.csv};
\addlegendentry{surrogate};
\legend{};
\end{axis}
\end{tikzpicture}
\end{scaletikzpicturetowidth}
\caption{\label{fig:nonlinear_wave_displacements}Displacements in $y$-direction $u_y$ recorded at $\bfx_1 = (0,1)^\T$ (left) and $\bfx_2 = (0,2)^\T$ (right) for the standard and surrogate method.}
\end{figure}
\begin{figure}
\centering
\begin{minipage}{0.54\textwidth}
\begin{scaletikzpicturetowidth}{\textwidth}
\begin{tikzpicture}[scale=\tikzscale,font=\large]
\begin{axis}[
xlabel={Time $t$},
ylabel={Energy},
xmajorgrids,
ymajorgrids,
title={Energy over time},
legend style={fill=white, fill opacity=0.6, draw opacity=1, text opacity=1},
legend pos=south east,
legend style={at={(0.5,0.75)},anchor=center,font=\small},
yticklabel style={
	/pgf/number format/fixed,
	/pgf/number format/precision=1
},
xmin=0.3,
xmax=7.5,
ymin=1.6e-5,
ymax=2.0e-5,
width=12cm,
height=7cm
]
\addplot[color1, very thick, restrict x to domain=0.3:7.5] table [x index = {0}, y expr=\thisrowno{4}/2, col sep=comma ] {./Results/Transient_Nonlinear_Elastodynamics_MP_fullannulus_wave_par/energy.csv};
\addlegendentry{$\nicefrac{1}{2}$ Total (standard)};

\addplot[color3, very thick, restrict x to domain=0.3:7.5] table [x index = {0}, y index={2}, col sep=comma] {./Results/Transient_Nonlinear_Elastodynamics_MP_fullannulus_wave_par/energy.csv};
\addlegendentry{Internal (standard)};

\addplot[color4, dashed, very thick, restrict x to domain=0.3:7.5] table [x index = {0}, y index={2}, col sep=comma] {./Results/Transient_Nonlinear_Elastodynamics_MP_fullannulus_wave_par_surr/energy.csv};
\addlegendentry{Internal (surrogate)};

\addplot[color5, very thick, restrict x to domain=0.3:7.5] table [x index = {0}, y index={1}, col sep=comma] {./Results/Transient_Nonlinear_Elastodynamics_MP_fullannulus_wave_par/energy.csv};
\addlegendentry{Kinetic (standard)};

\addplot[color7, dashed, very thick, restrict x to domain=0.3:7.5] table [x index = {0}, y index={1}, col sep=comma] {./Results/Transient_Nonlinear_Elastodynamics_MP_fullannulus_wave_par_surr/energy.csv};
\addlegendentry{Kinetic (surrogate)};

\end{axis}
\end{tikzpicture}
\end{scaletikzpicturetowidth}
\end{minipage}\hfill
\begin{minipage}{0.24\textwidth}
\begin{scaletikzpicturetowidth}{\textwidth}
\begin{tikzpicture}[scale=\tikzscale,font=\large]
\begin{axis}[
xlabel={\changed{Time step}},
ylabel={\changed{Required time} [s]},
xmajorgrids,
ymajorgrids,
title={\changed{Required time per time step}},
legend style={fill=white, fill opacity=0.6, draw opacity=1, text opacity=1,font=\small},
legend pos=north east,
height=7.6cm,
width=5.33cm,
xmin=0,
xmax=1500
]
\addplot[black, very thick] table [x index = {0}, y index={3}, col sep=comma] {./Results/Transient_Nonlinear_Elastodynamics_MP_fullannulus_wave_par/time.csv};
\addlegendentry{Standard};

\addplot[color2, very thick] table [x index = {0}, y index={3}, col sep=comma] {./Results/Transient_Nonlinear_Elastodynamics_MP_fullannulus_wave_par_surr/time.csv};
\addlegendentry{Surrogate};

\end{axis}
\end{tikzpicture}
\end{scaletikzpicturetowidth}
\end{minipage}\begin{minipage}{0.21\textwidth}
\begin{scaletikzpicturetowidth}{\textwidth}
\pgfplotsset{compat=1.11,
    /pgfplots/ybar legend/.style={
    /pgfplots/legend image code/.code={       \draw[##1,/tikz/.cd,yshift=-0.25em]
        (0cm,0cm) rectangle (3pt,0.8em);},
   },
}
\changed{\begin{tikzpicture}[scale=\tikzscale,font=\large]
\begin{axis}[ymin=0,
height=8.0cm,
x=5.0cm,
enlarge x limits={abs=0.72cm},
bar width=0.3cm,
ybar,
xmajorticks=false,
nodes near coords,
xmin=0.0,xmax=0.5,
ylabel={Accumulated time [s]},
ymax=110000,
legend style={at={(0.5,-0.01)}, anchor=north,legend columns=2,font=\small}]
\addplot+[black] coordinates {(0.1, 90497)};
\addplot+[color5] coordinates {(0.2, 9638)};
\addplot+[color2] coordinates {(0.3, 37455)};
\addplot+[color1] coordinates {(0.4, 9581)};
\legend{std (total), std (solve), surr (total), surr (solve)}
\end{axis}
\end{tikzpicture}}
\end{scaletikzpicturetowidth}
\end{minipage}
\caption{\label{fig:nonlinear_wave_energy} Total divided by two, kinetic, and internal energy over the time interval $[0.3, T]$ (left). \changed{Required time} per time step (center). \changed{Accumulated total time and accumulated time spent on the inversion of the tangent matrices} (right).}
\end{figure}
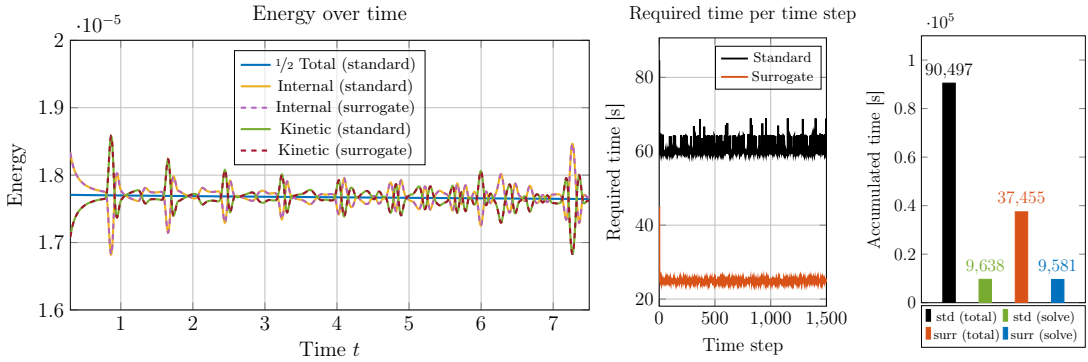

\begin{figure}
\begin{center}
\includegraphics[width=0.3\linewidth]{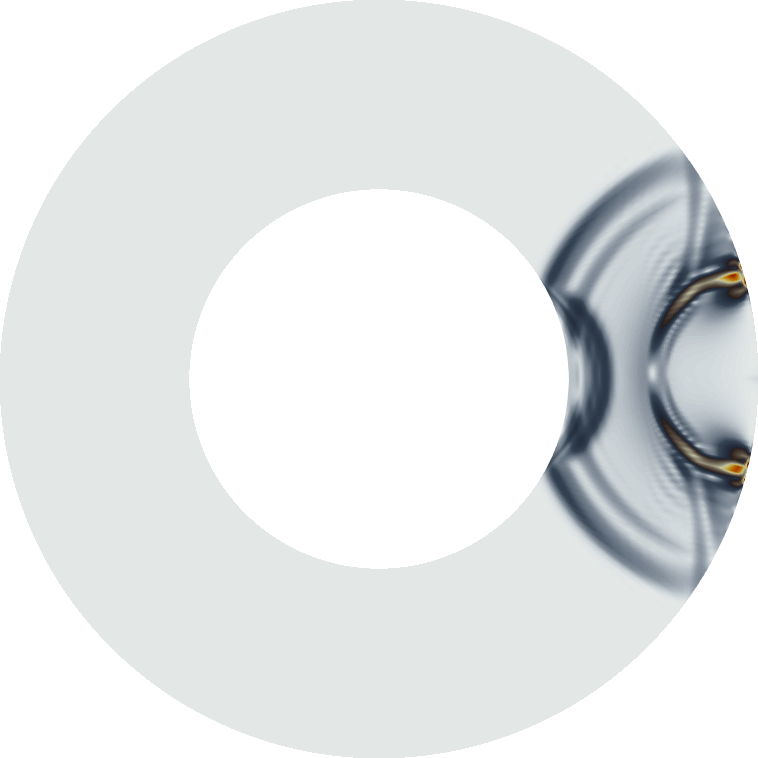}
\hfill
\includegraphics[width=0.3\linewidth]{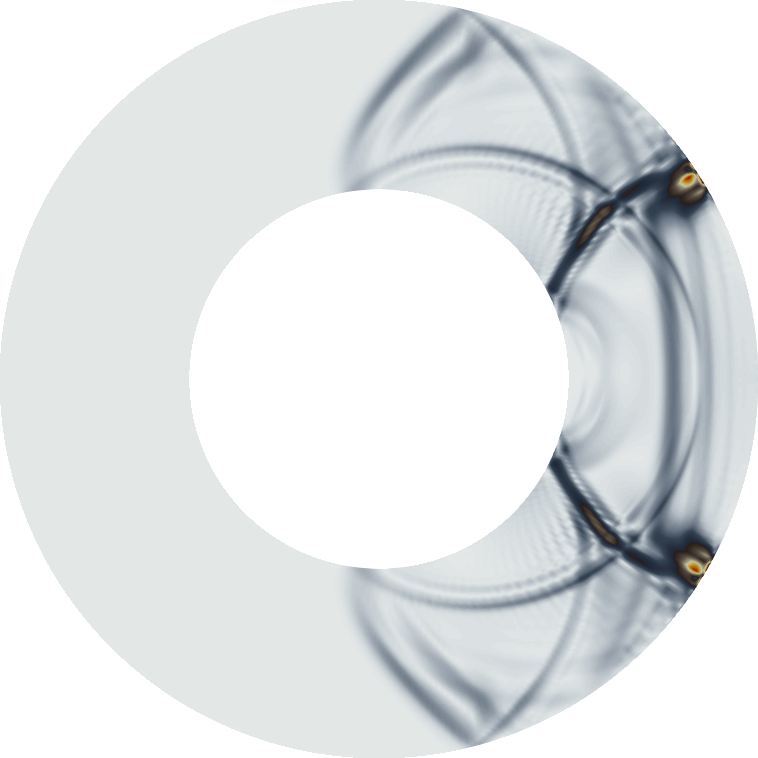}
\hfill
\includegraphics[width=0.3\linewidth]{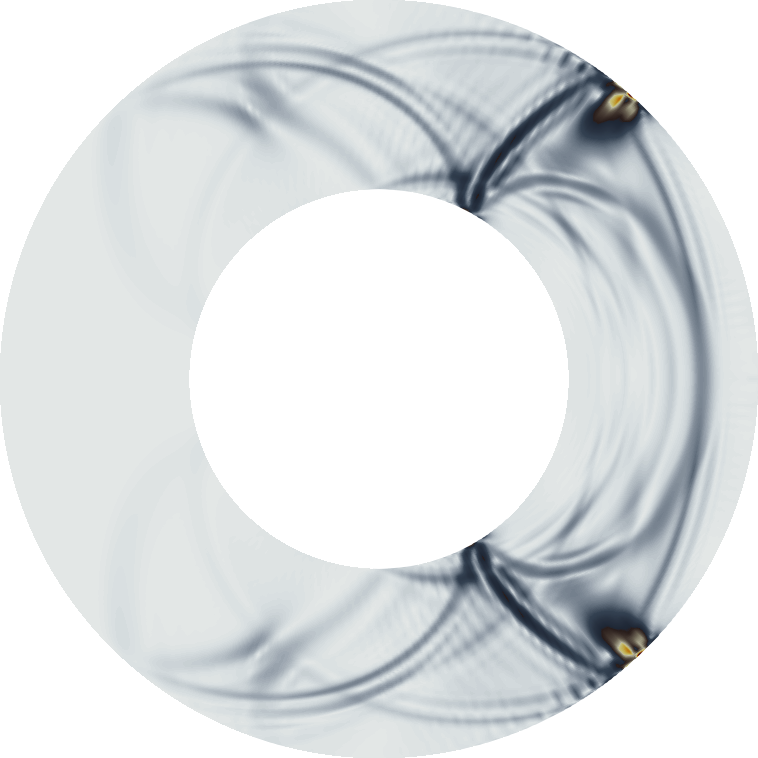}\\[1em]
\includegraphics[width=0.3\linewidth]{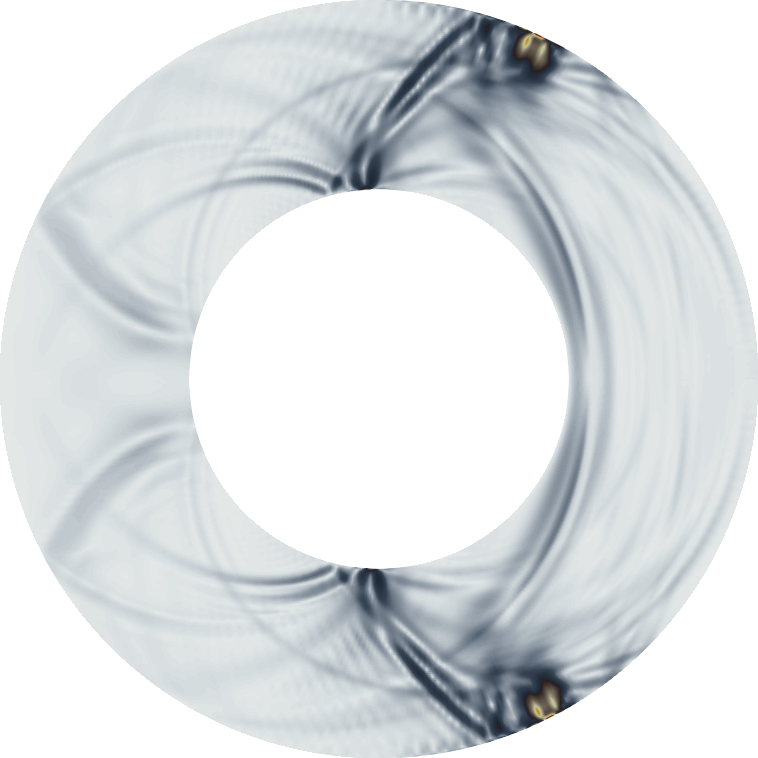}
\hfill
\includegraphics[width=0.3\linewidth]{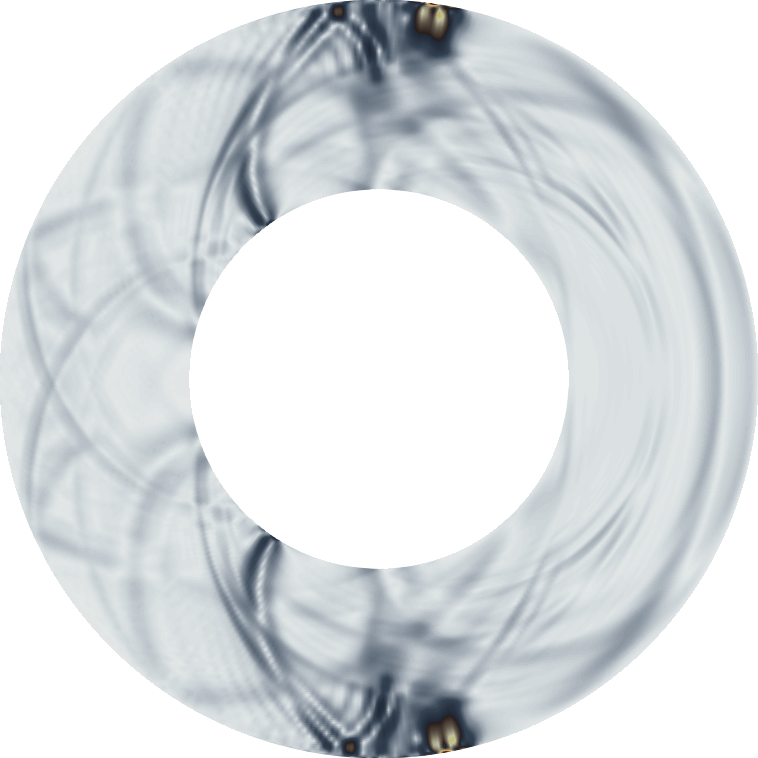}
\hfill
\includegraphics[width=0.3\linewidth]{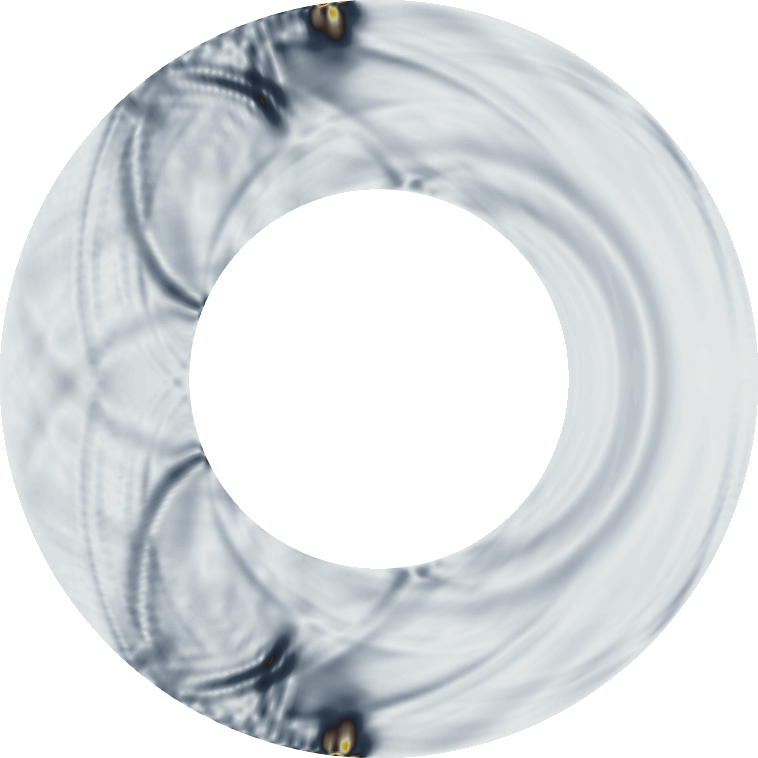}\\[1em]
\begin{scaletikzpicturetowidth}{0.3\textwidth}
\begin{tikzpicture}[scale=\tikzscale,font=\large]
\pgfmathsetlengthmacro\MajorTickLength{
  \pgfkeysvalueof{/pgfplots/major tick length} * 0.5
}
\begin{axis}[
title={\Large $\sigma_{v,M}$},
xmin=0, xmax=3e-2,
ymin=0, ymax=0.02,
axis on top,
scaled x ticks=false,
scaled y ticks=false,
ytick=\empty,
yticklabels=\empty,
yticklabel pos=right,
x tick label style={
  /pgf/number format/.cd,
            sci zerofill,
            precision=0,
  /tikz/.cd  
},
extra x tick style={
    font=\large,
    tick style=transparent,     yticklabel pos=left,
    x tick label style={
        /pgf/number format/.cd,
            std,
            precision=3,
      /tikz/.cd
    }
},
width=7cm,
height=1.82cm,
major tick length=\MajorTickLength,
max space between ticks=1000pt,
try min ticks=4,
]
\addplot graphics [
includegraphics cmd=\pgfimage,
xmin=\pgfkeysvalueof{/pgfplots/xmin}, 
xmax=\pgfkeysvalueof{/pgfplots/xmax}, 
ymin=\pgfkeysvalueof{/pgfplots/ymin}, 
ymax=\pgfkeysvalueof{/pgfplots/ymax}
] {./Figures/colormap_customized_yellow_gray_blue_rot.png};
\end{axis}
\end{tikzpicture}
\end{scaletikzpicturetowidth}
\end{center}
\caption{\label{fig:nonlinear_wave_stress}Von Mises stress $\sigma_{v,M}$ in the nonlinear hyperelastic wave problem at times $t \in \{1, 2, 3, 4, 5, 6\}$.}
\end{figure}
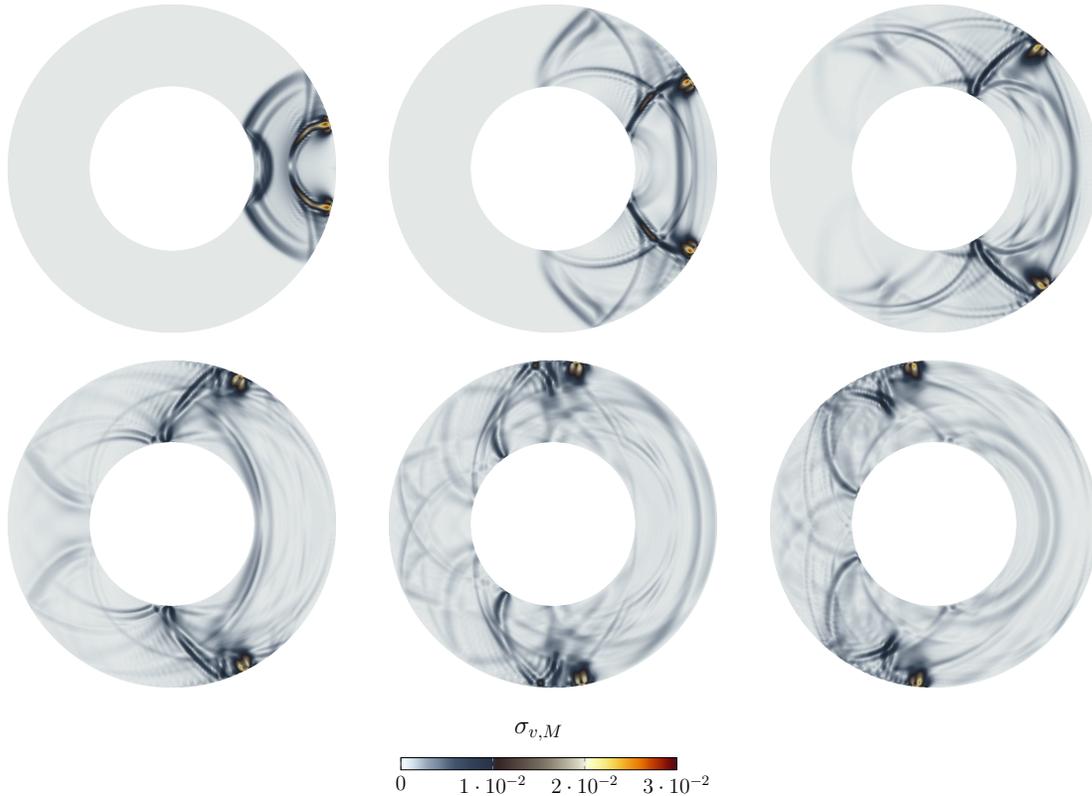

\section{Conclusion} \label{sec:conclusion}
In this work, we applied the surrogate matrix methodology to several problems emerging in the investigation of waves in the isogeometric setting.
We performed an a priori error analysis for the Helmholtz equation which demonstrated that the additional consistency error introduced by the presence of surrogate matrices is independent of the wave number.
Moreover, we presented a floating point analysis showing that the computational complexity of the methodology compares favorably to other state-of-the-art assembly techniques for isogeometric analysis.

We confirmed the theoretical error estimates for the Helmholtz equation by performing benchmark computations showing the correct convergence behavior.
We furthermore showed that the methodology is beneficial when applied to wave problems with PML absorbing boundary conditions by considering a linear problem in elastodynamics.
Finally, we applied the methodology to a transient, nonlinear, hyperelastic wave propagation problem with a material modeled by a compressible neo-Hookean material.
This last example showed the efficacy of the methodology for implicit time stepping schemes in which a nonlinear problem is solved by Newton's method in each time step.
Our numerical experiments demonstrate clear performance gains in all experiments and we observed speed-ups of up to $3178$\%, when compared to the reference assembly algorithm, without losing any significant accuracy.

Thus far, in order to address the feasibility of this methodology in isogeometric analysis, we have relied on the MATLAB software GeoPDEs \cite{de2011geopdes}.
This software greatly lends itself to rapid prototyping but is not actually suitable for high performance experiments and it only supports element-loop assembly.
Element-loop assembly is not necessary for the surrogate matrix methodology.
Row/column-loop assembly would only provide better performance.
Nevertheless, because our implementations employ element-loop assembly, which is presently the dominant assembly strategy employed in IGA software, it allows our experiments to directly suggest how surrogate matrices would perform in many other codes.

\section*{Acknowledgments}
The authors would like to thank Ren\'e Hiemstra for his insightful conversations about our work. 
This project has received funding from the European Union's Horizon 2020 research and innovation programme under grant agreement No 800898.
This work was also partly supported by the German Research Foundation through
the Priority Programme 1648 ``Software for Exascale Computing'' (SPPEXA) and by grant WO671/11-1.
\FloatBarrier
\phantomsection\bibliographystyle{abbrv}
\bibliography{main}

\end{document}